\documentclass[11pt]{amsart}
\usepackage{geometry}

\usepackage{amssymb}
\usepackage{amsmath}
\usepackage{amsfonts}
\usepackage{amsthm}
\usepackage{newtxtext,newtxmath}
\usepackage{bbm}
\usepackage[colorlinks=true, citecolor=blue, urlcolor=blue]{hyperref}
\usepackage{cite}
\usepackage{graphicx}
\usepackage{ytableau}
\usepackage[figurewithin=none]{caption}
\usepackage{float}
\usepackage{mathtools}
\usepackage{relsize}
\usepackage{fancyvrb}
\usepackage{xfrac}
\usepackage{mathrsfs}
\DeclareFontFamily{U}{rsfs}{\skewchar\font127 }
\DeclareFontShape{U}{rsfs}{m}{n}{%
   <-6> rsfs5
   <6-8> rsfs7
   <8-> rsfs10
}{}

\usepackage{tikz, pgfplots}
\usetikzlibrary{decorations.markings}
\usetikzlibrary{decorations.pathreplacing}
\usetikzlibrary{arrows,shapes,positioning}
\tikzstyle directed=[postaction={decorate,decoration={markings, mark=at position #1 with {\arrow[scale=1.25]{>}}}}]
\tikzstyle rdirected=[postaction={decorate,decoration={markings, mark=at position #1 with {\arrow[scale=1.25]{<}}}}]

\pgfsetlayers{nodelayer,edgelayer}
\pgfdeclarelayer{nodelayer}
\pgfdeclarelayer{edgelayer}
\usepackage[all]{xy}
\allowdisplaybreaks

\usepackage{multicol}
\usepackage{blindtext}

\newtheoremstyle{style}
{10pt} 
{10pt} 
{\itshape} 
{0pt} 
{\bfseries} 
{.} 
{10pt} 
{} 
\theoremstyle{style}
\newtheorem{theorem}{Theorem}[section]
\newtheorem{lemma}[theorem]{Lemma}
\newtheorem{question}[theorem]{Question}

\newtheorem{definition}[theorem]{Definition}
\newtheorem{remark}[theorem]{Remark}
\newtheorem{example}[theorem]{Example}

\numberwithin{equation}{section}
\numberwithin{figure}{section}
\numberwithin{equation}{section}
\numberwithin{figure}{section}

\newcommand{\bigslant}[2]{{\raisebox{.3em}{$#1$}\left/\raisebox{-.3em}{$#2$}\right.}}

\tikzset{mynode/.style={inner sep=3pt,fill,outer sep=0,circle =2pt}}
\tikzset{smallnode/.style={inner sep=1.5pt,fill,outer sep=0,circle =1pt}}
\tikzset{bnode/.style={inner sep=2pt,fill,outer sep=0,circle, color=blue}}
\tikzset{whitenode/.style={inner sep=3pt,draw,outer sep=0,circle}}
\tikzset{emptynode/.style={inner sep=1.25pt,fill=none,outer sep=0,circle,draw}}

\newcommand{\regDdiag}{\begin{tikzpicture}[scale=.5, baseline=-2] 
	\begin{pgfonlayer}{nodelayer}
		\node   (0) at (-1, 1) {};
		\node   (1) at (2, 1) {};
		\node   (2) at (-1, -1) {};
		\node   (3) at (2, -1) {};
		\node   (4) at (-0.5, -1) {};
		\node   (5) at (1.5, -1) {};
		\node   (6) at (1.5, 1) {};
		\node   (7) at (-0.5, 1) {};
		\node   (8) at (-0.5, 2) {};
		\node   (9) at (1.5, 2) {};
		\node   (10) at (-0.5, -2) {};
		\node   (11) at (1.5, -2) {};
		\node   (12) at (0.5, 1.5) {$\cdots$};
		\node   (13) at (0.5, -1.5) {$\cdots$};
		\node   (14) at (-0.5, 2.5) {$1$};
		\node   (15) at (1.5, 2.5) {$1$};
		\node   (16) at (-0.5, -2.5) {$1$};
		\node   (17) at (1.5, -2.5) {$1$};
		\node   (18) at (0.5, 0) {$D$};
	\end{pgfonlayer}
	\begin{pgfonlayer}{edgelayer}
		\draw (2.center) to (0.center);
		\draw (0.center) to (1.center);
		\draw (1.center) to (3.center);
		\draw (3.center) to (2.center);
		\draw (4.center) to (10.center);
		\draw (5.center) to (11.center);
		\draw (8.center) to (7.center);
		\draw (9.center) to (6.center);
	\end{pgfonlayer}
\end{tikzpicture}}

\pgfplotsset{compat=1.18}

\newcommand{\CC}{\mathbb{C}}
\newcommand{\ZZ}{\mathbb{Z}}
\newcommand{\NN}{\mathbb{N}}

\title{Diagrammatic Categories which arise from Representation Graphs}
\author{Ryan Reynolds}

\begin{document}

\maketitle

\pagenumbering{arabic}

\section*{Abstract}

The main result of this paper utilizes the representation graph of a group $G$, $R(V,G)$, and gives a general construction of a diagrammatic category $\mathbf{Dgrams}_{R(V,G)}$.  The proof of the main theorem shows that, given explicit criteria, there is an equivalence of categories between a quotient category of $\mathbf{Dgrams}_{R(V,G)}$ and a full subcategory of $G-\textbf{mod}$ with objects being the tensor products of finitely many irreducible $G$-modules.

{

\tableofcontents

}

\pagebreak

{

\section{Introduction}

In this paper, we provide a framework in which we can develop diagrammatic categories from the data of a representation graph (or McKay quiver) for a group $G$.  In general, there is a full and essentially surjective functor from this diagrammatic category onto a full monoidal subcategory of the category of all finite dimensional $G$-modules.  Furthermore, we provide criteria which, when satisfied, results in an induced functor which is faithful and, thus, an equivalence of categories. 

While this paper focuses on the relationship between our diagrammatic construction and $G$-modules over $\CC$, the results of the main theorem extend to other contexts.  For example, one can consider a fusion category.  In a fusion category, the fusion rule can be described as a graph.  Using this graph, we may define a diagrammatic category utilizing the construction in this paper which will be categorically equivalent to the fusion category with which we started.

To provide some concrete examples, we will focus on the following context for this paper.  Felix Klein classified the finite subgroups of the special unitary group, $SU(2)$.  There are two infinite families of finite subgroups along with 3 exceptional subgroups: the cyclic groups of order $n$, $\mathbf{C}_n$; the binary dihedral groups, $\mathbf{D}_n$, of order $4n$; the binary tetrahedral group $\mathbf{T}$; the binary octahedral group $\mathbf{O}$; and the binary icosahedral group $\mathbf{I}$.  Around 1980, McKay made the observation that certain affine Dynkin diagrams and the representation graphs associated with these finite subgroups are identical \cite{McKay81}.  We will focus on the cyclic groups and the binary tetrahedral group along with their representation graphs.  See Section \ref{sec:McKayGraphs} for details.

As a motivating example, consider for $k\in\ZZ_{\geq 0}$, $\delta\in\CC$, the diagrammatic Temperley--Lieb algebras, $TL_k(\delta)$, were developed by the authors of the same name in \cite{TL71}.  For $k\in\ZZ_{\geq 0}$, there are isomorphisms between the endomorphism algebra $$Z_k(SU(2)):=\operatorname{End}_{SU(2)}(V^{\otimes k})$$ of the natural module $V$ for $SU(2)$ and the Temperley--Lieb algebra $TL_k(2)$.  In \cite{BBH16}, Barnes, Benkart, and Halverson combined the work of McKay and Temperley--Lieb by describing the endomorphism algebras of the finite subgroups of $SU(2)$ and presenting diagrammatics for the $\mathbf{C}_n$ and $\mathbf{D}_n$ cases.  

The study of endomorphism algebras like $Z_k(SU(2))$ can be generalized to more general homomorphism spaces.  For example, we can study $\operatorname{Hom}_{SU(2)}\left(V^{\otimes k}, V^{\otimes \ell}\right)$ for all $k, \ell\in\ZZ_{\geq 0}$.  This gives us new tools, new perspective, and a richer understanding of the representation theory.  With this generalization in mind, the diagrammatic Temperley--Lieb \emph{category} was developed, see \cite{Tur94} and \cite{Chen14}.  This category admits a fully faithful monoidal functor to the category whose objects are tensor products of $V$ and the morphisms are all $SU(2)$-linear maps.  In particular, the Temperley--Lieb algebras appear as endomorphism algebras in the Temperley--Lieb category.  

Surprisingly, entire categories can be easier to derive than individual endomorphism algebras.  In particular, the Temperley--Lieb category has generating diagrams known as the cup, cap, and identity strand, and there is a diagrammatic basis for each space of homomorphisms, Hom$_{SU(2)}(V^{\otimes k}, V^{\otimes \ell})$, which can be described as all non-crossing diagrams with $k$ nodes on the bottom of the diagram and $\ell$ nodes on the top.

In this paper, we explore diagrammatic categories which expand the set of generating objects which correspond to all of the simple $\mathbf{C}_n$-modules and provide explicit relations giving a diagrammatic description of the monoidal full subcategory generated by the irreducible $\mathbf{C}_n$-modules, $\mathbf{C}_n$-\textbf{mod}$_{\text{irr}}$.  In order to give a more general constructions, in section \ref{chpt:Dgrams}, we utilize the representation graph of a group $G$, $R(V,G)$, and give a general construction of a diagrammatic category $\mathbf{Dgrams}_{R(V,G)}$.  The proof of the main theorem shows that, given explicit criteria, there is an equivalence of categories between a quotient category of $\mathbf{Dgrams}_{R(V,G)}$ and $G$-\textbf{mod}$_{\text{irr}}$.  In the final section, we give a few final remarks regarding generalization to directed graphs and give a few examples which show that these results apply outside of the context of $SU(2)$ and its finite subgroups.

We shall close out the introduction with a discussion of certain directions in which this work might extend.  For the constructions in this paper, the functor to the category of $G$-modules can be thought of as a functor to the category of $\CC$-vector spaces once we forget the $G$-action.  In other words, we have a representation of each of these diagrammatic categories.  Just as one group can have many representations, one category can have many interesting representations.  This is an active area of research.  

For example, Sam and Snowden explore the representation theory of the Brauer category in \cite{SS2020}.  They specifically mention that much of the theory they develop could be transferable to other categories, like the Temperley--Lieb categories and its variants.  In particular, we expect it applies to the categories introduced in this paper.

Similarly, Brundan and Vargas give a concrete diagrammatic definition of the affine partition category, and use it to study the representation theory of the partition category \cite{BV2021}.  It is with these two papers in mind that we may ask the following questions.

\begin{question}
Can we classify and study the representations of the diagrammatic categories associated to the finite subgroups of $SU(2)$?  In particular, what is the categorical representation theory of these diagrammatic categories, and can we extend some notions such as highest weight module, semi-simplicity, irreducible modules, etc. to these categories?
\end{question}

In addition to the above questions, there are other directions one might consider exploring.  The hands-on combinatorial nature of this area makes it easy to compute interesting examples and special cases.  Another direction could be to explore how these categories react to changes in certain parameters.  For example, the Temperley--Lieb category, when not considering the connection to $SU(2)$, can be defined with a parameter $\delta\in \CC$ where $\delta=2$ is the Temperley--Lieb category for $SU(2)$.  What would introducing such a parameter to these diagrammatic categories change about the combinatorics or representation theory?  For example, one might explore how these categories decategorify.

\section{Representation Graphs and Diagrammatic categories}
\label{chpt:Prelims}

\noindent This section discusses some of the background material this paper utilizes.  Specifically, we provide the construction of a representation graph, and we give some pertinent background for monoidal categories and an example of a diagrammatic category.

\subsection{Representation Graphs and the McKay Correspondence} 
\label{sec:McKayGraphs} 

\noindent This section is a summary of the work in \cite{BBH16}, which covers this material more comprehensively.  Their work provided important motivating ideas for the constructions in this paper.  

Let us set some notation.  Let $\{G^{(a)}\}_{a\in A}$ be a set of isomorphism class representatives for the simple $G$-modules.  Let $V$ be some $G$-module, not necessarily simple.

\begin{definition} The \textit{representation graph} $R(V,G)$ is a directed graph with nodes labeled by $a \in A$, and if $V\otimes G^{(a)}\cong\bigoplus\limits_{b} \left(G^{(b)}\right)^{m_b}$ where $m_b$ is the multiplicity of $G^{(b)}$ in $V\otimes G^{(a)}$, $R(V,G)$ has $m_b$ directed edges from node $a$ to node $b$.  In the event that there is a pair of directed edges, one from $a$ to $b$ and one from $b$ to $a$, we will represent this by a single undirected edge between $a$ to $b$.
\end{definition}

To illustrate the definition, let us construct an example explicitly. 

\begin{example}
Let $G=SU(2)$ and let $V=\CC^2$.  Let $G^{(a)}=V(a)$ for $a\in A:=\NN$.  Notice that $V=V(1)=G^{(1)}$ is simple, and in fact for each $a \in \mathbb{N}$, there is one irreducible $G$-module of dimension $a+1$.  From the Clebsch--Gordon formula, 
$V\otimes G^{(a)}=G^{(a-1)}\oplus G^{(a+1)}$ for all $a \in \mathbb{N}$.  Thus, the representation graph is the undirected graph

\begin{center}
\begin{tikzpicture}[{circ/.style}={shape=circle, inner sep=1pt, draw, node contents=}]
	\begin{pgfonlayer}{nodelayer}
		\node   (0) at (-4, 0) {};
		\node [thick, style=emptynode] (1) at (-2, 0) {$1$};
		\node [thick, style=emptynode] (2) at (0, 0) {$2$};
		\node [thick, style=emptynode] (3) at (2, 0) {$3$};
		\node   (4) at (4, 0) {};
		\node [thick, style=emptynode] (7) at (-4.25, 0) {$0$};
		\node   (8) at (-2.25, 0) {};
		\node   (9) at (-0.25, 0) {};
		\node   (10) at (0.25, 0) {};
		\node   (11) at (-1.75, 0) {};
		\node   (12) at (1.75, 0) {};
		\node   (13) at (2.25, 0) {};
		\node   (14) at (4.35, 0) {$\ \cdots$};
	\end{pgfonlayer}
	\begin{pgfonlayer}{edgelayer}
		\draw[thick] (0.center) to (8.center);
		\draw[thick] (11.center) to (9.center);
		\draw[thick] (10.center) to (12.center);
		\draw[thick] (13.center) to (4.center);
	\end{pgfonlayer}
\end{tikzpicture}
\end{center}

\noindent where the node $a$ corresponds to $G^{(a)}=V(a)$.
\end{example}

In the 19th century, Felix Klien classified all the finite subgroups of $SU(2)$.  There are two families indexed by $n\in\NN$: the cyclic groups $\mathbf{C}_n$ and the binary dihedral groups $\mathbf{D}_n$; along with three exceptional groups: the binary tetrahedral group, $\mathbf{T}$; binary octahedral group, $\mathbf{O}$; and binary icosahedral group, $\mathbf{I}$.  In 1980, McKay made his rather beautiful observation that the representation graphs $R(G,V)$ of these groups using the natural module $V$ for $SU(2)$ as the defining module are in one-to-one correspondence with the affine Dynkin diagrams of certain types.  The following example makes explicit the correspondence when considering the binary tetrahedral group.

\begin{example}
\label{BTex}

The binary tetrahedral group $\mathbf{T}$ is generated by $X$, $Y$, and $A$ where
\[
X=\left(\begin{array}{c c} i & 0 \\
 0 & -i
 \end{array}\right)\hspace{10mm}
 Y=\left(\begin{array}{c c} 0 & 1 \\
 -1 & 0
 \end{array}\right)\hspace{10mm}
 A=\dfrac{1}{2}\left(\begin{array}{c c} 1+i & 1+i \\
 -1+i & 1-i
 \end{array}\right)\hspace{5mm}
 \]
 
 \noindent and $i=\sqrt{-1}$.  Furthermore, the simple $\mathbf{T}$-modules can be characterized as follows:  there are three $1$-dimensional simple $\mathbf{T}$-modules which we will call $T^{(0)}$, $T^{(4)}$, and $T^{(4')}$, three $2$-dimensional simple $\mathbf{T}$-modules which we will call $T^{(1)}$, $T^{(3)}$, and $T^{(3')}$, and one $3$-dimensional simple $\mathbf{T}$-modules which we will call $T^{(2)}$.  To make this construction explicit, we fix an isomorphism class representative for each simple $\mathbf{T}$-module.  
 
\
 
\noindent $T^{(0)}$ is the trivial module.
 
\noindent$T^{(1)}=\CC\text{-span}\{w_{-1}, w_{1}\}$ where 

$X w_{-1}=iw_{-1}$,  $X w_{1}=-iw_{1}$, $Y w_{-1}=-w_{1}$, $Y w_{1}=w_{-1}$, 

$A w_{-1}=\frac{1}{2}(1+i)w_{-1}+\frac{1}{2}(i-1)w_{1}$, $A w_{1}=\frac{1}{2}(1+i)w_{-1}-\frac{1}{2}(i-1)w_{1}$.

\noindent $T^{(2)}=\CC\text{-span}\{w_{-2}, w_{2}, w_{0'}\}$ where 

$X w_{-2}=-w_{-2}$, $X w_{2}=-w_{2}$, $X w_{0'}=-w_{0'}$, 

$Y w_{-2}=w_{2}$, $Y w_{2}=w_{-2}$, $Y w_{0'}=-w_{0'}$, 

$A w_{-2}=\frac{1}{2}i w_{-2}-\frac{1}{2}i w_{2}-\frac{1}{2} w_{0'}$, $A w_{2}=\frac{1}{2}i w_{-2}-\frac{1}{2}i w_{2}+\frac{1}{2} w_{0'}$, and $A w_{0'}=i w_{-2}+i w_{2}$.

\noindent $T^{(3)}=\CC\text{-span}\{w_{-3}, w_{3}\}$ where

$X w_{-3}=iw_{-3}$,  $X w_{3}=-iw_{3}$, $Y w_{-3}=-w_{3}$, $Y w_{3}=w_{-3}$, 

$A w_{-3}=\frac{1}{4}(\sqrt{3}-1-i(1+\sqrt{3}))w_{-3}+\frac{1}{4}(\sqrt{3}+1+i(-1+\sqrt{3}))w_{3}$, 

$A w_{3}=\frac{1}{4}(\sqrt{3}-1-i(1+\sqrt{3}))w_{-3}-\frac{1}{4}(\sqrt{3}+1+i(-1+\sqrt{3}))w_{3}$.

\noindent $T^{(3')}=\CC\text{-span}\{w_{-3'}, w_{3'}\}$ where

$X w_{-3'}=iw_{-3'}$,  $X w_{3'}=-iw_{3'}$, $Y w_{-3'}=-w_{3'}$, $Y w_{3'}=w_{-3'}$, 

$A w_{-3'}=\frac{1}{4}(-\sqrt{3}-1+i(-1+\sqrt{3}))w_{-3'}+\frac{1}{4}(-\sqrt{3}+1-i(1+\sqrt{3}))w_{3'}$, 

$A w_{3'}=\frac{1}{4}(-\sqrt{3}-1+i(-1+\sqrt{3}))w_{-3'}-\frac{1}{4}(-\sqrt{3}+1-i(1+\sqrt{3}))w_{3'}$.

\noindent $T^{(4)}=\CC\text{-span}\{w_{4}\}$ where

$X w_{4}=w_{4}$, $Y w_{4}=w_{4}$, $A w_{4}=\frac{1}{2}(-i\sqrt{3}-1)w_{4}$.

\noindent $T^{(4')}=\CC\text{-span}\{w_{4}\}$ where

$X w_{4'}=w_{4'}$, $Y w_{4'}=w_{4'}$, $A w_{4'}=\frac{1}{2}(i\sqrt{3}-1)w_{4'}$.  

\

Notice that $T^{(1)}\cong V$ where $V$ is the natural module for $SU(2)$.  Now, we are ready to build the representation graph $R(V,\mathbf{T})$.  Firstly, the $i$th node of $R(V, \mathbf{T})$ corresponds to the simple $\mathbf{T}$-module $T^{(i)}$.  Using the definition of the simples above, we can compute explicitly the direct sum decompositions of certain modules.  In particular,

\begin{align}
\label{eqn:TSums}
T^{(1)}\otimes T^{(0)} &\cong T^{(1)}, &  T^{(1)}\otimes T^{(1)} &\cong T^{(0)}\oplus T^{(2)}, & T^{(1)}\otimes T^{(4')} &\cong T^{(3')}\\
 T^{(1)}\otimes T^{(4)} &\cong T^{(3)}, & T^{(1)}\otimes T^{(3)} &\cong T^{(2)}\oplus T^{(4)}, & T^{(1)}\otimes T^{(3')} &\cong T^{(2)}\oplus T^{(4')},\\
& & \text{and } T^{(1)}\otimes T^{(2)} &\cong T^{(1)}\oplus T^{(3)} \oplus T^{(3')}
\end{align}

Thus,

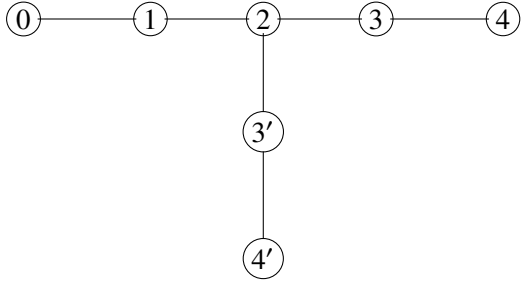
\begin{equation}
\label{eqn:RepGraphT}
\begin{tikzpicture}[scale=0.75, circ/.style={shape=circle, inner sep=1pt, draw, node contents=}]
	\begin{pgfonlayer}{nodelayer}
		\node   (0) at (-4, 0) {};
		\node[style=emptynode]  (1) at (-2, 0) {$1$};
		\node[style=emptynode]   (2) at (0, 0) {$2$};
		\node[style=emptynode]   (3) at (2, 0) {$3$};
		\node   (4) at (4, 0) {};
		\node[style=emptynode]   (5) at (0, -2) {$3'$};
		\node   (6) at (0, -3.9) {};
		\node[style=emptynode]   (7) at (-4.25, 0) {$0$};
		\node   (8) at (-2.25, 0) {};
		\node   (9) at (-0.25, 0) {};
		\node   (10) at (0.25, 0) {};
		\node   (11) at (-1.75, 0) {};
		\node   (12) at (1.75, 0) {};
		\node   (13) at (2.25, 0) {};
		\node[style=emptynode]   (14) at (4.25, 0) {$4$};
		\node   (15) at (0, -1.65) {};
		\node   (16) at (0, -2.35) {};
		\node[style=emptynode]   (17) at (0, -4.25) {$4'$};
		\node   (18) at (0, -0.25) {};
	\end{pgfonlayer}
	\begin{pgfonlayer}{edgelayer}
		\draw (0.center) to (8.center);
		\draw (11.center) to (9.center);
		\draw (10.center) to (12.center);
		\draw (13.center) to (4.center);
		\draw (18.center) to (15.center);
		\draw (16.center) to (6.center);
	\end{pgfonlayer}
\end{tikzpicture}
\end{equation}

\noindent is the realization of the representation graph $R(V, \mathbf{T})$.  Observe that this is the affine Dynkin diagram $\hat{E}_6$.  
\end{example}

In a similar manner, the representation graphs for the other finite subgroups of $SU(2)$, $\mathbf{C}_n$, $\mathbf{D}_n$, $\mathbf{O}$, and $\mathbf{I}$, respectively correspond to the Dynkin diagram $\hat{A}_{n-1}$, $\hat{D}_{n+2}$, $\hat{E}_7$, and $\hat{E}_8$.  

It is advantageous for this paper to establish some notation.  Given a representation graph $R(V,G)$, we let $P(a,b)$ be the set of all paths from $a$ to $b$.  We let $P(a,b)_k$ be the subset of $P(a,b)$ consisting of all paths of length $k$.  A path $\mathbf{p}\in P(a,b)_k$ can be identified with a $k$-tuple $\mathbf{p}=(a, b_1, b_2, \dots, b_{k-1}, b)$ which traverses the nodes $b_i\in I_G$ for $i\in\{1, 2,\dots, k-1\}$.  

\begin{example}
Considering the representation graph of $\mathbf{T}$, $R\left(T^{(1)},\mathbf{T}\right)$ from (\ref{eqn:RepGraphT}).  There are $5$ paths of length $4$ from the node labeled by $1$ to the node labeled by $3$.  Thus, $P(1,3)_4$ has $5$ elements, namely $(1,2,3,4,3)$, $(1,2,3,2,3)$, $(1,2,1,2,3)$, $(1,2,3',2,3)$, and $(1,0,1,2,3)$.  

Using (\ref{eqn:TSums}), we know that $T^{(3)}$ is a direct summand of $\left(T^{(1)}\right)^{\otimes 5}$ and has multiplicity $5$.  Consider the path $\mathbf{p}=\left(1,2,3,4,3\right)$.  Each path corresponds to a unique isomorphic copy of $T^{(3)}$ as a submodule of $\left(T^{(1)}\right)^{\otimes 5}$ in a canonical way.  This construction is given in general for a group $G$ in \ref{sec:Gmodirr}.
\end{example}


\subsection{Diagrammatic Temperley--Lieb Category}
\label{sec:Dcat}


\noindent As the main goal of this paper is to develop diagrammatic categories which describe certain categories of representations, let us begin with a few categorical notions.  In order to define a category, one must give a collection of objects and a collection of morphisms which contains the identity morphism for each object, are closed under composition, and satisfy associativity.  The diagrammatic categories in this paper will all be strict, monoidal, and $\CC$-linear.  The following are the necessary definitions from \cite{EGNO16} with some of the technical details suppressed.

\begin{definition}
A \emph{monoidal category} is a quintuple $\left(\mathcal{C}, \otimes, a, \mathbf{1}, \iota\right)$, where $\mathcal{C}$ is a category, $\otimes: \mathcal{C}\otimes\mathcal{C}\longrightarrow\mathcal{C}$ is a bifunctor called the tensor product bifunctor, $a:\left(X\otimes Y\right)\otimes Z \xrightarrow{\sim} X \otimes\left(Y\otimes Z \right)$ is the \emph{associator} and a natural isomorphism for all objects $X,\ Y,$ and $Z$ in $\mathcal{C}$, $\mathbf{1}$ is an object of $\mathcal{C}$, and $\iota:\mathbf{1}\otimes\mathbf{1}\xrightarrow{\sim}\mathbf{1}$ is the \emph{unitor} and an isomorphism, all subject to the pentagon axiom and the unit axiom.
\end{definition}

Essentially, a monoidal category allows for tensor products of objects and morphisms in which there is an associator and a unit object.  A $\CC$-linear category asserts that the class of morphisms are in fact vector spaces over the field $\CC$ and with composition acting linearly.  

In a similar way to group or monoid presentation, we can define a $\CC$-linear monoidal category using generators and relations.  For a technical discussion of this, see \cite{EGNO16, MacL98}.  Let $\mathcal{C}$ be a monoidal category.  A collection $S$ of objects in $\mathcal{C}$ generates the objects of $\mathcal{C}$ if every object can be realized as the tensor product of elements of $S$.  Furthermore, a collection $M$ of morphisms in $\mathcal{C}$ generates the morphisms of $\mathcal{C}$ if every morphism can be realized using linear combinations, compositions, and tensor products of elements of $M$.  On the other hand, given a set of objects $S$ and a set of morphisms $M$, we can construct the free monoidal category on these sets.  One can also impose \emph{relations} on morphisms between objects.  Let $R$ be a collection of relations for morphisms in $\mathcal{C}$, and let $\mathcal{I}_R$ be the tensor ideal generated by $R$.  If $\mathcal{C}$ is generated by $S$ and $M$, then the quotient category $\mathcal{C}/\mathcal{I}_R$ is said to be generated by $S$ and $M$ subject to the relations $R$.  

\begin{definition}
A \emph{strict} monoidal category is a monoidal category in which the associator and the unitor are identity morphisms.
\end{definition}

There is a subtle issue with the functors in this paper.  All of our diagrammatic categories are strict, yet the target categories are from representation theory, and the unitor of the category of $G$-modules for a group $G$ is not the identity morphism.  However, this is not really an issue since we have Mac Lane's Strictness Theorem from \cite{EGNO16, MacL98}.

\begin{theorem}
Any monoidal category is monoidally equivalent to a strict monoidal category.
\end{theorem}

In order to give an example of the above definitions, let us first discuss some motivation.  Much of our discussion will be centered around defining diagrammatic algebras and categories which are specifically designed to mirror the workings of a category coming from representation theory.  

For example, consider the well-known Temperley--Lieb algebra $TL_k(\delta)$ which can be defined by generators $e_1, \dots, e_{k-1}$ and subject to the relations $$e_i^2=\delta e_i, \ e_i e_{i\pm 1} e_i = e_i\text{, and }e_i e_j =e_j e_i\text{ for }\lvert i-j \rvert > 1.$$  The algebra $TL_k(\delta)$ can be viewed diagrammatically as well where 
\begin{center}
\begin{tikzpicture}[scale=.5, baseline=0]
	\begin{pgfonlayer}{nodelayer}
		\node   (0) at (-3, 1) {};
		\node   (1) at (-3, -1) {};
		\node   (2) at (-2, 0) {$\cdots$};
		\node   (3) at (-1, 1) {};
		\node   (4) at (-1, -1) {};
		\node   (5) at (0, 1) {};
		\node   (6) at (1, 1) {};
		\node   (7) at (0, -1) {};
		\node   (8) at (1, -1) {};
		\node   (9) at (2, 1) {};
		\node   (10) at (2, -1) {};
		\node   (11) at (3, 0) {$\cdots$};
		\node   (12) at (4, 1) {};
		\node   (13) at (4, -1) {};
		\node   (14) at (-5, 0) {$e_i:=$};
		\node   (15) at (-3, 1.5) {$\mathsmaller{1}$};
		\node   (16) at (0, 1.5) {$\mathsmaller{i}$};
		\node   (17) at (1, 1.5) {$\mathsmaller{i+1}$};
		\node   (18) at (4, 1.5) {$\mathsmaller{k}$};
	\end{pgfonlayer}
	\begin{pgfonlayer}{edgelayer}
		\draw (0.center) to (1.center);
		\draw (3.center) to (4.center);
		\draw [bend right=90, looseness=2.50] (5.center) to (6.center);
		\draw [bend left=90, looseness=2.50] (7.center) to (8.center);
		\draw (9.center) to (10.center);
		\draw (12.center) to (13.center);
	\end{pgfonlayer}
\end{tikzpicture}.
\end{center}  

\noindent Then the Temperley--Lieb algebra $TL_3(\delta)$ has a basis given by the following diagrams:

\

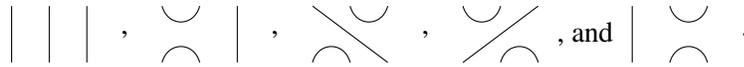
\begin{center}
\begin{tikzpicture}[scale=.5]
	\begin{pgfonlayer}{nodelayer}
		\node   (0) at (-6, 0.5) {};
		\node   (1) at (-5, 0.5) {};
		\node   (2) at (-4, 0.5) {};
		\node   (3) at (-6, -1) {};
		\node   (4) at (-5, -1) {};
		\node   (5) at (-4, -1) {};
		\node   (6) at (-2, 0.5) {};
		\node   (7) at (-1, 0.5) {};
		\node   (8) at (0, 0.5) {};
		\node   (9) at (-2, -1) {};
		\node   (10) at (-1, -1) {};
		\node   (11) at (0, -1) {};
		\node   (12) at (2, 0.5) {};
		\node   (13) at (3, 0.5) {};
		\node   (14) at (4, 0.5) {};
		\node   (15) at (2, -1) {};
		\node   (16) at (3, -1) {};
		\node   (17) at (4, -1) {};
		\node   (18) at (6, 0.5) {};
		\node   (19) at (7, 0.5) {};
		\node   (20) at (8, 0.5) {};
		\node   (21) at (6, -1) {};
		\node   (22) at (7, -1) {};
		\node   (23) at (8, -1) {};
		\node   (24) at (10.5, 0.5) {};
		\node   (25) at (11.5, 0.5) {};
		\node   (26) at (12.5, 0.5) {};
		\node   (27) at (10.5, -1) {};
		\node   (28) at (11.5, -1) {};
		\node   (29) at (12.5, -1) {};
		\node   (30) at (-3, -0.25) {,};
		\node   (31) at (1, -0.25) {,};
		\node   (32) at (5, -0.25) {,};
		\node   (33) at (9.25, -0.25) { , and  };
		\node   (34) at (13.5, -0.25) {.};
	\end{pgfonlayer}
	\begin{pgfonlayer}{edgelayer}
		\draw (0.center) to (3.center);
		\draw (1.center) to (4.center);
		\draw (2.center) to (5.center);
		\draw [bend left=75, looseness=1.50] (9.center) to (10.center);
		\draw [bend left=75, looseness=1.50] (7.center) to (6.center);
		\draw (8.center) to (11.center);
		\draw [bend left=75, looseness=1.50] (15.center) to (16.center);
		\draw [bend right=75, looseness=1.50] (13.center) to (14.center);
		\draw (12.center) to (17.center);
		\draw [bend right=75, looseness=1.50] (18.center) to (19.center);
		\draw [bend left=75, looseness=1.50] (22.center) to (23.center);
		\draw (21.center) to (20.center);
		\draw [bend left=75, looseness=1.50] (28.center) to (29.center);
		\draw [bend right=75, looseness=1.50] (25.center) to (26.center);
		\draw (24.center) to (27.center);
	\end{pgfonlayer}
\end{tikzpicture}
\end{center}

The composition product is given by vertically stacking diagrams as shown in the next example.  Furthermore, whenever there is a closed connected component, we delete it and multiply the resulting diagram by a factor of $\delta$.  

\begin{example}
Let

\begin{center}
\begin{tikzpicture}[scale=.65]
	\begin{pgfonlayer}{nodelayer}
		\node   (0) at (-1, 1) {};
		\node   (1) at (4, 1) {};
		\node   (2) at (5, 1) {};
		\node   (3) at (1, -0.5) {};
		\node   (4) at (2, -0.5) {};
		\node   (5) at (3, -0.5) {};
		\node   (6) at (1, 1) {};
		\node   (7) at (2, 1) {};
		\node   (8) at (-2, 1) {};
		\node   (9) at (-2, -0.5) {};
		\node   (10) at (-1, -0.5) {};
		\node   (11) at (0, -0.5) {};
		\node   (12) at (6, 1) {};
		\node   (13) at (0, 1) {};
		\node   (14) at (3, 1) {};
		\node   (15) at (4, -0.5) {};
		\node   (16) at (5, -0.5) {};
		\node   (17) at (6, -0.5) {};
		\node   (18) at (12, 1) {};
		\node   (19) at (13, 1) {};
		\node   (20) at (14, 1) {};
		\node   (21) at (12, -0.5) {};
		\node   (22) at (13, -0.5) {};
		\node   (23) at (14, -0.5) {};
		\node   (24) at (15, 1) {};
		\node   (25) at (16, 1) {};
		\node   (26) at (17, 1) {};
		\node   (27) at (15, -0.5) {};
		\node   (28) at (16, -0.5) {};
		\node   (29) at (17, -0.5) {};
		\node   (30) at (9, 1) {};
		\node   (31) at (10, 1) {};
		\node   (32) at (11, 1) {};
		\node   (33) at (9, -0.5) {};
		\node   (34) at (10, -0.5) {};
		\node   (35) at (11, -0.5) {};
		\node   (54) at (-3, 0.25) {$d_1=$};
		\node   (56) at (8, 0.25) {and $d_2=$};
	\end{pgfonlayer}
	\begin{pgfonlayer}{edgelayer}
		\draw (0.center) to (3.center);
		\draw (1.center) to (4.center);
		\draw (2.center) to (5.center);
		\draw [bend left=75, looseness=1.50] (9.center) to (10.center);
		\draw [bend left=75, looseness=1.50] (7.center) to (6.center);
		\draw (8.center) to (11.center);
		\draw [bend left=75, looseness=1.50] (15.center) to (16.center);
		\draw [bend right=60] (13.center) to (14.center);
		\draw (12.center) to (17.center);
		\draw [bend right=75, looseness=1.50] (18.center) to (19.center);
		\draw [bend left=75, looseness=1.50] (22.center) to (23.center);
		\draw [bend left=75, looseness=1.50] (28.center) to (29.center);
		\draw [bend right=75, looseness=1.50] (25.center) to (26.center);
		\draw [bend left=75, looseness=1.50] (33.center) to (34.center);
		\draw [bend left=75, looseness=1.50] (31.center) to (30.center);
		\draw (32.center) to (35.center);
		\draw [bend right=75, looseness=1.50] (20.center) to (24.center);
		\draw [bend left=300] (27.center) to (21.center);
	\end{pgfonlayer}
\end{tikzpicture}.
\end{center}

\noindent We connect the diagrams in the obvious way:

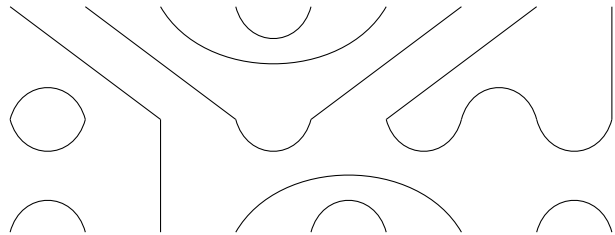
\begin{center}
\begin{tikzpicture}
	\begin{pgfonlayer}{nodelayer}
		\node   (0) at (-1, 1) {};
		\node   (1) at (4, 1) {};
		\node   (2) at (5, 1) {};
		\node   (3) at (1, -0.5) {};
		\node   (4) at (2, -0.5) {};
		\node   (5) at (3, -0.5) {};
		\node   (6) at (1, 1) {};
		\node   (7) at (2, 1) {};
		\node   (8) at (-2, 1) {};
		\node   (9) at (-2, -0.5) {};
		\node   (10) at (-1, -0.5) {};
		\node   (11) at (0, -0.5) {};
		\node   (12) at (6, 1) {};
		\node   (13) at (0, 1) {};
		\node   (14) at (3, 1) {};
		\node   (15) at (4, -0.5) {};
		\node   (16) at (5, -0.5) {};
		\node   (17) at (6, -0.5) {};
		\node   (18) at (1, -0.5) {};
		\node   (19) at (2, -0.5) {};
		\node   (20) at (3, -0.5) {};
		\node   (21) at (1, -2) {};
		\node   (22) at (2, -2) {};
		\node   (23) at (3, -2) {};
		\node   (24) at (4, -0.5) {};
		\node   (25) at (5, -0.5) {};
		\node   (26) at (6, -0.5) {};
		\node   (27) at (4, -2) {};
		\node   (28) at (5, -2) {};
		\node   (29) at (6, -2) {};
		\node   (30) at (-2, -0.5) {};
		\node   (31) at (-1, -0.5) {};
		\node   (32) at (0, -0.5) {};
		\node   (33) at (-2, -2) {};
		\node   (34) at (-1, -2) {};
		\node   (35) at (0, -2) {};
		\node   (39) at (1, 1) {};
		\node   (40) at (2, 1) {};
		\node   (41) at (3, 1) {};
		\node   (45) at (4, 1) {};
		\node   (46) at (5, 1) {};
		\node   (47) at (6, 1) {};
		\node   (51) at (-2, 1) {};
		\node   (52) at (-1, 1) {};
		\node   (53) at (0, 1) {};
		\node   (54) at (-3, -0.5) {};
	\end{pgfonlayer}
	\begin{pgfonlayer}{edgelayer}
		\draw (0.center) to (3.center);
		\draw (1.center) to (4.center);
		\draw (2.center) to (5.center);
		\draw [bend left=75, looseness=1.50] (9.center) to (10.center);
		\draw [bend left=75, looseness=1.50] (7.center) to (6.center);
		\draw (8.center) to (11.center);
		\draw [bend left=75, looseness=1.50] (15.center) to (16.center);
		\draw [bend right=60] (13.center) to (14.center);
		\draw (12.center) to (17.center);
		\draw [bend right=75, looseness=1.50] (18.center) to (19.center);
		\draw [bend left=75, looseness=1.50] (22.center) to (23.center);
		\draw [bend left=75, looseness=1.50] (28.center) to (29.center);
		\draw [bend right=75, looseness=1.50] (25.center) to (26.center);
		\draw [bend left=75, looseness=1.50] (33.center) to (34.center);
		\draw [bend left=75, looseness=1.50] (31.center) to (30.center);
		\draw (32.center) to (35.center);
		\draw [bend right=75, looseness=1.50] (20.center) to (24.center);
		\draw [bend left=300] (27.center) to (21.center);
	\end{pgfonlayer}
\end{tikzpicture}
\end{center}

\

\noindent and we use isotopies to straighten out connected components, as well as delete any connected components contained completely in the middle of the diagram to get

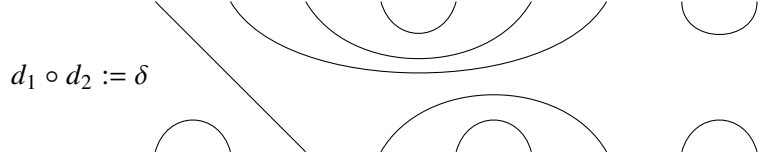
\begin{center}
\begin{tikzpicture}
	\begin{pgfonlayer}{nodelayer}
		\node   (0) at (-1, 0) {};
		\node   (1) at (4, 0) {};
		\node   (2) at (5, 0) {};
		\node   (6) at (1, 0) {};
		\node   (7) at (2, 0) {};
		\node   (8) at (-2, 0) {};
		\node   (12) at (6, 0) {};
		\node   (13) at (0, 0) {};
		\node   (14) at (3, 0) {};
		\node   (21) at (1, -2) {};
		\node   (22) at (2, -2) {};
		\node   (23) at (3, -2) {};
		\node   (27) at (4, -2) {};
		\node   (28) at (5, -2) {};
		\node   (29) at (6, -2) {};
		\node   (33) at (-2, -2) {};
		\node   (34) at (-1, -2) {};
		\node   (35) at (0, -2) {};
		\node   (39) at (1, 0) {};
		\node   (40) at (2, 0) {};
		\node   (41) at (3, 0) {};
		\node   (45) at (4, 0) {};
		\node   (46) at (5, 0) {};
		\node   (47) at (6, 0) {};
		\node   (52) at (-1, 0) {};
		\node   (53) at (0, 0) {};
		\node   (54) at (-3, -1) {$d_1\circ d_2:=\delta$};
	\end{pgfonlayer}
	\begin{pgfonlayer}{edgelayer}
		\draw [bend left=75, looseness=1.50] (7.center) to (6.center);
		\draw [bend right=60] (13.center) to (14.center);
		\draw [bend left=75, looseness=1.50] (22.center) to (23.center);
		\draw [bend left=75, looseness=1.50] (28.center) to (29.center);
		\draw [bend left=75, looseness=1.50] (33.center) to (34.center);
		\draw [bend left=300] (27.center) to (21.center);
		\draw (8.center) to (35.center);
		\draw [bend right=60, looseness=0.75] (52.center) to (45.center);
		\draw [bend right=90, looseness=1.50] (46.center) to (47.center);
	\end{pgfonlayer}
\end{tikzpicture}
\end{center}

\end{example}

By setting $\delta = 2$, we get the following theorem.

\begin{theorem}\cite{West95} For all $k\geq 0$, there is an isomorphism of algebras

\begin{equation*}
TL_k(2)\ \xrightarrow{\cong} \text{ End}_{SU(2)}\left(V^{\otimes k}\right).
\end{equation*}

\end{theorem}

Thus, we have a diagrammatic presentation for the endomorphism algebra End$_{SU(2)}(V^{\otimes k})$.  

We are now ready to give an example of a monoidal $\CC$-linear category given by generators and relations.  In particular, we can generalize this description and obtain the Temperley--Lieb category $TL(\delta)$ by allowing the number of vertices on top and bottom to vary.  Thus, $TL(\delta)$ can be defined as the monoidal $\CC$-linear category generated by one object $\mathlarger{\mathlarger{\mathlarger{\mathlarger{\cdot}}}}$ and the morphisms 

\begin{center}
\begin{tikzpicture}[scale=.5, baseline=0]
	\begin{pgfonlayer}{nodelayer}
		\node  (0) at (0, 1) {};
		\node  (1) at (0, -1) {};
	\end{pgfonlayer}
	\begin{pgfonlayer}{edgelayer}
		\draw (0.center) to (1.center);
	\end{pgfonlayer}
\end{tikzpicture} ,\hspace{10mm} \begin{tikzpicture}[scale=.5, baseline=6]
	\begin{pgfonlayer}{nodelayer}
		\node  (0) at (0, 0) {};
		\node  (1) at (1, 0) {};
	\end{pgfonlayer}
	\begin{pgfonlayer}{edgelayer}
		\draw [bend left=90, looseness=3.50] (0.center) to (1.center);
	\end{pgfonlayer}
\end{tikzpicture} , \hspace{10mm} and \begin{tikzpicture}[scale=.5, baseline=6]
	\begin{pgfonlayer}{nodelayer}
		\node  (0) at (0, 1) {};
		\node  (1) at (1, 1) {};
	\end{pgfonlayer}
	\begin{pgfonlayer}{edgelayer}
		\draw [bend right=90, looseness=3.50] (0.center) to (1.center);
	\end{pgfonlayer}
\end{tikzpicture} .
\end{center}

\noindent Composition is given by vertical concatenation, when this is possible.  The monoidal product is given by horizontal concatenation.  These operations are subject to the same relations as above, namely isotopy equivalence and a factor of $\delta$ gets multiplied for each closed connected component deleted.  There is then a fully faithful functor $$TL(2)\  \xrightarrow{\cong}\ SU(2)\text{-}\textbf{mod}$$ given on objects by $\mathlarger{\mathlarger{\mathlarger{\mathlarger{\cdot}}}}^{\otimes k}\mapsto V^{\otimes k}$.  This functor defines an equivalence into $SU(2)\text{-}\textbf{mod}_V$.  From this equivalence, we have a diagrammatic basis for the spaces of $SU(2)$-invariant homomorphisms, Hom$_{SU(2)}\left(V^{\otimes k },V^{\otimes \ell}\right)$ for all $k, \ell\in\ZZ$.  In particular, the non-crossing diagrams with $k$ nodes on bottom and $\ell$ nodes on top and where each node has valence precisely $1$ form this basis.

For this paper, the categories we work with will be semisimple.  That is, every object in the category will be isomorphic to the direct sum of simple objects.


\section{Categories with Irreducible \texorpdfstring{$\mathbf{C}_n$}{Cn}-modules as Objects}
\label{chpt:CnIrr}


Let us explore some new families of categories.  We will again use the convention that the empty diagram is the morphism from $(0)$ to $(0)$ which represents multiplication by $1$ where $(0)=\bar{0}$ is the identity object.

\begin{definition}
Let $\mathcal{C}_n^{\text{irr}}$ be the $\CC$-linear monoidal category with objects generated by $a\in \ZZ/n\ZZ$ with the tensor product being defined by concatenation.  Denote the concatenation of the integers $a_1, a_2, \dots , a_k$ as $[a_1, a_2, \dots , a_k]$.  The morphisms are generated by the following diagrams:

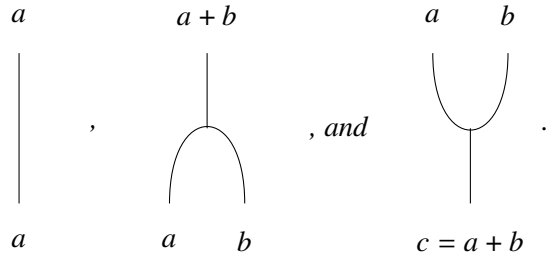
\begin{center}
\begin{tikzpicture}
	\begin{pgfonlayer}{nodelayer}
		\node   (0) at (-4, 0) {};
		\node   (1) at (-3, 0) {};
		\node   (2) at (-3.5, 2) {};
		\node   (3) at (-3.5, 1) {};
		\node   (4) at (-6, 2) {};
		\node   (5) at (-6, 0) {};
		\node   (6) at (-0.5, 2) {};
		\node   (7) at (0.5, 2) {};
		\node   (8) at (0, 1) {};
		\node   (9) at (0, 0) {};
		\node   (10) at (-6, 2.5) {$a$};
		\node   (11) at (-6, -0.5) {$a$};
		\node   (12) at (-3.5, 2.5) {$a+b$};
		\node   (13) at (-4, -0.5) {$a$};
		\node   (14) at (-3, -0.5) {$b$};
		\node   (15) at (0, -0.5) {$c=a+b$};
		\node   (16) at (-0.5, 2.5) {$a$};
		\node   (17) at (0.5, 2.5) {$b$};
		\node   (18) at (-5, 1) {, };
		\node   (19) at (-1.75, 1) {, and};
		\node   (20) at (1,1) {.};
	\end{pgfonlayer}
	\begin{pgfonlayer}{edgelayer}
		\draw (4.center) to (5.center);
		\draw [bend left=90, looseness=3.50] (0.center) to (1.center);
		\draw (3.center) to (2.center);
		\draw (8.center) to (9.center);
		\draw [bend right=90, looseness=3.50] (6.center) to (7.center);
	\end{pgfonlayer}
\end{tikzpicture}
\end{center}

\noindent where $a,\ b,\text{ and }c\in \ZZ/n\ZZ$.  We will sometimes refer to these as the identity diagram, the merge diagram, and the split diagram respectively.  

We impose the following relations:

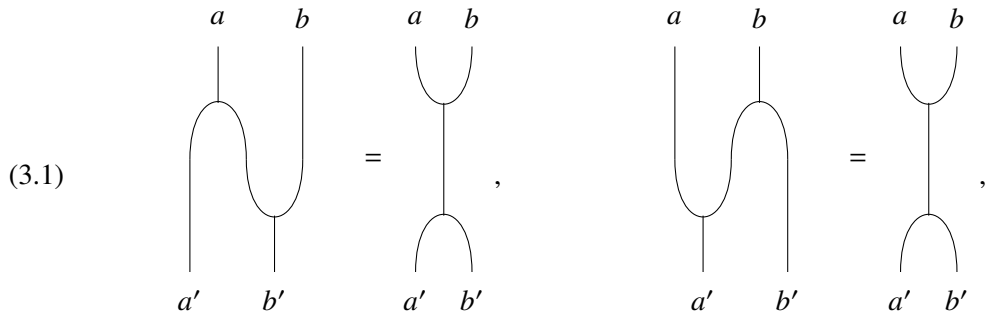
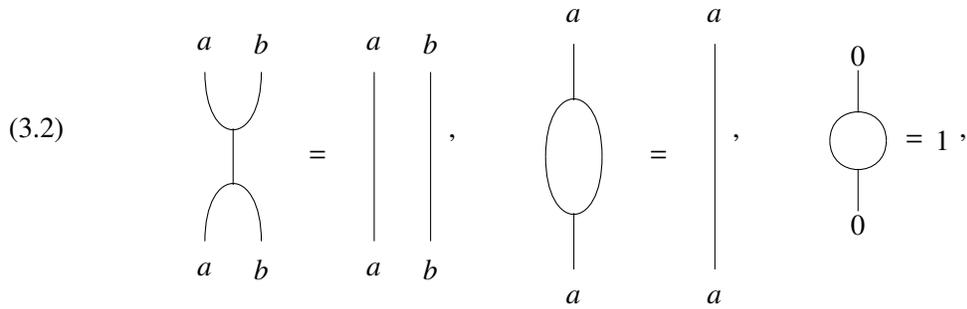
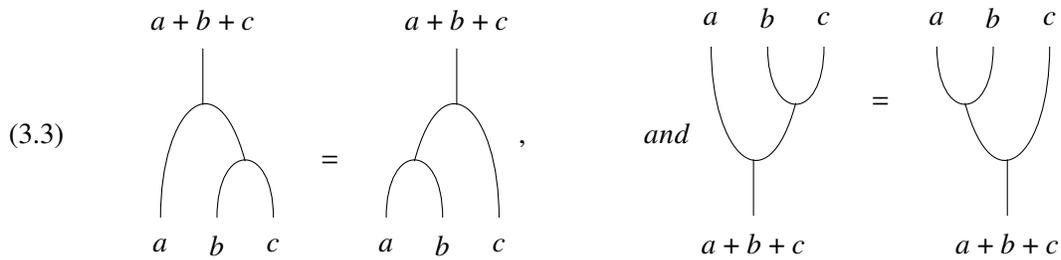
\begin{center}
\begin{equation}
\label{eqn:CnnRelation1}
\begin{tikzpicture}[scale=.75, baseline=12]
	\begin{pgfonlayer}{nodelayer}
		\node   (0) at (-3, -1) {};
		\node   (1) at (-3, 1) {};
		\node   (2) at (-2, 1) {};
		\node   (3) at (-1, 1) {};
		\node   (4) at (-1, 3) {};
		\node   (5) at (-1.5, 0) {};
		\node   (6) at (-1.5, -1) {};
		\node   (7) at (-2.5, 2) {};
		\node   (8) at (-2.5, 3) {};
		\node   (9) at (1, -1) {};
		\node   (10) at (2, -1) {};
		\node   (11) at (1.5, 0) {};
		\node   (12) at (1.5, 2) {};
		\node   (13) at (2, 3) {};
		\node   (14) at (1, 3) {};
		\node   (15) at (-2.5, 3.5) {$a$};
		\node   (16) at (-1, 3.5) {$b$};
		\node   (17) at (-3, -1.5) {$a'$};
		\node   (18) at (-1.5, -1.5) {$b'$};
		\node   (19) at (1, -1.5) {$a'$};
		\node   (20) at (2, -1.5) {$b'$};
		\node   (21) at (1, 3.5) {$a$};
		\node   (22) at (2, 3.5) {$b$};
		\node   (23) at (0.25, 1) {$=$};
	\end{pgfonlayer}
	\begin{pgfonlayer}{edgelayer}
		\draw (0.center) to (1.center);
		\draw [bend left=90, looseness=3.50] (1.center) to (2.center);
		\draw [bend right=90, looseness=3.50] (2.center) to (3.center);
		\draw (5.center) to (6.center);
		\draw (7.center) to (8.center);
		\draw (4.center) to (3.center);
		\draw [bend right=90, looseness=3.50] (14.center) to (13.center);
		\draw (12.center) to (11.center);
		\draw [bend left=90, looseness=3.50] (9.center) to (10.center);
	\end{pgfonlayer}
\end{tikzpicture}, \hspace{20mm}
\begin{tikzpicture}[scale=.75, baseline=12]
	\begin{pgfonlayer}{nodelayer}
		\node   (0) at (-1, -1) {};
		\node   (1) at (-1, 1) {};
		\node   (2) at (-2, 1) {};
		\node   (3) at (-3, 1) {};
		\node   (4) at (-3, 3) {};
		\node   (5) at (-2.5, 0) {};
		\node   (6) at (-2.5, -1) {};
		\node   (7) at (-1.5, 2) {};
		\node   (8) at (-1.5, 3) {};
		\node   (9) at (1, -1) {};
		\node   (10) at (2, -1) {};
		\node   (11) at (1.5, 0) {};
		\node   (12) at (1.5, 2) {};
		\node   (13) at (2, 3) {};
		\node   (14) at (1, 3) {};
		\node   (15) at (-3, 3.5) {$a$};
		\node   (16) at (-1.5, 3.5) {$b$};
		\node   (17) at (-1, -1.5) {$b'$};
		\node   (18) at (-2.5, -1.5) {$a'$};
		\node   (19) at (1, -1.5) {$a'$};
		\node   (20) at (2, -1.5) {$b'$};
		\node   (21) at (1, 3.5) {$a$};
		\node   (22) at (2, 3.5) {$b$};
		\node   (23) at (0.25, 1) {$=$};
	\end{pgfonlayer}
	\begin{pgfonlayer}{edgelayer}
		\draw (0.center) to (1.center);
		\draw [bend right=90, looseness=3.50] (1.center) to (2.center);
		\draw [bend left=90, looseness=3.50] (2.center) to (3.center);
		\draw (5.center) to (6.center);
		\draw (7.center) to (8.center);
		\draw (4.center) to (3.center);
		\draw [bend right=90, looseness=3.50] (14.center) to (13.center);
		\draw (12.center) to (11.center);
		\draw [bend left=90, looseness=3.50] (9.center) to (10.center);
	\end{pgfonlayer}
\end{tikzpicture}, 
\end{equation}

\begin{equation}
\label{eqn:CnnRelation2}
\begin{tikzpicture}[scale=.75, baseline=18]
	\begin{pgfonlayer}{nodelayer}
		\node   (9) at (-2, -1) {};
		\node   (10) at (-1, -1) {};
		\node   (11) at (-1.5, 0) {};
		\node   (12) at (-1.5, 1) {};
		\node   (13) at (-1, 2) {};
		\node   (14) at (-2, 2) {};
		\node   (19) at (-2, -1.5) {$a$};
		\node   (20) at (-1, -1.5) {$b$};
		\node   (21) at (-2, 2.5) {$a$};
		\node   (22) at (-1, 2.5) {$b$};
		\node   (23) at (0, 0.5) {$=$};
		\node   (24) at (1, 2) {};
		\node   (25) at (2, 2) {};
		\node   (26) at (1, -1) {};
		\node   (27) at (2, -1) {};
		\node   (28) at (1, 2.5) {$a$};
		\node   (29) at (2, 2.5) {$b$};
		\node   (30) at (1, -1.5) {$a$};
		\node   (31) at (2, -1.5) {$b$};
	\end{pgfonlayer}
	\begin{pgfonlayer}{edgelayer}
		\draw [bend right=90, looseness=3.50] (14.center) to (13.center);
		\draw (12.center) to (11.center);
		\draw [bend left=90, looseness=3.50] (9.center) to (10.center);
		\draw (24.center) to (26.center);
		\draw (25.center) to (27.center);
	\end{pgfonlayer}
\end{tikzpicture}, \hspace{10mm}
\begin{tikzpicture}[scale=.75, baseline=18]
	\begin{pgfonlayer}{nodelayer}
		\node   (23) at (0, 0.5) {$=$};
		\node   (25) at (1, 2.5) {};
		\node   (27) at (1, -1.5) {};
		\node   (29) at (1, 3) {$a$};
		\node   (31) at (1, -2) {$a$};
		\node   (33) at (-1.5, 2.5) {};
		\node   (34) at (-1.5, 1.5) {};
		\node   (35) at (-2, 0.5) {};
		\node   (36) at (-1, 0.5) {};
		\node   (37) at (-1.5, -0.5) {};
		\node   (38) at (-1.5, -1.5) {};
		\node   (39) at (-1.5, -2) {$a$};
		\node   (40) at (-1.5, 3) {$a$};
	\end{pgfonlayer}
	\begin{pgfonlayer}{edgelayer}
		\draw (25.center) to (27.center);
		\draw (33.center) to (34.center);
		\draw [bend left=90, looseness=3.50] (35.center) to (36.center);
		\draw [bend left=90, looseness=3.50] (36.center) to (35.center);
		\draw (38.center) to (37.center);
	\end{pgfonlayer}
\end{tikzpicture}, \hspace{10mm}
\begin{tikzpicture}[scale=.75, baseline=12]
	\begin{pgfonlayer}{nodelayer}
		\node   (0) at (-1, 2) {$0$};
		\node   (1) at (-1, 1.75) {};
		\node   (2) at (-1, 1) {};
		\node   (3) at (-1.5, 0.5) {};
		\node   (4) at (-0.5, 0.5) {};
		\node   (5) at (-1, 0) {};
		\node   (6) at (-1, -0.75) {};
		\node   (7) at (-1, -1) {$0$};
		\node   (8) at (0, 0.5) {$=$};
		\node   (9) at (0.5, 0.5) {$1$};
	\end{pgfonlayer}
	\begin{pgfonlayer}{edgelayer}
		\draw (1.center) to (2.center);
		\draw (5.center) to (6.center);
		\draw [bend left=90, looseness=1.75] (3.center) to (4.center);
		\draw [bend right=90, looseness=1.75] (3.center) to (4.center);
	\end{pgfonlayer}
\end{tikzpicture},
\end{equation}
 
\begin{equation}
\label{eqn:CnnRelation3}
\begin{tikzpicture}[scale=.75, baseline=28]
	\begin{pgfonlayer}{nodelayer}
		\node   (0) at (-4, 0) {};
		\node   (1) at (-3, 0) {};
		\node   (2) at (-2, 0) {};
		\node   (3) at (-2.5, 1) {};
		\node   (4) at (-3.25, 2) {};
		\node   (5) at (-3.25, 3) {};
		\node   (6) at (2, 0) {};
		\node   (7) at (0, 0) {};
		\node   (8) at (1, 0) {};
		\node   (9) at (0.5, 1) {};
		\node   (10) at (1.25, 2) {};
		\node   (11) at (1.25, 3) {};
		\node   (12) at (-4, -0.5) {$a$};
		\node   (13) at (-3, -0.5) {$b$};
		\node   (14) at (-2, -0.5) {$c$};
		\node   (15) at (-3.25, 3.5) {$a+b+c$};
		\node   (16) at (0, -0.5) {$a$};
		\node   (17) at (1, -0.5) {$b$};
		\node   (18) at (2, -0.5) {$c$};
		\node   (19) at (1.25, 3.5) {$a+b+c$};
		\node   (20) at (-1, 1) {$=$};
	\end{pgfonlayer}
	\begin{pgfonlayer}{edgelayer}
		\draw [bend left=90, looseness=3.50] (1.center) to (2.center);
		\draw [in=105, out=90, looseness=2.75] (0.center) to (3.center);
		\draw (5.center) to (4.center);
		\draw [bend left=90, looseness=3.50] (7.center) to (8.center);
		\draw [in=75, out=90, looseness=2.75] (6.center) to (9.center);
		\draw (11.center) to (10.center);
	\end{pgfonlayer}
\end{tikzpicture}, \hspace{15mm} and 
\begin{tikzpicture}[scale=.75, baseline=6]
	\begin{pgfonlayer}{nodelayer}
		\node   (0) at (-3, 2) {};
		\node   (1) at (-2, 2) {};
		\node   (2) at (-1, 2) {};
		\node   (3) at (-1.5, 1) {};
		\node   (4) at (-2.25, 0) {};
		\node   (5) at (-2.25, -1) {};
		\node   (6) at (3, 2) {};
		\node   (7) at (1, 2) {};
		\node   (8) at (2, 2) {};
		\node   (9) at (1.5, 1) {};
		\node   (10) at (2.25, 0) {};
		\node   (11) at (2.25, -1) {};
		\node   (12) at (-3, 2.5) {$a$};
		\node   (13) at (-2, 2.5) {$b$};
		\node   (14) at (-1, 2.5) {$c$};
		\node   (15) at (-2.25, -1.5) {$a+b+c$};
		\node   (16) at (1, 2.5) {$a$};
		\node   (17) at (2, 2.5) {$b$};
		\node   (18) at (3, 2.5) {$c$};
		\node   (19) at (2.25, -1.5) {$a+b+c$};
		\node   (20) at (0, 1) {$=$};
	\end{pgfonlayer}
	\begin{pgfonlayer}{edgelayer}
		\draw [bend right=90, looseness=3.50] (1.center) to (2.center);
		\draw [in=-105, out=-90, looseness=2.75] (0.center) to (3.center);
		\draw (5.center) to (4.center);
		\draw [bend right=90, looseness=3.50] (7.center) to (8.center);
		\draw [in=-75, out=-90, looseness=2.75] (6.center) to (9.center);
		\draw (11.center) to (10.center);
	\end{pgfonlayer}
\end{tikzpicture}
\end{equation}
\end{center}

\noindent where $a,\ b,\ c,\ a',\text{ and } b' \in \ZZ/n\ZZ$, and $a+b\equiv a'+b' \mod n$.
\end{definition}

\begin{remark} It is worth noting that the when using the split map on an integer mod $n$, we must specify which two integers are the target.  For example, 

\begin{equation}\begin{tikzpicture}[scale=.5, baseline=18]
	\begin{pgfonlayer}{nodelayer}
		\node   (6) at (-1.5, 2) {};
		\node   (7) at (-0.5, 2) {};
		\node   (8) at (-1, 1) {};
		\node   (9) at (-1, 0) {};
		\node   (15) at (-1, -0.5) {$c$};
		\node   (16) at (-1.5, 2.5) {$a$};
		\node   (17) at (-0.5, 2.5) {$b$};
		\node   (19) at (0.5, 1) { and};
		\node   (20) at (1.5, 2) {};
		\node   (21) at (2.5, 2) {};
		\node   (22) at (2, 1) {};
		\node   (23) at (2, 0) {};
		\node   (24) at (2, -0.5) {$c$};
		\node   (25) at (1.5, 2.5) {$a'$};
		\node   (26) at (2.5, 2.5) {$b'$};
	\end{pgfonlayer}
	\begin{pgfonlayer}{edgelayer}
		\draw (8.center) to (9.center);
		\draw [bend right=90, looseness=3.50] (6.center) to (7.center);
		\draw (22.center) to (23.center);
		\draw [bend right=90, looseness=3.50] (20.center) to (21.center);
	\end{pgfonlayer}
\end{tikzpicture}
\end{equation} are equal if and only if $a'\equiv a \mod n$ and $b'\equiv b \mod n$.  
\end{remark}

Given a diagram $d$, we will denote $s_d$ as the number of split diagrams used in the construction of $d$ and $m_d$ as the number of merge diagrams used in $d$.

\begin{lemma}
\label{lem:TensorDiff}
The difference $s_d-m_d$ is precisely the difference between the number of tensor factors in the target and the number of tensor factors in the source.
\end{lemma}

\begin{proof}
Fix a diagram $d\in \operatorname{Hom}_{\mathcal{C}^{\text{irr}}}\left(\bigotimes\limits_{i=1}^{k} a_i,\bigotimes\limits_{j=1}^{\ell} b_j\right)$ with $k,\ell\in\NN$.  Notice, there are $\ell$ tensor factors corresponding to $\ell$ strings on the bottom of $d$.  Reading the $d$ from bottom to top, observe that a merge diagram will subtract one from the number of tensor factors of the top of $d$, and a split diagram will add one to the number of tensor factors of the top of the $d$.  Furthermore, an identity strand will not change the number of tensor factors.  Thus, $s_d-m_d=k-\ell$.
\end{proof}

\begin{lemma}
\label{lem:AnyDiagIsIdentity}
Any diagram in $\operatorname{Hom}_{\mathcal{C}^{\text{irr}}}\left(a,a\right)$ is equal to \begin{tikzpicture}[scale=.5, baseline=0]
	\begin{pgfonlayer}{nodelayer}
		\node  (0) at (0, 1) {};
		\node  (1) at (0, -1) {};
		\node  (2) at (0,-1.5) {$a$};
		\node  (3) at (0,1.5) {$a$};
	\end{pgfonlayer}
	\begin{pgfonlayer}{edgelayer}
		\draw (0.center) to (1.center);
	\end{pgfonlayer}
\end{tikzpicture} as morphisms in $\mathcal{C}_n^{\text{irr}}$.
\end{lemma} 

\begin{proof}
Let 
\begin{equation}\begin{tikzpicture}[scale=.5, baseline=0]
	\begin{pgfonlayer}{nodelayer}
		\node   (0) at (-1, 2) {};
		\node   (1) at (-1, 1.25) {};
		\node   (2) at (-2, 1) {};
		\node   (3) at (0, 1) {};
		\node   (4) at (0, 0) {};
		\node   (5) at (-2, 0) {};
		\node   (6) at (-1, -1) {};
		\node   (7) at (-1, -0.25) {};
		\node   (8) at (-1, 0.5) {$d$};
		\node   (9) at (-1, 2.5) {$a$};
		\node   (10) at (-1, -1.5) {$a$};
	\end{pgfonlayer}
	\begin{pgfonlayer}{edgelayer}
		\draw (0.center) to (1.center);
		\draw (2.center) to (5.center);
		\draw (3.center) to (4.center);
		\draw (4.center) to (5.center);
		\draw (2.center) to (3.center);
		\draw (7.center) to (6.center);
	\end{pgfonlayer}
\end{tikzpicture}
\end{equation} 
be a diagram in Hom$_{\mathcal{C}^{\text{irr}}}\left(a,a\right)$.  From Lemma \ref{lem:TensorDiff}, the number of merge diagrams in $d$ is equal to the number of split diagrams in $d$.  So, we induct on the number of split diagrams in $d$, $s_d$.  If $s_d=0$, then $m_d=0$ as well, and $d$ must be the identity strand on $a$.

Now, assume we have that any diagram $d'$ is equal to the identity strand on $a$ for $s_{d'}<k$ for some $k$.  Suppose $s_{d'}=k$.  Then, we can isolate a highest split diagram in $d'$.  Thus we have 

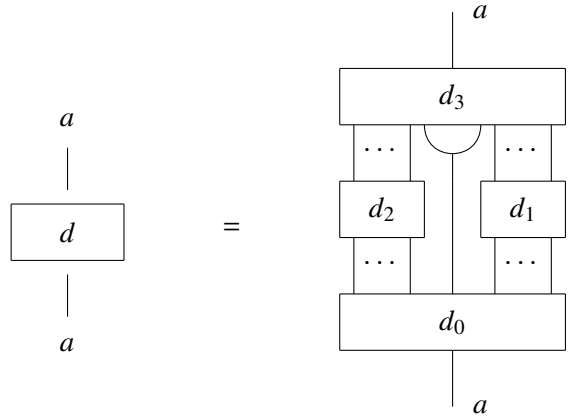
\begin{center}
\begin{tikzpicture}[scale=.75, baseline=10]
	\begin{pgfonlayer}{nodelayer}
		\node   (0) at (-1, 2) {};
		\node   (1) at (-1, 1.25) {};
		\node   (2) at (-2, 1) {};
		\node   (3) at (0, 1) {};
		\node   (4) at (0, 0) {};
		\node   (5) at (-2, 0) {};
		\node   (6) at (-1, -1) {};
		\node   (7) at (-1, -0.25) {};
		\node   (8) at (-1, 0.5) {$d$};
		\node   (9) at (-1, 2.5) {$a$};
		\node   (10) at (-1, -1.5) {$a$};
	\end{pgfonlayer}
	\begin{pgfonlayer}{edgelayer}
		\draw (0.center) to (1.center);
		\draw (2.center) to (5.center);
		\draw (3.center) to (4.center);
		\draw (4.center) to (5.center);
		\draw (2.center) to (3.center);
		\draw (7.center) to (6.center);
	\end{pgfonlayer}
\end{tikzpicture} \hspace{10mm} $=$ \hspace{10mm}
\begin{tikzpicture}[scale=.75, baseline=-20]
	\begin{pgfonlayer}{nodelayer}
		\node   (0) at (0, 3) {};
		\node   (1) at (0, 2) {};
		\node   (2) at (-2, 2) {};
		\node   (3) at (-2, 1) {};
		\node   (4) at (2, 1) {};
		\node   (5) at (2, 2) {};
		\node   (6) at (-0.5, 1) {};
		\node   (7) at (0.5, 1) {};
		\node   (8) at (0.75, 1) {};
		\node   (9) at (1.75, 1) {};
		\node   (10) at (-1.75, 1) {};
		\node   (11) at (-0.75, 1) {};
		\node   (12) at (-1.75, 0) {};
		\node   (13) at (-0.75, 0) {};
		\node   (14) at (0, 0.5) {};
		\node   (15) at (0, 0) {};
		\node   (16) at (0.75, 0) {};
		\node   (17) at (1.75, 0) {};
		\node   (18) at (-2, 0) {};
		\node   (19) at (-0.5, 0) {};
		\node   (20) at (-2, -1) {};
		\node   (21) at (-0.5, -1) {};
		\node   (22) at (0, -2) {};
		\node   (23) at (0.5, -1) {};
		\node   (24) at (2, -1) {};
		\node   (25) at (0.5, 0) {};
		\node   (26) at (2, 0) {};
		\node   (27) at (-1.75, -1) {};
		\node   (28) at (-0.75, -1) {};
		\node   (29) at (0.75, -1) {};
		\node   (30) at (1.75, -1) {};
		\node   (31) at (-1.75, -2) {};
		\node   (32) at (-0.75, -2) {};
		\node   (33) at (0.75, -2) {};
		\node   (34) at (1.75, -2) {};
		\node   (35) at (-2, -2) {};
		\node   (36) at (2, -2) {};
		\node   (37) at (2, -3) {};
		\node   (38) at (-2, -3) {};
		\node   (39) at (0, -3) {};
		\node   (40) at (0, -4) {};
		\node   (41) at (0, 1.5) {$d_3$};
		\node   (42) at (-1.25, -0.5) {$d_2$};
		\node   (43) at (1.25, -0.5) {$d_1$};
		\node   (44) at (0, -2.5) {$d_0$};
		\node   (45) at (0.5, -4) {$a$};
		\node   (46) at (0.5, 3) {$a$};
		\node   (47) at (-1.25, 0.5) {$\cdots$};
		\node   (48) at (1.25, 0.5) {$\cdots$};
		\node   (49) at (1.25, -1.5) {$\cdots$};
		\node   (50) at (-1.25, -1.5) {$\cdots$};
	\end{pgfonlayer}
	\begin{pgfonlayer}{edgelayer}
		\draw (0.center) to (1.center);
		\draw (2.center) to (5.center);
		\draw (5.center) to (4.center);
		\draw (3.center) to (2.center);
		\draw (10.center) to (12.center);
		\draw (11.center) to (13.center);
		\draw [bend right=90, looseness=1.75] (6.center) to (7.center);
		\draw (3.center) to (4.center);
		\draw (15.center) to (14.center);
		\draw (8.center) to (16.center);
		\draw (9.center) to (17.center);
		\draw (18.center) to (20.center);
		\draw (20.center) to (21.center);
		\draw (21.center) to (19.center);
		\draw (19.center) to (18.center);
		\draw (25.center) to (23.center);
		\draw (23.center) to (24.center);
		\draw (24.center) to (26.center);
		\draw (26.center) to (25.center);
		\draw (27.center) to (31.center);
		\draw (28.center) to (32.center);
		\draw (29.center) to (33.center);
		\draw (30.center) to (34.center);
		\draw (36.center) to (35.center);
		\draw (22.center) to (15.center);
		\draw (35.center) to (38.center);
		\draw (38.center) to (37.center);
		\draw (36.center) to (37.center);
		\draw (39.center) to (40.center);
	\end{pgfonlayer}
\end{tikzpicture}
\end{center}

\noindent where $d_3$ only contains identity strands and merge diagrams.  In particular, since $d_3$ has at least two tensor factors in the domain and only one tensor factor in the codomain, by Lemma \ref{lem:TensorDiff}, $d_3$ must contain at least one merge diagram.  Using the associativity relation for merge diagrams in (\ref{eqn:CnnRelation3}) iteratively, we can position a merge diagram directly above the split diagram.  Hence,

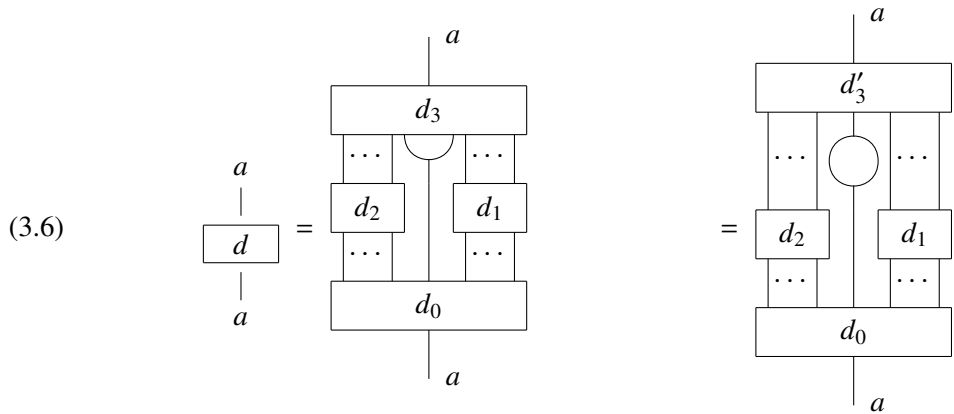
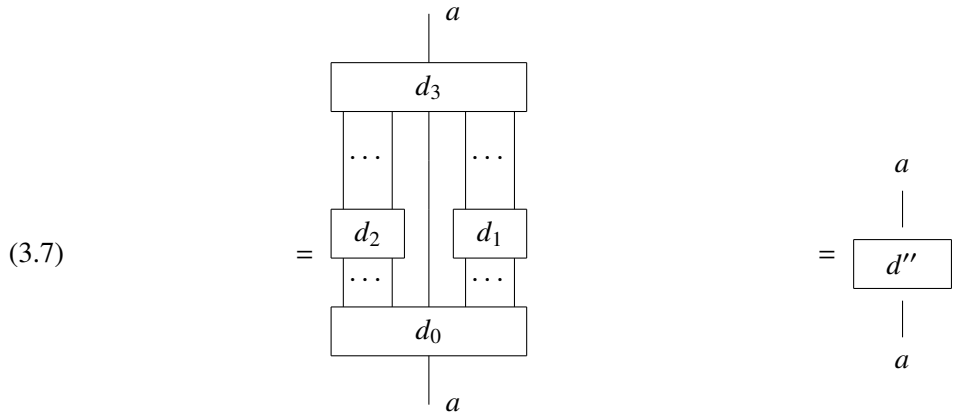
\begin{align}
\begin{tikzpicture}[scale=.5, baseline=10]
	\begin{pgfonlayer}{nodelayer}
		\node   (0) at (-1, 2) {};
		\node   (1) at (-1, 1.25) {};
		\node   (2) at (-2, 1) {};
		\node   (3) at (0, 1) {};
		\node   (4) at (0, 0) {};
		\node   (5) at (-2, 0) {};
		\node   (6) at (-1, -1) {};
		\node   (7) at (-1, -0.25) {};
		\node   (8) at (-1, 0.5) {$d$};
		\node   (9) at (-1, 2.5) {$a$};
		\node   (10) at (-1, -1.5) {$a$};
	\end{pgfonlayer}
	\begin{pgfonlayer}{edgelayer}
		\draw (0.center) to (1.center);
		\draw (2.center) to (5.center);
		\draw (3.center) to (4.center);
		\draw (4.center) to (5.center);
		\draw (2.center) to (3.center);
		\draw (7.center) to (6.center);
	\end{pgfonlayer}
\end{tikzpicture} &=
\begin{tikzpicture}[scale=.65, baseline=-20]
	\begin{pgfonlayer}{nodelayer}
		\node   (0) at (0, 3) {};
		\node   (1) at (0, 2) {};
		\node   (2) at (-2, 2) {};
		\node   (3) at (-2, 1) {};
		\node   (4) at (2, 1) {};
		\node   (5) at (2, 2) {};
		\node   (6) at (-0.5, 1) {};
		\node   (7) at (0.5, 1) {};
		\node   (8) at (0.75, 1) {};
		\node   (9) at (1.75, 1) {};
		\node   (10) at (-1.75, 1) {};
		\node   (11) at (-0.75, 1) {};
		\node   (12) at (-1.75, 0) {};
		\node   (13) at (-0.75, 0) {};
		\node   (14) at (0, 0.5) {};
		\node   (15) at (0, 0) {};
		\node   (16) at (0.75, 0) {};
		\node   (17) at (1.75, 0) {};
		\node   (18) at (-2, 0) {};
		\node   (19) at (-0.5, 0) {};
		\node   (20) at (-2, -1) {};
		\node   (21) at (-0.5, -1) {};
		\node   (22) at (0, -2) {};
		\node   (23) at (0.5, -1) {};
		\node   (24) at (2, -1) {};
		\node   (25) at (0.5, 0) {};
		\node   (26) at (2, 0) {};
		\node   (27) at (-1.75, -1) {};
		\node   (28) at (-0.75, -1) {};
		\node   (29) at (0.75, -1) {};
		\node   (30) at (1.75, -1) {};
		\node   (31) at (-1.75, -2) {};
		\node   (32) at (-0.75, -2) {};
		\node   (33) at (0.75, -2) {};
		\node   (34) at (1.75, -2) {};
		\node   (35) at (-2, -2) {};
		\node   (36) at (2, -2) {};
		\node   (37) at (2, -3) {};
		\node   (38) at (-2, -3) {};
		\node   (39) at (0, -3) {};
		\node   (40) at (0, -4) {};
		\node   (41) at (0, 1.5) {$d_3$};
		\node   (42) at (-1.25, -0.5) {$d_2$};
		\node   (43) at (1.25, -0.5) {$d_1$};
		\node   (44) at (0, -2.5) {$d_0$};
		\node   (45) at (0.5, -4) {$a$};
		\node   (46) at (0.5, 3) {$a$};
		\node   (47) at (-1.25, 0.5) {$\cdots$};
		\node   (48) at (1.25, 0.5) {$\cdots$};
		\node   (49) at (1.25, -1.5) {$\cdots$};
		\node   (50) at (-1.25, -1.5) {$\cdots$};
	\end{pgfonlayer}
	\begin{pgfonlayer}{edgelayer}
		\draw (0.center) to (1.center);
		\draw (2.center) to (5.center);
		\draw (5.center) to (4.center);
		\draw (3.center) to (2.center);
		\draw (10.center) to (12.center);
		\draw (11.center) to (13.center);
		\draw [bend right=90, looseness=1.75] (6.center) to (7.center);
		\draw (3.center) to (4.center);
		\draw (15.center) to (14.center);
		\draw (8.center) to (16.center);
		\draw (9.center) to (17.center);
		\draw (18.center) to (20.center);
		\draw (20.center) to (21.center);
		\draw (21.center) to (19.center);
		\draw (19.center) to (18.center);
		\draw (25.center) to (23.center);
		\draw (23.center) to (24.center);
		\draw (24.center) to (26.center);
		\draw (26.center) to (25.center);
		\draw (27.center) to (31.center);
		\draw (28.center) to (32.center);
		\draw (29.center) to (33.center);
		\draw (30.center) to (34.center);
		\draw (36.center) to (35.center);
		\draw (22.center) to (15.center);
		\draw (35.center) to (38.center);
		\draw (38.center) to (37.center);
		\draw (36.center) to  (37.center);
		\draw (39.center) to (40.center);
	\end{pgfonlayer}
\end{tikzpicture} &=\begin{tikzpicture}[scale=.65, baseline=-10]
	\begin{pgfonlayer}{nodelayer}
		\node   (0) at (0, 4) {};
		\node   (1) at (0, 3) {};
		\node   (2) at (-2, 3) {};
		\node   (3) at (-2, 2) {};
		\node   (4) at (2, 2) {};
		\node   (5) at (2, 3) {};
		\node   (6) at (-0.5, 1) {};
		\node   (7) at (0.5, 1) {};
		\node   (8) at (0.75, 2) {};
		\node   (9) at (1.75, 2) {};
		\node   (10) at (-1.75, 2) {};
		\node   (11) at (-0.75, 2) {};
		\node   (12) at (-1.75, 0) {};
		\node   (13) at (-0.75, 0) {};
		\node   (14) at (0, 0.5) {};
		\node   (15) at (0, 0) {};
		\node   (16) at (0.75, 0) {};
		\node   (17) at (1.75, 0) {};
		\node   (18) at (-2, 0) {};
		\node   (19) at (-0.5, 0) {};
		\node   (20) at (-2, -1) {};
		\node   (21) at (-0.5, -1) {};
		\node   (22) at (0, -2) {};
		\node   (23) at (0.5, -1) {};
		\node   (24) at (2, -1) {};
		\node   (25) at (0.5, 0) {};
		\node   (26) at (2, 0) {};
		\node   (27) at (-1.75, -1) {};
		\node   (28) at (-0.75, -1) {};
		\node   (29) at (0.75, -1) {};
		\node   (30) at (1.75, -1) {};
		\node   (31) at (-1.75, -2) {};
		\node   (32) at (-0.75, -2) {};
		\node   (33) at (0.75, -2) {};
		\node   (34) at (1.75, -2) {};
		\node   (35) at (-2, -2) {};
		\node   (36) at (2, -2) {};
		\node   (37) at (2, -3) {};
		\node   (38) at (-2, -3) {};
		\node   (39) at (0, -3) {};
		\node   (40) at (0, -4) {};
		\node   (41) at (0, 2.5) {$d_3'$};
		\node   (42) at (-1.25, -0.5) {$d_2$};
		\node   (43) at (1.25, -0.5) {$d_1$};
		\node   (44) at (0, -2.5) {$d_0$};
		\node   (45) at (0.5, -4) {$a$};
		\node   (46) at (0.5, 4) {$a$};
		\node   (47) at (-1.25, 1) {$\cdots$};
		\node   (48) at (1.25, 1) {$\cdots$};
		\node   (49) at (1.25, -1.5) {$\cdots$};
		\node   (50) at (-1.25, -1.5) {$\cdots$};
		\node   (51) at (0, 2) {};
		\node   (52) at (0, 1.5) {};
	\end{pgfonlayer}
	\begin{pgfonlayer}{edgelayer}
		\draw (0.center) to (1.center);
		\draw (2.center) to (5.center);
		\draw (5.center) to (4.center);
		\draw (3.center) to (2.center);
		\draw (10.center) to (12.center);
		\draw (11.center) to (13.center);
		\draw [bend right=90, looseness=1.75] (6.center) to (7.center);
		\draw (3.center) to (4.center);
		\draw (15.center) to (14.center);
		\draw (8.center) to (16.center);
		\draw (9.center) to (17.center);
		\draw (18.center) to (20.center);
		\draw (20.center) to (21.center);
		\draw (21.center) to (19.center);
		\draw (19.center) to (18.center);
		\draw (25.center) to (23.center);
		\draw (23.center) to (24.center);
		\draw (24.center) to (26.center);
		\draw (26.center) to (25.center);
		\draw (27.center) to (31.center);
		\draw (28.center) to (32.center);
		\draw (29.center) to (33.center);
		\draw (30.center) to (34.center);
		\draw (36.center) to (35.center);
		\draw (22.center) to (15.center);
		\draw (35.center) to (38.center);
		\draw (38.center) to (37.center);
		\draw (36.center) to (37.center);
		\draw (39.center) to (40.center);
		\draw (51.center) to (52.center);
		\draw [bend left=90, looseness=1.75] (6.center) to (7.center);
	\end{pgfonlayer}
\end{tikzpicture} \\
&=
\begin{tikzpicture}[scale=.65, baseline=-20]
	\begin{pgfonlayer}{nodelayer}
		\node   (0) at (0, 4) {};
		\node   (1) at (0, 3) {};
		\node   (2) at (-2, 3) {};
		\node   (3) at (-2, 2) {};
		\node   (4) at (2, 2) {};
		\node   (5) at (2, 3) {};
		\node   (8) at (0.75, 2) {};
		\node   (9) at (1.75, 2) {};
		\node   (10) at (-1.75, 2) {};
		\node   (11) at (-0.75, 2) {};
		\node   (12) at (-1.75, 0) {};
		\node   (13) at (-0.75, 0) {};
		\node   (14) at (0, 1) {};
		\node   (15) at (0, 0) {};
		\node   (16) at (0.75, 0) {};
		\node   (17) at (1.75, 0) {};
		\node   (18) at (-2, 0) {};
		\node   (19) at (-0.5, 0) {};
		\node   (20) at (-2, -1) {};
		\node   (21) at (-0.5, -1) {};
		\node   (22) at (0, -2) {};
		\node   (23) at (0.5, -1) {};
		\node   (24) at (2, -1) {};
		\node   (25) at (0.5, 0) {};
		\node   (26) at (2, 0) {};
		\node   (27) at (-1.75, -1) {};
		\node   (28) at (-0.75, -1) {};
		\node   (29) at (0.75, -1) {};
		\node   (30) at (1.75, -1) {};
		\node   (31) at (-1.75, -2) {};
		\node   (32) at (-0.75, -2) {};
		\node   (33) at (0.75, -2) {};
		\node   (34) at (1.75, -2) {};
		\node   (35) at (-2, -2) {};
		\node   (36) at (2, -2) {};
		\node   (37) at (2, -3) {};
		\node   (38) at (-2, -3) {};
		\node   (39) at (0, -3) {};
		\node   (40) at (0, -4) {};
		\node   (41) at (0, 2.5) {$d_3$};
		\node   (42) at (-1.25, -0.5) {$d_2$};
		\node   (43) at (1.25, -0.5) {$d_1$};
		\node   (44) at (0, -2.5) {$d_0$};
		\node   (45) at (0.5, -4) {$a$};
		\node   (46) at (0.5, 4) {$a$};
		\node   (47) at (-1.25, 1) {$\cdots$};
		\node   (48) at (1.25, 1) {$\cdots$};
		\node   (49) at (1.25, -1.5) {$\cdots$};
		\node   (50) at (-1.25, -1.5) {$\cdots$};
		\node   (51) at (0, 2) {};
		\node   (52) at (0, 1) {};
	\end{pgfonlayer}
	\begin{pgfonlayer}{edgelayer}
		\draw (0.center) to (1.center);
		\draw (2.center) to (5.center);
		\draw (5.center) to (4.center);
		\draw (3.center) to (2.center);
		\draw (10.center) to (12.center);
		\draw (11.center) to (13.center);
		\draw (3.center) to (4.center);
		\draw (15.center) to (14.center);
		\draw (8.center) to (16.center);
		\draw (9.center) to (17.center);
		\draw (18.center) to (20.center);
		\draw (20.center) to (21.center);
		\draw (21.center) to (19.center);
		\draw (19.center) to (18.center);
		\draw (25.center) to (23.center);
		\draw (23.center) to (24.center);
		\draw (24.center) to (26.center);
		\draw (26.center) to (25.center);
		\draw (27.center) to (31.center);
		\draw (28.center) to (32.center);
		\draw (29.center) to (33.center);
		\draw (30.center) to (34.center);
		\draw (36.center) to (35.center);
		\draw (22.center) to (15.center);
		\draw (35.center) to (38.center);
		\draw (38.center) to (37.center);
		\draw (36.center) to (37.center);
		\draw (39.center) to (40.center);
		\draw (51.center) to (52.center);
	\end{pgfonlayer}
\end{tikzpicture} &=
\begin{tikzpicture}[scale=.65, baseline=10]
	\begin{pgfonlayer}{nodelayer}
		\node   (0) at (-1, 2) {};
		\node   (1) at (-1, 1.25) {};
		\node   (2) at (-2, 1) {};
		\node   (3) at (0, 1) {};
		\node   (4) at (0, 0) {};
		\node   (5) at (-2, 0) {};
		\node   (6) at (-1, -1) {};
		\node   (7) at (-1, -0.25) {};
		\node   (8) at (-1, 0.5) {$d''$};
		\node   (9) at (-1, 2.5) {$a$};
		\node   (10) at (-1, -1.5) {$a$};
	\end{pgfonlayer}
	\begin{pgfonlayer}{edgelayer}
		\draw (0.center) to (1.center);
		\draw (2.center) to (5.center);
		\draw (3.center) to (4.center);
		\draw (4.center) to (5.center);
		\draw (2.center) to (3.center);
		\draw (7.center) to (6.center);
	\end{pgfonlayer}
\end{tikzpicture}
\end{align}

\noindent where $d''$ has $k-1$ splits, i.e. $s_{d''}<k$.  Therefore, using the induction hypothesis

\begin{equation}
\begin{tikzpicture}[scale=.75, baseline=10]
	\begin{pgfonlayer}{nodelayer}
		\node   (0) at (-1, 2) {};
		\node   (1) at (-1, 1.25) {};
		\node   (2) at (-2, 1) {};
		\node   (3) at (0, 1) {};
		\node   (4) at (0, 0) {};
		\node   (5) at (-2, 0) {};
		\node   (6) at (-1, -1) {};
		\node   (7) at (-1, -0.25) {};
		\node   (8) at (-1, 0.5) {$d$};
		\node   (9) at (-1, 2.5) {$a$};
		\node   (10) at (-1, -1.5) {$a$};
	\end{pgfonlayer}
	\begin{pgfonlayer}{edgelayer}
		\draw (0.center) to (1.center);
		\draw (2.center) to (5.center);
		\draw (3.center) to (4.center);
		\draw (4.center) to (5.center);
		\draw (2.center) to (3.center);
		\draw (7.center) to (6.center);
	\end{pgfonlayer}
\end{tikzpicture} \hspace{10mm} = \hspace{10mm} \begin{tikzpicture}[scale=.75, baseline=0]
	\begin{pgfonlayer}{nodelayer}
		\node  (0) at (0, 1) {};
		\node  (1) at (0, -1) {};
		\node  (2) at (0,-1.5) {$a$};
		\node  (3) at (0,1.5) {$a$};
	\end{pgfonlayer}
	\begin{pgfonlayer}{edgelayer}
		\draw (0.center) to (1.center);
	\end{pgfonlayer}
\end{tikzpicture}
\end{equation}
\end{proof}

The above lemmas will be helpful in proving Theorem \ref{thm:CnirrFaithful}.  Now let us explore a particular category from representation theory.

We denote $\mathbf{C}_n\textbf{-mod}_{\text{irr}}$ as the full $\CC$-linear monoidal subcategory of $\mathbf{C}_n\textbf{-mod}$ where the generating objects are the irreducible $\mathbf{C}_n$-modules $\mathbf{C}_n^{(a)}$ where $a\in\ZZ/n\ZZ$.  As all irreducible $\mathbf{C}_n$-modules are $1$-dimensional, and $$\operatorname{dim}(M\otimes N)=\operatorname{dim}(M)\cdot\operatorname{dim}(N),$$ then by Schur's Lemma, Hom$_{\mathbf{C}_n}(M,N)$ is either $0$-dimensional or $1$-dimensional.  If it is $1$-dimensional, then $M\cong \mathbf{C}_n^{(a)} \cong N$ for some $a\in\ZZ/n\ZZ$.  We pick bases for these irreducible $\mathbf{C}_n$-modules and let $v_a$ denote our chosen basis vector of $\mathbf{C}_n^{(a)}$.

We can choose $\mathbf{C}_n$-module homomorphisms, $$m^c_{a,b}: \mathbf{C}_n^{(a)}\otimes \mathbf{C}_n^{(b)}\longrightarrow \mathbf{C}_n^{(c)}$$ $$v_a\otimes v_b\mapsto v_c$$ where $c\equiv a+b \mod n$, $$s_c^{a,b}: \mathbf{C}_n^{(c)}\longrightarrow \mathbf{C}_n^{(a)}\otimes \mathbf{C}_n^{(b)}$$ $$v_c \mapsto v_a\otimes v_b$$ where $c\equiv a+b \mod n$, and $$id_a: \mathbf{C}_n^{(a)}\longrightarrow \mathbf{C}_n^{(a)}$$ is the identity map.  Notice that $m_{a,b}^c\circ s_c^{a,b}=id_c$ for all $a,b,c\in\ZZ/n\ZZ$. 

\begin{theorem}
There exists a well-defined functor of monoidal $\CC$-linear categories $\mathcal{F}_n^{\text{irr}}:  \mathcal{C}_n^{\text{irr}} \longrightarrow \mathbf{C}_n\textbf{-mod}_{\text{irr}}$ determined by the following rules:  $$a\mapsto \mathbf{C}_n^{(a)}$$

\begin{center}
\begin{tikzpicture}[scale=.75, baseline=18]
	\begin{pgfonlayer}{nodelayer}
		\node   (0) at (-1, 0) {};
		\node   (1) at (0, 0) {};
		\node   (2) at (-0.5, 1) {};
		\node   (3) at (-0.5, 2) {};
		\node   (4) at (-1, -0.5) {$a$};
		\node   (5) at (0, -0.5) {$b$};
		\node   (6) at (-0.5, 2.5) {$a+b$};
	\end{pgfonlayer}
	\begin{pgfonlayer}{edgelayer}
		\draw [bend left=90, looseness=3.50] (0.center) to (1.center);
		\draw (2.center) to (3.center);
	\end{pgfonlayer}
\end{tikzpicture}$\mapsto m_{a,b}$, 
\begin{tikzpicture}[scale=.75, baseline=18]
	\begin{pgfonlayer}{nodelayer}
		\node   (2) at (-0.5, 0) {};
		\node   (3) at (-0.5, 2) {};
		\node   (5) at (-0.5, -0.5) {$a$};
		\node   (6) at (-0.5, 2.5) {$a$};
	\end{pgfonlayer}
	\begin{pgfonlayer}{edgelayer}
		\draw (2.center) to (3.center);
	\end{pgfonlayer}
\end{tikzpicture}$\mapsto id_a$, 
\begin{tikzpicture}[scale=.75, baseline=18]
	\begin{pgfonlayer}{nodelayer}
		\node   (2) at (-0.5, 0) {};
		\node   (3) at (-0.5, 1) {};
		\node   (4) at (-1, 2) {};
		\node   (5) at (0, 2) {};
		\node   (6) at (-1, 2.5) {$a$};
		\node   (7) at (0, 2.5) {$b$};
		\node   (8) at (-0.5, -0.5) {$c$};
	\end{pgfonlayer}
	\begin{pgfonlayer}{edgelayer}
		\draw (2.center) to (3.center);
		\draw [bend right=90, looseness=3.50] (4.center) to (5.center);
	\end{pgfonlayer}
\end{tikzpicture}$\mapsto s_{c}^{a,b}$

\end{center}

\noindent for each $a,b,c\in \ZZ/n\ZZ$.
\end{theorem}

\begin{proof}
We check that the above relations are preserved by the functor $\mathcal{F}_n^{\text{irr}}$.

As \begin{align*}
\mathcal{F}_n^{\text{irr}}\left(\begin{tikzpicture}[scale=.45, baseline=12]
	\begin{pgfonlayer}{nodelayer}
		\node   (0) at (-3, -1) {};
		\node   (1) at (-3, 1) {};
		\node   (2) at (-2, 1) {};
		\node   (3) at (-1, 1) {};
		\node   (4) at (-1, 3) {};
		\node   (5) at (-1.5, 0) {};
		\node   (6) at (-1.5, -1) {};
		\node   (7) at (-2.5, 2) {};
		\node   (8) at (-2.5, 3) {};
		\node   (15) at (-2.5, 3.5) {$a$};
		\node   (16) at (-1, 3.5) {$b$};
		\node   (17) at (-3, -1.5) {$a'$};
		\node   (18) at (-1.5, -1.5) {$b'$};
	\end{pgfonlayer}
	\begin{pgfonlayer}{edgelayer}
		\draw (0.center) to (1.center);
		\draw [bend left=90, looseness=3.50] (1.center) to (2.center);
		\draw [bend right=90, looseness=3.50] (2.center) to (3.center);
		\draw (5.center) to (6.center);
		\draw (7.center) to (8.center);
		\draw (4.center) to (3.center);
	\end{pgfonlayer}
\end{tikzpicture}\right)(v_{a'}\otimes v_{b'})&=(m_{a',c}^a\otimes id_b \circ id_{a'} \otimes s_{b'}^{c,b}) (v_{a'}\otimes v_{b'})\\
&=(m_{a',c}^a\otimes id_b) (v_{a'} \otimes v_{c} \otimes v_{b})=v_a \otimes v_b;\\
\mathcal{F}_n^{\text{irr}}\left(\begin{tikzpicture}[scale=.45, baseline=5]
	\begin{pgfonlayer}{nodelayer}
		\node   (9) at (-2, -1) {};
		\node   (10) at (-1, -1) {};
		\node   (11) at (-1.5, 0) {};
		\node   (12) at (-1.5, 1) {};
		\node   (13) at (-1, 2) {};
		\node   (14) at (-2, 2) {};
		\node   (19) at (-2, -1.5) {$a'$};
		\node   (20) at (-1, -1.5) {$b'$};
		\node   (21) at (-2, 2.5) {$a$};
		\node   (22) at (-1, 2.5) {$b$};
	\end{pgfonlayer}
	\begin{pgfonlayer}{edgelayer}
		\draw [bend right=90, looseness=3.50] (14.center) to (13.center);
		\draw (12.center) to (11.center);
		\draw [bend left=90, looseness=3.50] (9.center) to (10.center);
	\end{pgfonlayer}
\end{tikzpicture}\right)(v_{a'}\otimes v_{b'})&=s_{c}^{a,b}\circ m_{a',b'}^{c}(v_{a'}\otimes v_{b'})=s_{c}^{a,b}(v_c)=v_a\otimes v_b;\\
\mathcal{F}_n^{\text{irr}}\left(\begin{tikzpicture}[scale=.45, baseline=12]
	\begin{pgfonlayer}{nodelayer}
		\node   (0) at (-1, -1) {};
		\node   (1) at (-1, 1) {};
		\node   (2) at (-2, 1) {};
		\node   (3) at (-3, 1) {};
		\node   (4) at (-3, 3) {};
		\node   (5) at (-2.5, 0) {};
		\node   (6) at (-2.5, -1) {};
		\node   (7) at (-1.5, 2) {};
		\node   (8) at (-1.5, 3) {};
		\node   (15) at (-3, 3.5) {$a$};
		\node   (16) at (-1.5, 3.5) {$b$};
		\node   (17) at (-1, -1.5) {$b'$};
		\node   (18) at (-2.5, -1.5) {$a'$};
	\end{pgfonlayer}
	\begin{pgfonlayer}{edgelayer}
		\draw (0.center) to (1.center);
		\draw [bend right=90, looseness=3.50] (1.center) to (2.center);
		\draw [bend left=90, looseness=3.50] (2.center) to (3.center);
		\draw (5.center) to (6.center);
		\draw (7.center) to (8.center);
		\draw (4.center) to (3.center);
	\end{pgfonlayer}
\end{tikzpicture}\right)(v_{a'}\otimes v_{b'})&=(id_a\otimes m_{c,b'}^a \circ s_{a'}^{a,c}\otimes id_{b'}) (v_{a'}\otimes v_{b'})\\
&=(id_a\otimes m_{c,b'}^a) (v_{a} \otimes v_{c} \otimes v_{b'})=v_a \otimes v_b.
\end{align*}

Thus, \begin{align}
\mathcal{F}_n^{\text{irr}}\left(\begin{tikzpicture}[scale=.45, baseline=12]
	\begin{pgfonlayer}{nodelayer}
		\node   (0) at (-3, -1) {};
		\node   (1) at (-3, 1) {};
		\node   (2) at (-2, 1) {};
		\node   (3) at (-1, 1) {};
		\node   (4) at (-1, 3) {};
		\node   (5) at (-1.5, 0) {};
		\node   (6) at (-1.5, -1) {};
		\node   (7) at (-2.5, 2) {};
		\node   (8) at (-2.5, 3) {};
		\node   (15) at (-2.5, 3.5) {$a$};
		\node   (16) at (-1, 3.5) {$b$};
		\node   (17) at (-3, -1.5) {$a'$};
		\node   (18) at (-1.5, -1.5) {$b'$};
	\end{pgfonlayer}
	\begin{pgfonlayer}{edgelayer}
		\draw (0.center) to (1.center);
		\draw [bend left=90, looseness=3.50] (1.center) to (2.center);
		\draw [bend right=90, looseness=3.50] (2.center) to (3.center);
		\draw (5.center) to (6.center);
		\draw (7.center) to (8.center);
		\draw (4.center) to (3.center);
	\end{pgfonlayer}
\end{tikzpicture}\right)=\mathcal{F}_n^{\text{irr}}\left(\begin{tikzpicture}[scale=.45, baseline=10]
	\begin{pgfonlayer}{nodelayer}
		\node   (9) at (-2, -1) {};
		\node   (10) at (-1, -1) {};
		\node   (11) at (-1.5, 0) {};
		\node   (12) at (-1.5, 1) {};
		\node   (13) at (-1, 2) {};
		\node   (14) at (-2, 2) {};
		\node   (19) at (-2, -1.5) {$a'$};
		\node   (20) at (-1, -1.5) {$b'$};
		\node   (21) at (-2, 2.5) {$a$};
		\node   (22) at (-1, 2.5) {$b$};
	\end{pgfonlayer}
	\begin{pgfonlayer}{edgelayer}
		\draw [bend right=90, looseness=3.50] (14.center) to (13.center);
		\draw (12.center) to (11.center);
		\draw [bend left=90, looseness=3.50] (9.center) to (10.center);
	\end{pgfonlayer}
\end{tikzpicture}\right)&=\mathcal{F}_n^{\text{irr}}\left(\begin{tikzpicture}[scale=.45, baseline=12]
	\begin{pgfonlayer}{nodelayer}
		\node   (0) at (-1, -1) {};
		\node   (1) at (-1, 1) {};
		\node   (2) at (-2, 1) {};
		\node   (3) at (-3, 1) {};
		\node   (4) at (-3, 3) {};
		\node   (5) at (-2.5, 0) {};
		\node   (6) at (-2.5, -1) {};
		\node   (7) at (-1.5, 2) {};
		\node   (8) at (-1.5, 3) {};
		\node   (15) at (-3, 3.5) {$a$};
		\node   (16) at (-1.5, 3.5) {$b$};
		\node   (17) at (-1, -1.5) {$b'$};
		\node   (18) at (-2.5, -1.5) {$a'$};
	\end{pgfonlayer}
	\begin{pgfonlayer}{edgelayer}
		\draw (0.center) to (1.center);
		\draw [bend right=90, looseness=3.50] (1.center) to (2.center);
		\draw [bend left=90, looseness=3.50] (2.center) to (3.center);
		\draw (5.center) to (6.center);
		\draw (7.center) to (8.center);
		\draw (4.center) to (3.center);
	\end{pgfonlayer}
\end{tikzpicture}\right).
\end{align}

As

\begin{align*}
\mathcal{F}_n^{\text{irr}}\left(\begin{tikzpicture}[scale=.45, baseline=10]
	\begin{pgfonlayer}{nodelayer}
		\node   (9) at (-2, -1) {};
		\node   (10) at (-1, -1) {};
		\node   (11) at (-1.5, 0) {};
		\node   (12) at (-1.5, 1) {};
		\node   (13) at (-1, 2) {};
		\node   (14) at (-2, 2) {};
		\node   (19) at (-2, -1.5) {$a$};
		\node   (20) at (-1, -1.5) {$b$};
		\node   (21) at (-2, 2.5) {$a$};
		\node   (22) at (-1, 2.5) {$b$};
	\end{pgfonlayer}
	\begin{pgfonlayer}{edgelayer}
		\draw [bend right=90, looseness=3.50] (14.center) to (13.center);
		\draw (12.center) to (11.center);
		\draw [bend left=90, looseness=3.50] (9.center) to (10.center);
	\end{pgfonlayer}
\end{tikzpicture}\right)(v_{a}\otimes v_{b})&=s_{c}^{a,b}\circ m_{a,b}^{c}(v_{a}\otimes v_{b})=s_{c}^{a,b}(v_c)=v_a\otimes v_b;\\
\mathcal{F}_n^{\text{irr}}\left(\begin{tikzpicture}[scale=.5, baseline=10]
	\begin{pgfonlayer}{nodelayer}
		\node   (24) at (1, 2) {};
		\node   (25) at (2, 2) {};
		\node   (26) at (1, -1) {};
		\node   (27) at (2, -1) {};
		\node   (28) at (1, 2.5) {$a$};
		\node   (29) at (2, 2.5) {$b$};
		\node   (30) at (1, -1.5) {$a$};
		\node   (31) at (2, -1.5) {$b$};
	\end{pgfonlayer}
	\begin{pgfonlayer}{edgelayer}
		\draw (24.center) to (26.center);
		\draw (25.center) to (27.center);
	\end{pgfonlayer}
\end{tikzpicture}\right)(v_{a}\otimes v_{b})&=id_a \otimes id_b (v_{a}\otimes v_{b})=v_a\otimes v_b.
\end{align*}

\noindent Thus, 
\begin{align}
\mathcal{F}_n^{\text{irr}}\left(\begin{tikzpicture}[scale=.5, baseline=10]
	\begin{pgfonlayer}{nodelayer}
		\node   (9) at (-2, -1) {};
		\node   (10) at (-1, -1) {};
		\node   (11) at (-1.5, 0) {};
		\node   (12) at (-1.5, 1) {};
		\node   (13) at (-1, 2) {};
		\node   (14) at (-2, 2) {};
		\node   (19) at (-2, -1.5) {$a$};
		\node   (20) at (-1, -1.5) {$b$};
		\node   (21) at (-2, 2.5) {$a$};
		\node   (22) at (-1, 2.5) {$b$};
	\end{pgfonlayer}
	\begin{pgfonlayer}{edgelayer}
		\draw [bend right=90, looseness=3.50] (14.center) to (13.center);
		\draw (12.center) to (11.center);
		\draw [bend left=90, looseness=3.50] (9.center) to (10.center);
	\end{pgfonlayer}
\end{tikzpicture}\right)=\mathcal{F}_n^{\text{irr}}\left(\begin{tikzpicture}[scale=.5, baseline=10]
	\begin{pgfonlayer}{nodelayer}
		\node   (24) at (1, 2) {};
		\node   (25) at (2, 2) {};
		\node   (26) at (1, -1) {};
		\node   (27) at (2, -1) {};
		\node   (28) at (1, 2.5) {$a$};
		\node   (29) at (2, 2.5) {$b$};
		\node   (30) at (1, -1.5) {$a$};
		\node   (31) at (2, -1.5) {$b$};
	\end{pgfonlayer}
	\begin{pgfonlayer}{edgelayer}
		\draw (24.center) to (26.center);
		\draw (25.center) to (27.center);
	\end{pgfonlayer}
\end{tikzpicture}\right).
\end{align}

\begin{align*}
\mathcal{F}_n^{\text{irr}}\left(\begin{tikzpicture}[scale=.5, baseline=18]
	\begin{pgfonlayer}{nodelayer}
		\node   (0) at (-4, 0) {};
		\node   (1) at (-3, 0) {};
		\node   (2) at (-2, 0) {};
		\node   (3) at (-2.5, 1) {};
		\node   (4) at (-3.25, 2) {};
		\node   (5) at (-3.25, 3) {};
		\node   (12) at (-4, -0.5) {$a$};
		\node   (13) at (-3, -0.5) {$b$};
		\node   (14) at (-2, -0.5) {$c$};
		\node   (15) at (-3.25, 3.5) {$a+b+c$};
	\end{pgfonlayer}
	\begin{pgfonlayer}{edgelayer}
		\draw [bend left=90, looseness=3.50] (1.center) to (2.center);
		\draw [in=105, out=90, looseness=2.75] (0.center) to (3.center);
		\draw (5.center) to (4.center);
	\end{pgfonlayer}
\end{tikzpicture}\right)(v_{a}\otimes v_{b}\otimes v_{c})&=m_{a,b+c}^{a+b+c}\circ id_a\otimes m_{b,c}^{b+c}(v_{a}\otimes v_{b}\otimes v_{c})\\
&=m_{a,b+c}^{a+b+c}(v_a\otimes v_{b+c})=v_{a+b+c};\\
\mathcal{F}_n^{\text{irr}}\left(\begin{tikzpicture}[scale=.5, baseline=18]
	\begin{pgfonlayer}{nodelayer}
		\node   (6) at (2, 0) {};
		\node   (7) at (0, 0) {};
		\node   (8) at (1, 0) {};
		\node   (9) at (0.5, 1) {};
		\node   (10) at (1.25, 2) {};
		\node   (11) at (1.25, 3) {};
		\node   (16) at (0, -0.5) {$a$};
		\node   (17) at (1, -0.5) {$b$};
		\node   (18) at (2, -0.5) {$c$};
		\node   (19) at (1.25, 3.5) {$a+b+c$};
	\end{pgfonlayer}
	\begin{pgfonlayer}{edgelayer}
		\draw [bend left=90, looseness=3.50] (7.center) to (8.center);
		\draw [in=75, out=90, looseness=2.75] (6.center) to (9.center);
		\draw (11.center) to (10.center);
	\end{pgfonlayer}
\end{tikzpicture}\right)(v_{a}\otimes v_{b}\otimes v_{c})&=m_{a+b,c}^{a+b+c}\circ m_{a,b}^{a+b}\otimes id_c(v_{a}\otimes v_{b}\otimes v_{c})\\
&=m_{a+b,c}^{a+b+c}(v_{a+b}\otimes v_{c})=v_{a+b+c}.
\end{align*}

\noindent Thus, 
\begin{align}
    \mathcal{F}_n^{\text{irr}}\left(\begin{tikzpicture}[scale=.5, baseline=18]
	\begin{pgfonlayer}{nodelayer}
		\node   (0) at (-4, 0) {};
		\node   (1) at (-3, 0) {};
		\node   (2) at (-2, 0) {};
		\node   (3) at (-2.5, 1) {};
		\node   (4) at (-3.25, 2) {};
		\node   (5) at (-3.25, 3) {};
		\node   (12) at (-4, -0.5) {$a$};
		\node   (13) at (-3, -0.5) {$b$};
		\node   (14) at (-2, -0.5) {$c$};
		\node   (15) at (-3.25, 3.5) {$a+b+c$};
	\end{pgfonlayer}
	\begin{pgfonlayer}{edgelayer}
		\draw [bend left=90, looseness=3.50] (1.center) to (2.center);
		\draw [in=105, out=90, looseness=2.75] (0.center) to (3.center);
		\draw (5.center) to (4.center);
	\end{pgfonlayer}
\end{tikzpicture}\right)&=\mathcal{F}_n^{\text{irr}}\left(\begin{tikzpicture}[scale=.5, baseline=18]
	\begin{pgfonlayer}{nodelayer}
		\node   (6) at (2, 0) {};
		\node   (7) at (0, 0) {};
		\node   (8) at (1, 0) {};
		\node   (9) at (0.5, 1) {};
		\node   (10) at (1.25, 2) {};
		\node   (11) at (1.25, 3) {};
		\node   (16) at (0, -0.5) {$a$};
		\node   (17) at (1, -0.5) {$b$};
		\node   (18) at (2, -0.5) {$c$};
		\node   (19) at (1.25, 3.5) {$a+b+c$};
	\end{pgfonlayer}
	\begin{pgfonlayer}{edgelayer}
		\draw [bend left=90, looseness=3.50] (7.center) to (8.center);
		\draw [in=75, out=90, looseness=2.75] (6.center) to (9.center);
		\draw (11.center) to (10.center);
	\end{pgfonlayer}
\end{tikzpicture}\right).
\end{align}

As \begin{align*}
\mathcal{F}_n^{\text{irr}}\left(\begin{tikzpicture}[scale=.5, baseline=6]
	\begin{pgfonlayer}{nodelayer}
		\node   (0) at (-3, 2) {};
		\node   (1) at (-2, 2) {};
		\node   (2) at (-1, 2) {};
		\node   (3) at (-1.5, 1) {};
		\node   (4) at (-2.25, 0) {};
		\node   (5) at (-2.25, -1) {};
		\node   (12) at (-3, 2.5) {$a$};
		\node   (13) at (-2, 2.5) {$b$};
		\node   (14) at (-1, 2.5) {$c$};
		\node   (15) at (-2.25, -1.5) {$a+b+c$};
	\end{pgfonlayer}
	\begin{pgfonlayer}{edgelayer}
		\draw [bend right=90, looseness=3.50] (1.center) to (2.center);
		\draw [in=-105, out=-90, looseness=2.75] (0.center) to (3.center);
		\draw (5.center) to (4.center);
	\end{pgfonlayer}
\end{tikzpicture}\right)(v_{a+b+c})&= id_a\otimes s_{b+c}^{b,c}\circ s_{a+b+c}^{a,b+c}(v_{a+b+c})\\
&=id_a\otimes s_{b+c}^{b,c}(v_a\otimes v_{b+c})=v_{a} \otimes v_b \otimes v_{c};\\
\mathcal{F}_n^{\text{irr}}\left(\begin{tikzpicture}[scale=.5, baseline=6]
	\begin{pgfonlayer}{nodelayer}
		\node   (6) at (3, 2) {};
		\node   (7) at (1, 2) {};
		\node   (8) at (2, 2) {};
		\node   (9) at (1.5, 1) {};
		\node   (10) at (2.25, 0) {};
		\node   (11) at (2.25, -1) {};
		\node   (16) at (1, 2.5) {$a$};
		\node   (17) at (2, 2.5) {$b$};
		\node   (18) at (3, 2.5) {$c$};
		\node   (19) at (2.25, -1.5) {$a+b+c$};
	\end{pgfonlayer}
	\begin{pgfonlayer}{edgelayer}
		\draw [bend right=90, looseness=3.50] (7.center) to (8.center);
		\draw [in=-75, out=-90, looseness=2.75] (6.center) to (9.center);
		\draw (11.center) to (10.center);
	\end{pgfonlayer}
\end{tikzpicture}\right)(v_{a+b+c})&=s_{a+b}^{a,b}\otimes id_c\circ s_{a+b+c}^{a+b,c}(v_{a+b+c})\\
&=s_{a+b}^{a,b}\otimes id_c(v_{a+b}\otimes v_{c})=v_{a}\otimes v_{b}\otimes v_{c}.
\end{align*}

\noindent Thus, 
\begin{align}
\mathcal{F}_n^{\text{irr}}\left(\begin{tikzpicture}[scale=.5, baseline=6]
	\begin{pgfonlayer}{nodelayer}
		\node   (0) at (-3, 2) {};
		\node   (1) at (-2, 2) {};
		\node   (2) at (-1, 2) {};
		\node   (3) at (-1.5, 1) {};
		\node   (4) at (-2.25, 0) {};
		\node   (5) at (-2.25, -1) {};
		\node   (12) at (-3, 2.5) {$a$};
		\node   (13) at (-2, 2.5) {$b$};
		\node   (14) at (-1, 2.5) {$c$};
		\node   (15) at (-2.25, -1.5) {$a+b+c$};
	\end{pgfonlayer}
	\begin{pgfonlayer}{edgelayer}
		\draw [bend right=90, looseness=3.50] (1.center) to (2.center);
		\draw [in=-105, out=-90, looseness=2.75] (0.center) to (3.center);
		\draw (5.center) to (4.center);
	\end{pgfonlayer}
\end{tikzpicture}\right)&=\mathcal{F}_n^{\text{irr}}\left(\begin{tikzpicture}[scale=.5, baseline=6]
	\begin{pgfonlayer}{nodelayer}
		\node   (6) at (3, 2) {};
		\node   (7) at (1, 2) {};
		\node   (8) at (2, 2) {};
		\node   (9) at (1.5, 1) {};
		\node   (10) at (2.25, 0) {};
		\node   (11) at (2.25, -1) {};
		\node   (16) at (1, 2.5) {$a$};
		\node   (17) at (2, 2.5) {$b$};
		\node   (18) at (3, 2.5) {$c$};
		\node   (19) at (2.25, -1.5) {$a+b+c$};
	\end{pgfonlayer}
	\begin{pgfonlayer}{edgelayer}
		\draw [bend right=90, looseness=3.50] (7.center) to (8.center);
		\draw [in=-75, out=-90, looseness=2.75] (6.center) to (9.center);
		\draw (11.center) to (10.center);
	\end{pgfonlayer}
\end{tikzpicture}\right).
\end{align}

As \begin{align*}
\mathcal{F}_n^{\text{irr}}\left(\begin{tikzpicture}[scale=.25, baseline=3]
	\begin{pgfonlayer}{nodelayer}
		\node   (33) at (-1.5, 2.5) {};
		\node   (34) at (-1.5, 1.5) {};
		\node   (35) at (-2, 0.5) {};
		\node   (36) at (-1, 0.5) {};
		\node   (37) at (-1.5, -0.5) {};
		\node   (38) at (-1.5, -1.5) {};
		\node   (39) at (-1.5, -2) {$a$};
		\node   (40) at (-1.5, 3) {$a$};
	\end{pgfonlayer}
	\begin{pgfonlayer}{edgelayer}
		\draw (33.center) to (34.center);
		\draw [bend left=90, looseness=3.50] (35.center) to (36.center);
		\draw [bend left=90, looseness=3.50] (36.center) to (35.center);
		\draw (38.center) to (37.center);
	\end{pgfonlayer}
\end{tikzpicture}\right)(v_{a})&=  m_{b,c}^{a}\circ s_{a}^{b,c}(v_{a})=m_{b,c}^{a}(v_b \otimes v_c)=v_a\\
\mathcal{F}_n^{\text{irr}}\left(\begin{tikzpicture}[scale=.25, baseline=3]
	\begin{pgfonlayer}{nodelayer}
		\node   (25) at (1, 2.5) {};
		\node   (27) at (1, -1.5) {};
		\node   (29) at (1, 3) {$a$};
		\node   (31) at (1, -2) {$a$};
	\end{pgfonlayer}
	\begin{pgfonlayer}{edgelayer}
		\draw (25.center) to (27.center);
	\end{pgfonlayer}
\end{tikzpicture}\right)(v_a)&= id_a(v_{a})=v_a. 
\end{align*}

Thus, 
\begin{equation}
\mathcal{F}_n^{\text{irr}}\left(\begin{tikzpicture}[scale=.25, baseline=3]
	\begin{pgfonlayer}{nodelayer}
		\node   (33) at (-1.5, 2.5) {};
		\node   (34) at (-1.5, 1.5) {};
		\node   (35) at (-2, 0.5) {};
		\node   (36) at (-1, 0.5) {};
		\node   (37) at (-1.5, -0.5) {};
		\node   (38) at (-1.5, -1.5) {};
		\node   (39) at (-1.5, -2) {$a$};
		\node   (40) at (-1.5, 3) {$a$};
	\end{pgfonlayer}
	\begin{pgfonlayer}{edgelayer}
		\draw (33.center) to (34.center);
		\draw [bend left=90, looseness=3.50] (35.center) to (36.center);
		\draw [bend left=90, looseness=3.50] (36.center) to (35.center);
		\draw (38.center) to (37.center);
	\end{pgfonlayer}
\end{tikzpicture}\right)=\mathcal{F}_n^{\text{irr}}\left(\begin{tikzpicture}[scale=.25, baseline=3]
	\begin{pgfonlayer}{nodelayer}
		\node   (25) at (1, 2.5) {};
		\node   (27) at (1, -1.5) {};
		\node   (29) at (1, 3) {$a$};
		\node   (31) at (1, -2) {$a$};
	\end{pgfonlayer}
	\begin{pgfonlayer}{edgelayer}
		\draw (25.center) to (27.center);
	\end{pgfonlayer}
\end{tikzpicture}\right).
\end{equation}

Therefore, the functor $\mathcal{F}_n^{\text{irr}}$ is well defined.

\end{proof}

\begin{theorem}
The functors $\mathcal{F}_n^{\text{irr}}:  \mathcal{C}_n^{\text{irr}} \longrightarrow \mathbf{C}_n\textbf{-mod}_{\text{irr}}$ are full.
\end{theorem}

\begin{proof}
It suffices to show that given $[k_1, k_2, \dots, k_l] \in (\ZZ/n\ZZ)^l$ and a morphism $f\in \mathbf{C}_n\textbf{-mod}_{\text{irr}}$ where $f:\bigotimes\limits_{i=1}^{l}\mathbf{C}_n^{(k_i)}\longrightarrow \bigotimes\limits_{j=1}^{k_1+k_2+\cdots+k_l} \mathbf{C}_n^{(1)}$, there exists a morphism $d\in\mathcal{C}_n^{\text{irr}}$ such that $\mathcal{F}_n^{\text{irr}}(d)=f$.  As $\bigotimes\limits_{i=1}^{l}\mathbf{C}_n^{(k_i)}$ and $\bigotimes\limits_{j=1}^{k_1+k_2+\cdots+k_l} \mathbf{C}_n^{(1)}$ are one dimensional $\mathbf{C}_n$-modules, they are irreducible as $\mathbf{C}_n$-modules, and thus, up to scaling, there is only one non-zero morphism.  Said another way, we need only show there exists a $d\in\mathcal{C}_n^{\text{irr}}$ such that $\mathcal{F}_n^{\text{irr}}(d):\bigotimes\limits_{i=1}^{l}\mathbf{C}_n^{(k_i)}\longrightarrow \bigotimes\limits_{j=1}^{k_1+k_2+\cdots+k_l} \mathbf{C}_n^{(1)}$.

We construct a diagram from $[k_1, k_2, \dots, k_l]$ to $[1, 1, \dots, 1]$ where there are $k_1+k_2+\cdots+k_l=1+1+\cdots+1$.  Consider the following diagram:

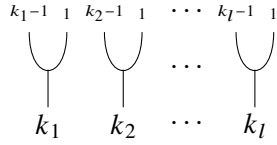
\begin{center}
\begin{tikzpicture}[scale=.5]
	\begin{pgfonlayer}{nodelayer}
		\node   (0) at (-4.5, -0.5) {$k_1$};
		\node   (1) at (-2.5, -0.5) {$k_2$};
		\node   (2) at (-0.75, -0.5) {$\cdots$};
		\node   (3) at (1, -0.5) {$k_l$};
		\node   (4) at (-4.5, 0) {};
		\node   (5) at (-2.5, 0) {};
		\node   (6) at (1, 0) {};
		\node   (7) at (-4.5, 1) {};
		\node   (8) at (-5, 2) {};
		\node   (9) at (-4, 2) {};
		\node   (10) at (-2.5, 1) {};
		\node   (11) at (-3, 2) {};
		\node   (12) at (-2, 2) {};
		\node   (13) at (1, 1) {};
		\node   (14) at (0.5, 2) {};
		\node   (15) at (1.5, 2) {};
		\node   (16) at (-0.75, 1) {$\cdots$};
		\node   (17) at (-5, 2.5) {$\mathsmaller{\mathsmaller{k_1-1}}$};
		\node   (18) at (-4, 2.5) {$\mathsmaller{\mathsmaller{1}}$};
		\node   (19) at (-3, 2.5) {$\mathsmaller{\mathsmaller{k_2-1}}$};
		\node   (20) at (-2, 2.5) {$\mathsmaller{\mathsmaller{1}}$};
		\node   (21) at (0.5, 2.5) {$\mathsmaller{\mathsmaller{k_l-1}}$};
		\node   (22) at (1.5, 2.5) {$\mathsmaller{\mathsmaller{1}}$};
		\node   (23) at (-0.75, 2.5) {$\cdots$};
	\end{pgfonlayer}
	\begin{pgfonlayer}{edgelayer}
		\draw (4.center) to (7.center);
		\draw (5.center) to (10.center);
		\draw (6.center) to (13.center);
		\draw [bend right=90, looseness=3.50] (8.center) to (9.center);
		\draw [bend right=90, looseness=3.50] (11.center) to (12.center);
		\draw [bend right=90, looseness=3.50] (14.center) to (15.center);
	\end{pgfonlayer}
\end{tikzpicture}
\end{center}

\noindent for $k_i\not\equiv 1 \mod n$.  If $k_i\equiv 1 \mod n$, we replace the split diagram with the identity strand.  We continue to stack the split diagram until $k_i-j \equiv 1\mod n$ and tensor with the identity strand where needed.  This process is finite, and thus, we get the resulting diagram:

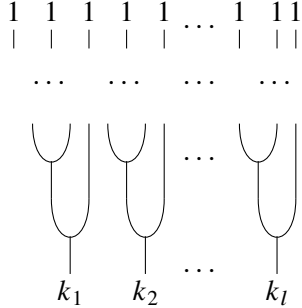
\begin{center}
\begin{tikzpicture}[scale=.5]
	\begin{pgfonlayer}{nodelayer}
		\node   (0) at (-4.5, -0.5) {$k_1$};
		\node   (1) at (-2.5, -0.5) {$k_2$};
		\node   (2) at (-1, 0) {$\cdots$};
		\node   (3) at (1, -0.5) {$k_l$};
		\node   (4) at (-4.5, 0) {};
		\node   (5) at (-2.5, 0) {};
		\node   (6) at (1, 0) {};
		\node   (7) at (-4.5, 1) {};
		\node   (8) at (-5, 2) {};
		\node   (9) at (-4, 2) {};
		\node   (10) at (-2.5, 1) {};
		\node   (11) at (-3, 2) {};
		\node   (12) at (-2, 2) {};
		\node   (13) at (1, 1) {};
		\node   (14) at (0.5, 2) {};
		\node   (15) at (1.5, 2) {};
		\node   (16) at (-1, 3) {$\cdots$};
		\node   (17) at (-5, 3) {};
		\node   (18) at (-4, 4) {};
		\node   (19) at (-3, 3) {};
		\node   (20) at (-2, 4) {};
		\node   (21) at (0.5, 3) {};
		\node   (22) at (1.5, 4) {};
		\node   (23) at (-5.5, 4) {};
		\node   (24) at (-4.5, 4) {};
		\node   (25) at (-3.5, 4) {};
		\node   (26) at (-2.5, 4) {};
		\node   (27) at (0, 4) {};
		\node   (28) at (1, 4) {};
		\node   (29) at (-5, 5) {$\cdots$};
		\node   (30) at (-3, 5) {$\cdots$};
		\node   (31) at (-1, 5) {$\cdots$};
		\node   (32) at (1, 5) {$\cdots$};
		\node   (33) at (-6, 6) {};
		\node   (34) at (-5, 6) {};
		\node   (35) at (-4, 6) {};
		\node   (36) at (-3, 6) {};
		\node   (37) at (-2, 6) {};
		\node   (39) at (0, 6) {};
		\node   (40) at (1, 6) {};
		\node   (41) at (1.5, 6) {};
		\node   (42) at (-6, 6.5) {};
		\node   (43) at (-5, 6.5) {};
		\node   (44) at (-4, 6.5) {};
		\node   (45) at (-3, 6.5) {};
		\node   (46) at (-2, 6.5) {};
		\node   (47) at (-1, 6.5) {$\cdots$};
		\node   (48) at (0, 6.5) {};
		\node   (49) at (1, 6.5) {};
		\node   (50) at (1.5, 6.5) {};
		\node   (51) at (-6, 7) {$1$};
		\node   (52) at (-5, 7) {$1$};
		\node   (53) at (-4, 7) {$1$};
		\node   (54) at (-3, 7) {$1$};
		\node   (55) at (-2, 7) {$1$};
		\node   (57) at (0, 7) {$1$};
		\node   (58) at (1, 7) {$1$};
		\node   (59) at (1.5, 7) {$1$};
	\end{pgfonlayer}
	\begin{pgfonlayer}{edgelayer}
		\draw (4.center) to (7.center);
		\draw (5.center) to (10.center);
		\draw (6.center) to (13.center);
		\draw [bend right=90, looseness=3.50] (8.center) to (9.center);
		\draw [bend right=90, looseness=3.50] (11.center) to (12.center);
		\draw [bend right=90, looseness=3.50] (14.center) to (15.center);
		\draw (8.center) to (17.center);
		\draw [bend right=90, looseness=3.50] (23.center) to (24.center);
		\draw (18.center) to (9.center);
		\draw (33.center) to (42.center);
		\draw (34.center) to (43.center);
		\draw (35.center) to (44.center);
		\draw (36.center) to (45.center);
		\draw (37.center) to (46.center);
		\draw [bend right=90, looseness=3.50] (25.center) to (26.center);
		\draw (19.center) to (11.center);
		\draw (20.center) to (12.center);
		\draw [bend right=90, looseness=3.50] (27.center) to (28.center);
		\draw (14.center) to (21.center);
		\draw (22.center) to (15.center);
		\draw (50.center) to (41.center);
		\draw (49.center) to (40.center);
		\draw (48.center) to (39.center);
	\end{pgfonlayer}
\end{tikzpicture}
\end{center}

It is clear that the image of this diagram under the functor $\mathcal{F}_n^{\text{irr}}$ is a non-zero homomorphism which sends the vector $v_{k_1}\otimes v_{k_2}\otimes \cdots \otimes v_{k_l}$ to the vector $v_1\otimes v_1 \otimes\cdots \otimes v_1$, and therefore, $\mathcal{F}_n^{\text{irr}}$ is full. 
\end{proof}

\begin{theorem}
\label{thm:CnirrFaithful}
The functor $\mathcal{F}_n^{\text{irr}}:  \mathcal{C}_n^{\text{irr}} \longrightarrow \mathbf{C}_n\textbf{-mod}_{\text{irr}}$ is faithful.
\end{theorem}

\begin{proof}
Let $[a_1,a_2, \dots, a_{m_1}] \in \left(\ZZ/n\ZZ\right)^{m_1}$ and $[b_1,b_2,\dots, b_{m_2}] \in \left(\ZZ/n\ZZ\right)^{m_2}$.  We have that Hom$_{\mathbf{C}_n}\left(\bigotimes\limits_{i=1}^{m_2}\mathbf{C}_n^{(a_i)},\bigotimes\limits_{j=1}^{m_2}\mathbf{C}_n^{(b_j)}\right)$ has dimension $1$ iff $\sum\limits_{i=1}^{m_1}a_i\equiv\sum\limits_{j=1}^{m_2}b_j\mod n$ and $0$ otherwise.  Thus, it suffices to show that there is one morphism up to scaling by $\CC$ in $\mathcal{C}_n^{\text{irr}}$ between $[a_1,a_2, \dots, a_{m_1}]$ and $[b_1,b_2,\dots, b_{m_2}]$ when $\sum\limits_{i=1}^{m_1}a_i\equiv\sum\limits_{j=1}^{m_2}b_j\mod n$.

Consider a diagram $d\in \operatorname{Hom}_{\mathcal{C}_n^{\text{irr}}}\left(\bigotimes\limits_{i=1}^{m_1}a_i, \bigotimes\limits_{j=1}^{m_2}b_j\right)$ where $\sum\limits_{i=1}^{m_1}a_i\equiv\sum\limits_{j=1}^{m_2}b_j\mod n$:

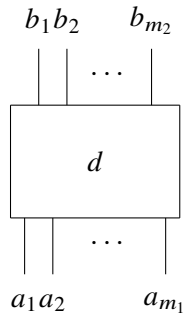
\begin{center}
\begin{tikzpicture}[scale=.75]
	\begin{pgfonlayer}{nodelayer}
		\node   (0) at (-1.5, -1) {};
		\node   (1) at (1.5, -1) {};
		\node   (2) at (-1.5, 1) {};
		\node   (3) at (1.5, 1) {};
		\node   (4) at (-1, 1) {};
		\node   (5) at (-0.5, 1) {};
		\node   (6) at (1, 1) {};
		\node   (7) at (-1.25, -1) {};
		\node   (8) at (-0.75, -1) {};
		\node   (9) at (1.25, -1) {};
		\node   (10) at (-1, 2) {};
		\node   (11) at (-0.5, 2) {};
		\node   (12) at (1, 2) {};
		\node   (13) at (-1.25, -2) {};
		\node   (14) at (-0.75, -2) {};
		\node   (15) at (1.25, -2) {};
		\node   (16) at (0.25, 1.5) {$\cdots$};
		\node   (17) at (0.25, -1.5) {$\cdots$};
		\node   (18) at (0, 0) {$d$};
		\node   (19) at (-1.25, -2.5) {$a_1$};
		\node   (20) at (-0.75, -2.5) {$a_2$};
		\node   (21) at (1.25, -2.5) {$a_{m_1}$};
		\node   (22) at (-1, 2.5) {$b_1$};
		\node   (23) at (-0.5, 2.5) {$b_2$};
		\node   (24) at (1, 2.5) {$b_{m_2}$};
	\end{pgfonlayer}
	\begin{pgfonlayer}{edgelayer}
		\draw (10.center) to (4.center);
		\draw (11.center) to (5.center);
		\draw (12.center) to (6.center);
		\draw (2.center) to (3.center);
		\draw (3.center) to (1.center);
		\draw (1.center) to (0.center);
		\draw (0.center) to (2.center);
		\draw (7.center) to (13.center);
		\draw (8.center) to (14.center);
		\draw (9.center) to (15.center);
	\end{pgfonlayer}
\end{tikzpicture}
\end{center}

Using the relation from \ref{eqn:CnnRelation2}:

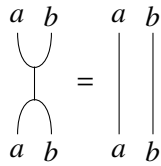
\begin{center}
\begin{tikzpicture}[scale=.45, baseline=0]
	\begin{pgfonlayer}{nodelayer}
		\node   (9) at (-2, -1) {};
		\node   (10) at (-1, -1) {};
		\node   (11) at (-1.5, 0) {};
		\node   (12) at (-1.5, 1) {};
		\node   (13) at (-1, 2) {};
		\node   (14) at (-2, 2) {};
		\node   (19) at (-2, -1.5) {$a$};
		\node   (20) at (-1, -1.5) {$b$};
		\node   (21) at (-2, 2.5) {$a$};
		\node   (22) at (-1, 2.5) {$b$};
		\node   (23) at (0, 0.5) {$=$};
		\node   (24) at (1, 2) {};
		\node   (25) at (2, 2) {};
		\node   (26) at (1, -1) {};
		\node   (27) at (2, -1) {};
		\node   (28) at (1, 2.5) {$a$};
		\node   (29) at (2, 2.5) {$b$};
		\node   (30) at (1, -1.5) {$a$};
		\node   (31) at (2, -1.5) {$b$};
	\end{pgfonlayer}
	\begin{pgfonlayer}{edgelayer}
		\draw [bend right=90, looseness=3.50] (14.center) to (13.center);
		\draw (12.center) to (11.center);
		\draw [bend left=90, looseness=3.50] (9.center) to (10.center);
		\draw (24.center) to (26.center);
		\draw (25.center) to (27.center);
	\end{pgfonlayer}
\end{tikzpicture} 
\end{center}

\noindent we can rewrite $d$ as

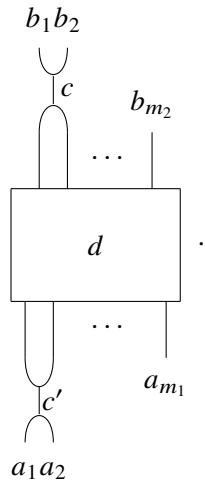
\begin{center}
\begin{tikzpicture}[baseline=0, scale=.75]
	\begin{pgfonlayer}{nodelayer}
		\node   (0) at (-1.5, -1) {};
		\node   (1) at (1.5, -1) {};
		\node   (2) at (-1.5, 1) {};
		\node   (3) at (1.5, 1) {};
		\node   (4) at (-1, 1) {};
		\node   (5) at (-0.5, 1) {};
		\node   (6) at (1, 1) {};
		\node   (7) at (-1.25, -1) {};
		\node   (8) at (-0.75, -1) {};
		\node   (9) at (1.25, -1) {};
		\node   (10) at (-1, 2) {};
		\node   (11) at (-0.5, 2) {};
		\node   (12) at (1, 2) {};
		\node   (13) at (-1.25, -2) {};
		\node   (14) at (-0.75, -2) {};
		\node   (15) at (1.25, -2) {};
		\node   (16) at (0.25, 1.5) {$\cdots$};
		\node   (17) at (0.25, -1.5) {$\cdots$};
		\node   (18) at (0, 0) {$d$};
		\node   (19) at (-1.25, -3.5) {};
		\node   (20) at (-0.75, -3.5) {};
		\node   (21) at (1.25, -2.5) {$a_{m_1}$};
		\node   (22) at (-1, 3.5) {};
		\node   (23) at (-0.5, 3.5) {};
		\node   (24) at (1, 2.5) {$b_{m_2}$};
		\node   (25) at (-0.75, 3) {};
		\node   (26) at (-0.75, 2.5) {};
		\node   (27) at (-1, -3) {};
		\node   (28) at (-1, -2.5) {};
		\node   (29) at (-1.25, -4) {$a_1$};
		\node   (30) at (-0.75, -4) {$a_2$};
		\node   (31) at (-1, 4) {$b_1$};
		\node   (32) at (-0.5, 4) {$b_2$};
		\node   (33) at (-0.5, 2.75) {$c$};
		\node   (33) at (-0.75, -2.75) {$c'$};
	\end{pgfonlayer}
	\begin{pgfonlayer}{edgelayer}
		\draw (10.center) to (4.center);
		\draw (11.center) to (5.center);
		\draw (12.center) to (6.center);
		\draw (2.center) to (3.center);
		\draw (3.center) to (1.center);
		\draw (1.center) to (0.center);
		\draw (0.center) to (2.center);
		\draw (7.center) to (13.center);
		\draw (8.center) to (14.center);
		\draw (9.center) to (15.center);
		\draw [bend left=90, looseness=3.25] (10.center) to (11.center);
		\draw [bend right=90, looseness=3.25] (22.center) to (23.center);
		\draw (25.center) to (26.center);
		\draw [bend right=90, looseness=3.50] (13.center) to (14.center);
		\draw [bend left=90, looseness=3.25] (19.center) to (20.center);
		\draw (28.center) to (27.center);
	\end{pgfonlayer}
\end{tikzpicture}.
\end{center}

Now, consider using the relation iteratively to get the following equality:

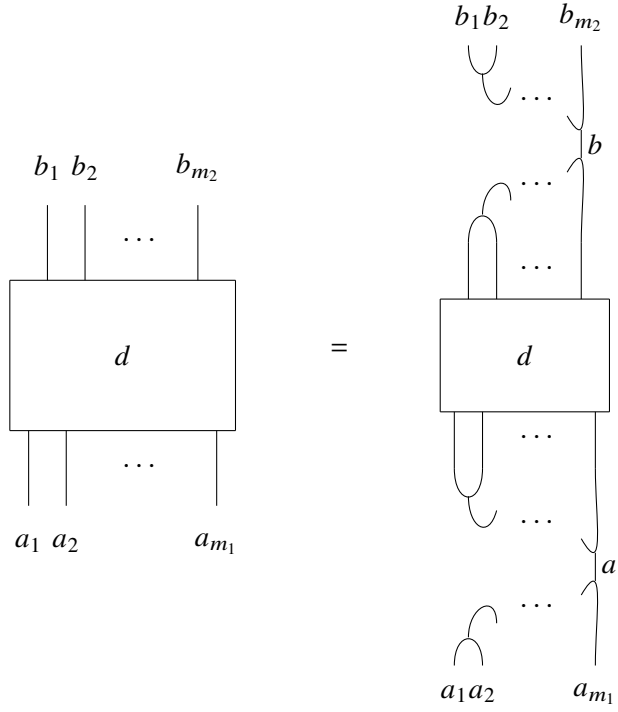
\begin{center}
\begin{tikzpicture}[baseline=0]
	\begin{pgfonlayer}{nodelayer}
		\node   (0) at (-1.5, -1) {};
		\node   (1) at (1.5, -1) {};
		\node   (2) at (-1.5, 1) {};
		\node   (3) at (1.5, 1) {};
		\node   (4) at (-1, 1) {};
		\node   (5) at (-0.5, 1) {};
		\node   (6) at (1, 1) {};
		\node   (7) at (-1.25, -1) {};
		\node   (8) at (-0.75, -1) {};
		\node   (9) at (1.25, -1) {};
		\node   (10) at (-1, 2) {};
		\node   (11) at (-0.5, 2) {};
		\node   (12) at (1, 2) {};
		\node   (13) at (-1.25, -2) {};
		\node   (14) at (-0.75, -2) {};
		\node   (15) at (1.25, -2) {};
		\node   (16) at (0.25, 1.5) {$\cdots$};
		\node   (17) at (0.25, -1.5) {$\cdots$};
		\node   (18) at (0, 0) {$d$};
		\node   (19) at (-1.25, -2.5) {$a_1$};
		\node   (20) at (-0.75, -2.5) {$a_2$};
		\node   (21) at (1.25, -2.5) {$a_{m_1}$};
		\node   (22) at (-1, 2.5) {$b_1$};
		\node   (23) at (-0.5, 2.5) {$b_2$};
		\node   (24) at (1, 2.5) {$b_{m_2}$};
	\end{pgfonlayer}
	\begin{pgfonlayer}{edgelayer}
		\draw (10.center) to (4.center);
		\draw (11.center) to (5.center);
		\draw (12.center) to (6.center);
		\draw (2.center) to (3.center);
		\draw (3.center) to (1.center);
		\draw (1.center) to (0.center);
		\draw (0.center) to (2.center);
		\draw (7.center) to (13.center);
		\draw (8.center) to (14.center);
		\draw (9.center) to (15.center);
	\end{pgfonlayer}
\end{tikzpicture}
\hspace{10mm}$=$\hspace{10mm}
\begin{tikzpicture}[scale=.75,baseline=0]
	\begin{pgfonlayer}{nodelayer}
		\node   (0) at (-1.5, -1) {};
		\node   (1) at (1.5, -1) {};
		\node   (2) at (-1.5, 1) {};
		\node   (3) at (1.5, 1) {};
		\node   (4) at (-1, 1) {};
		\node   (5) at (-0.5, 1) {};
		\node   (6) at (1, 1) {};
		\node   (7) at (-1.25, -1) {};
		\node   (8) at (-0.75, -1) {};
		\node   (9) at (1.25, -1) {};
		\node   (10) at (-1, 2) {};
		\node   (11) at (-0.5, 2) {};
		\node   (12) at (1, 2) {};
		\node   (13) at (-1.25, -2) {};
		\node   (14) at (-0.75, -2) {};
		\node   (15) at (1.25, -2) {};
		\node   (16) at (0.25, 1.5) {$\cdots$};
		\node   (17) at (0.25, -1.5) {$\cdots$};
		\node   (18) at (0, 0) {$d$};
		\node   (19) at (-1.25, -6) {$a_1$};
		\node   (20) at (-0.75, -6) {$a_2$};
		\node   (21) at (1.25, -6) {$a_{m_1}$};
		\node   (22) at (-1, 6) {$b_1$};
		\node   (23) at (-0.5, 6) {$b_2$};
		\node   (24) at (1, 6) {$b_{m_2}$};
		\node   (25) at (-1.25, -5.5) {};
		\node   (26) at (-0.75, -5.5) {};
		\node   (27) at (1.25, -5.5) {};
		\node   (28) at (1, -4.25) {};
		\node   (29) at (1, -3.25) {};
		\node   (30) at (1.25, -3.5) {};
		\node   (31) at (1.25, -4) {};
		\node   (32) at (-1, -5) {};
		\node   (33) at (-1, -2.5) {};
		\node   (34) at (-0.5, -2.75) {};
		\node   (35) at (-0.5, -4.75) {};
		\node   (36) at (0.25, -3) {$\cdots$};
		\node   (37) at (0.25, -4.5) {$\cdots$};
		\node   (38) at (-1, 5.5) {};
		\node   (39) at (-0.5, 5.5) {};
		\node   (40) at (1, 5.5) {};
		\node   (41) at (-1, 2) {};
		\node   (42) at (-0.5, 2) {};
		\node   (43) at (1, 2) {};
		\node   (44) at (0.75, 3.25) {};
		\node   (45) at (0.75, 4.25) {};
		\node   (46) at (1, 4) {};
		\node   (47) at (1, 3.5) {};
		\node   (48) at (-0.75, 2.5) {};
		\node   (49) at (-0.75, 5) {};
		\node   (50) at (-0.25, 4.75) {};
		\node   (51) at (-0.25, 2.75) {};
		\node   (52) at (0.25, 4.5) {$\cdots$};
		\node   (53) at (0.25, 3) {$\cdots$};
		\node   (54) at (1.5, -3.75) {$a$};
		\node   (55) at (1.25, 3.75) {$b$};
	\end{pgfonlayer}
	\begin{pgfonlayer}{edgelayer}
		\draw (10.center) to (4.center);
		\draw (11.center) to (5.center);
		\draw (12.center) to (6.center);
		\draw (2.center) to (3.center);
		\draw (3.center) to (1.center);
		\draw (1.center) to (0.center);
		\draw (0.center) to (2.center);
		\draw (7.center) to (13.center);
		\draw (8.center) to (14.center);
		\draw (9.center) to (15.center);
		\draw (30.center) to (31.center);
		\draw [in=60, out=90, looseness=1.75] (27.center) to (28.center);
		\draw [in=-60, out=-90, looseness=1.75] (15.center) to (29.center);
		\draw [bend left=90, looseness=3.25] (25.center) to (26.center);
		\draw [bend right=90, looseness=3.50] (13.center) to (14.center);
		\draw [in=255, out=-90, looseness=2.50] (33.center) to (34.center);
		\draw [in=90, out=90, looseness=2.50] (32.center) to (35.center);
		\draw (46.center) to (47.center);
		\draw [in=60, out=90, looseness=1.75] (43.center) to (44.center);
		\draw [in=-60, out=-90, looseness=1.75] (40.center) to (45.center);
		\draw [bend left=90, looseness=3.25] (41.center) to (42.center);
		\draw [bend right=90, looseness=3.50] (38.center) to (39.center);
		\draw [in=255, out=-90, looseness=2.50] (49.center) to (50.center);
		\draw [in=90, out=90, looseness=2.50] (48.center) to (51.center);
	\end{pgfonlayer}
\end{tikzpicture}
\end{center}

\noindent where $a\equiv\sum\limits_{i=1}^{m_1}a_i \mod n$ and $b\equiv\sum\limits_{j=1}^{m_2}b_j \mod n$, and since $\sum\limits_{i=1}^{m_1}a_i\equiv\sum\limits_{j=1}^{m_2}b_j\mod n$, then $a\equiv b\mod n$.  Thus by Lemma \ref{lem:AnyDiagIsIdentity} the diagram on the right hand side of the equation is equal to

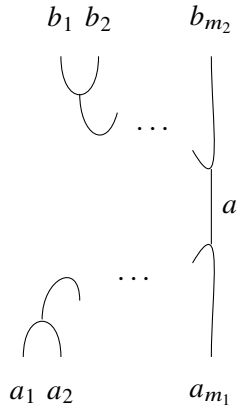
\begin{center}
\begin{tikzpicture}
	\begin{pgfonlayer}{nodelayer}
		\node   (19) at (-1.5, -2.5) {$a_1$};
		\node   (20) at (-1, -2.5) {$a_2$};
		\node   (21) at (1, -2.5) {$a_{m_1}$};
		\node   (22) at (-1, 2.5) {$b_1$};
		\node   (23) at (-0.5, 2.5) {$b_2$};
		\node   (24) at (1, 2.5) {$b_{m_2}$};
		\node   (25) at (-1.5, -2) {};
		\node   (26) at (-1, -2) {};
		\node   (27) at (1, -2) {};
		\node   (28) at (0.75, -0.75) {};
		\node   (31) at (1, -0.5) {};
		\node   (32) at (-1.25, -1.5) {};
		\node   (35) at (-0.75, -1.25) {};
		\node   (37) at (0, -1) {$\cdots$};
		\node   (38) at (-1, 2) {};
		\node   (39) at (-0.5, 2) {};
		\node   (40) at (1, 2) {};
		\node   (45) at (0.75, 0.75) {};
		\node   (46) at (1, 0.5) {};
		\node   (49) at (-0.75, 1.5) {};
		\node   (50) at (-0.25, 1.25) {};
		\node   (52) at (0.25, 1) {$\cdots$};
		\node   (53) at (1.25, 0) {$a$};
	\end{pgfonlayer}
	\begin{pgfonlayer}{edgelayer}
		\draw [in=60, out=90, looseness=1.75] (27.center) to (28.center);
		\draw [bend left=90, looseness=3.25] (25.center) to (26.center);
		\draw [in=90, out=90, looseness=2.50] (32.center) to (35.center);
		\draw [in=-60, out=-90, looseness=1.75] (40.center) to (45.center);
		\draw [bend right=90, looseness=3.50] (38.center) to (39.center);
		\draw [in=255, out=-90, looseness=2.50] (49.center) to (50.center);
		\draw (46.center) to (31.center);
	\end{pgfonlayer}
\end{tikzpicture}
\end{center}

\noindent where $a\equiv \sum\limits_{i=1}^{m_1}a_i\equiv\sum\limits_{j=1}^{m_2}b_j\mod n$.  Thus, there is one diagram up to scaling in $\operatorname{Hom}_{\mathcal{C}_n^{\text{irr}}}\left(\bigotimes\limits_{i=1}^{m_1}a_i, \bigotimes\limits_{j=1}^{m_2}b_j\right)$.  Therefore, the functor $\mathcal{F}_n^{\text{irr}}$ is faithful.

\end{proof}

Now that we are familiar with a specific example of the type of diagrammatic category we would like to construct, the next section develops diagrammatic categories which utilize the representation graphs given a group $G$ and a module $V$.


\section{The Categories \texorpdfstring{$G$}{G}-\textbf{mod}\texorpdfstring{$_{\text{irr}}$}{irr}
 and \texorpdfstring{$\mathbf{Dgrams}_{R(V,G)}$}{DgramsRVG}}
\label{chpt:Dgrams}


\noindent Let $\Gamma$ be a directed graph with no multiple parallel edges, that is, no two nodes have two or more directed edges with the same direction between them, and with the set of vertices indexed by the set $I_{\Gamma}$.  For example, let $\Gamma$ be the representation graph of one of the finite subgroups of $SU(2)$ as discussed in \ref{chpt:Prelims}.  The constructions in this section can be used to define a diagrammatic category associated to the graph $\Gamma$.  Furthermore, as we have done in this paper, one can start with a semisimple symmetric monoidal $k$-linear category over some field $k$ and consider the full subcategory $\mathcal{C}$ where the objects are monoidally generated by the simple objects.  Regardless of whether or not this category comes from representation theory, we may construct a representation graph.  That is, we may fix a simple object $x$ and construct the graph with nodes corresponding to the simple objects and a directed edge from vertex $v$ to vertex $u$ if the simple object corresponding to $u$ is a direct summand of $x\otimes v$.  If this directed graph has no multiple edges, then using the ideas in this section, one can diagrammatically define a category which is categorically equivalent to $\mathcal{C}$.


\subsection{The Category \texorpdfstring{$G$}{G}-\textbf{mod}\texorpdfstring{$_{\text{irr}}$}{irr}}
\label{sec:Gmodirr}


We first consider when $\Gamma$ is the representation graph for a group $G$.  Let $G$ be a group, not necessarily finite, and let $V$ be a $G$-module such that the resulting representation graph $R(V,G)$ is a connected graph with no multiple parallel edges.  Furthermore, we will assume that $V$ corresponds to a node in $R(V,G)$.  That is, we will assume $V$ is a simple $G$-module.  It is worth mentioning here that $SU(2)$ and the finite subgroups of $SU(2)$ are examples of such $G$, but there are others as well.  See \cite{AEA22, EP14, KO02, Fren10}.

First, let us set some notation.  Let $\left\{G^{(a)}\right\}_{a\in I_G}$ be a set of fixed isomorphism class representatives of simple $G$-modules with $I_G$ being an indexing set for the finite-dimensional simple $G$-modules.  Furthermore, as $V$ is a simple $G$-module and $I_G$ is an indexing set for the simple, $G$-modules, one of the elements of $I_G$ corresponds to $V$.  For notational convenience, we let this index be the symbol $1$.  In particular, we will use $V$ and $G^{(1)}$ interchangeably.  Furthermore, for $a,b \in I_G$, we will also use $b\rightarrow a$ to denote that $b$ is adjacent to $a$ in $R(V,G)$.  Note that in an undirected graph $b\rightarrow a$ implies $a\rightarrow b$.

\begin{definition}
We let $G$-\textbf{mod}$_{\text{irr}}$ be the full monoidal subcategory of $G$-\textbf{mod} with objects generated by $G^{(a)}$ where $a\in I_{G}$.
\end{definition}

Notice, the morphisms of this category are elements of the $\CC$-vector spaces $$\operatorname{Hom}_{G}\left(\bigotimes\limits_{i=1}^n G^{(a_i)},\bigotimes\limits_{j=1}^m G^{(b_j)}\right)$$ where $a_i, b_j\in I_{G}$ and $n,m\in\NN$.  

We define certain $G$-module homomorphisms concretely.  Since $R(V,G)$ has no multiple edges by assumption, the space Hom$_{G}\left(V\otimes G^{(a)}, G^{(b)}\right)$ is $1$-dimensional for each $b$ adjacent to $a$ in $R(V,G)$ and $0$-dimensional otherwise.  We may choose to fix a map in Hom$_{G}\left(V\otimes G^{(a)}, G^{(b)}\right)$ for each $b$ adjacent to $a$ and name them $m_{1\ a}^{\ b}$.  Notice that the choice of each $m_{1\ a}^{\ b}$, while non-trivial, is only up up to a scalar by Schur's Lemma.  Furthermore, with the $m_{1\ a}^{\ b}$ fixed, for each $b$ which is adjacent to $a$ in $R(V,G)$ there are unique non-zero $G$-module homomorphisms, which we name $s_{\ b}^{1\ a}$, which span Hom$_{G}\left(G^{(b)}, V\otimes G^{(a)}\right)$ such that the following is satisfied:

\begin{equation}
\label{eqn:msidempotents}
\sum\limits_{b\rightarrow a}s_{\ b}^{1\ a} \circ m_{1\ a}^{\ b}=\operatorname{id}_{G^{(1)}\otimes G^{(a)}}.
\end{equation}

\begin{remark}
    If $R(V,G)$ has multiple edges from $a$ to $b$, then Hom$_{G}\left(V\otimes G^{(a)}, G^{(b)}\right)$ has dimension greater than $1$.  In this case, we may still define a set of maps which when extended linearly describe the whole space.  If the set is minimal, this is analogous to choosing a basis for Hom$_{G}\left(V\otimes G^{(a)}, G^{(b)}\right)$.  Furthermore, the relationships between the maps in Hom$_{G}\left(V\otimes G^{(a)}, G^{(b)}\right)$ and the maps in Hom$_{G}\left(G^{(b)}, V\otimes G^{(a)}\right)$ would require exploration.  As most groups admit representation graphs which contain multiple edges between nodes, it would be interesting to explore the extension of the ideas in this section to these much more ubiquitous representation graphs.
\end{remark}

Let us consider an example:  let $\mathbf{T}$ be the binary tetrahedral group.  We will use notation consistent with Example \ref{BTex}.

\begin{example}
We let $\mathbf{T}$-\textbf{mod}$_{\text{irr}}$ be the full monoidal subcategory of $\mathbf{T}$-\textbf{mod} with objects generated by $T^{(a)}$ with $a\in I_{\mathbf{T}}=\{1,2,3,4,3',4'\}$.  Notice, the morphisms of this category are in Hom$_{\mathbf{T}}\left(\bigotimes\limits_{k=1}^{n} T^{(a_k)},\bigotimes\limits_{\ell=1}^{m} T^{(b_{\ell})}\right)$.  

We will consider the following $ \mathbf{T}$-module homomorphisms:

\begin{center}
\begin{align*}
m_{1\ 0}^{\ 1}: &T^{(1)}\otimes T^{(0)}\longrightarrow T^{(1)}   &  
m_{1\ 1}^{\ 2}: &T^{(1)}\otimes T^{(1)}\longrightarrow T^{(2)}\\
&v_{-1}\otimes 1\mapsto v_{-1}   &   &v_{-1}\otimes v_{-1}\mapsto v_{-2}\\
&v_{1}\otimes 1\mapsto v_{1}   &   &v_{1}\otimes v_{1}\mapsto v_{2}\\ 
&    &    &v_{-1}\otimes v_{1} + v_{1}\otimes v_{-1}\mapsto v_{0'}\\
\end{align*}
\end{center}
\begin{center}
\begin{align*}
m_{1\ 1}^{\ 0}: &T^{(1)}\otimes T^{(1)}\longrightarrow T^{(0)}   &   
m_{1\ 2}^{\ 1}: &T^{(1)}\otimes T^{(2)}\longrightarrow T^{(1)}\\
&v_{-1}\otimes v_{1} - v_{1}\otimes v_{-1}\mapsto 1   &   &-\frac{1}{2}v_{-1}\otimes v_{0'} + v_{1}\otimes v_{-2}\mapsto v_{-1}\\
 &   &   &\frac{1}{2}v_{1}\otimes v_{0'} - v_{-1}\otimes v_{2}\mapsto v_{1}\\
\end{align*}
\begin{align*}
m_{1\ 3}^{\ 2}: T^{(1)}\otimes T^{(3)}\longrightarrow T^{(2)}   \\
v_{1}\otimes v_{3}\mapsto v_{2} - i\sqrt{3}v_{-2}   \\
v_{-1}\otimes v_{-3}\mapsto v_{-2} - i\sqrt{3}v_{2}   \\
v_{-1}\otimes v_{3} + v_{1}\otimes v_{-3}\mapsto -2v_{0'}  
\end{align*}
\begin{align*}
m_{1\ 2}^{\ 3}: T^{(1)}\otimes T^{(2)}\longrightarrow T^{(3)}\\
-v_{-1}\otimes v_{0'}-v_{1}\otimes v_{-2} + i\sqrt{3}v_{1}\otimes v_{2} \mapsto v_{-3}\\
v_{1}\otimes v_{0'} + v_{-1}\otimes v_{2} - i\sqrt{3}v_{-1}\otimes v_{-2} \mapsto v_{3}\\
\end{align*}
\begin{align*}
m_{1\ 3'}^{\ 2}: T^{(1)}\otimes T^{(3')}\longrightarrow T^{(2)}\\
v_{1}\otimes v_{3^{\prime}}\mapsto v_{-2} + i\sqrt{3}v_{2}\\
v_{-1}\otimes v_{-3^{\prime}}\mapsto v_{2} + i\sqrt{3}v_{-2}\\
v_{-1}\otimes v_{3^{\prime}} + v_{1}\otimes v_{-3^{\prime}}\mapsto -2v_{0^{\prime}}
\end{align*}
\begin{align*}
m_{1\ 2}^{\ 3^{\prime}}: T^{(1)}\otimes T^{(2)}\longrightarrow T^{(3^{\prime})}\\
-v_{-1}\otimes v_{0^{\prime}}-v_{1}\otimes v_{-2} - i\sqrt{3}v_{1}\otimes v_{2} \mapsto v_{-3^{\prime}}\\
v_{1}\otimes v_{0^{\prime}} + v_{-1}\otimes v_{2} + i\sqrt{3}v_{-1}\otimes v_{-2} \mapsto v_{3^{\prime}}
\end{align*}
\begin{align*}
m_{1\ 3}^{\ 4}: &T^{(1)}\otimes T^{(3)}\longrightarrow T^{(4)}   & 
 m_{1\ 3^{\prime}}^{\ 4^{\prime}}: &T^{(1)}\otimes T^{(3^{\prime})}\longrightarrow T^{(4^{\prime})}\\
&v_{-1}\otimes v_{3} - v_{1}\otimes v_{-3}\mapsto v_{4}   &   &v_{-1}\otimes v_{3^{\prime}} - v_{1}\otimes v_{-3^{\prime}}\mapsto v_{4^{\prime}}\\
\end{align*}
\begin{align*}
m_{1\ 4}^{\ 3}: &T^{(1)}\otimes T^{(4)}\longrightarrow T^{(3)}   &  
  m_{1\ 4^{\prime}}^{\ 3^{\prime}}: &T^{(1)}\otimes T^{(4^{\prime})}\longrightarrow T^{(3^{\prime})}\\
&v_{-1}\otimes v_{4}\mapsto v_{-3}    &   &v_{-1}\otimes v_{4^{\prime}}\mapsto v_{-3^{\prime}}\\
&v_{1}\otimes v_{4}\mapsto v_{3}   &   &v_{1}\otimes v_{4^{\prime}}\mapsto v_{3^{\prime}}
\end{align*}
\end{center}

We will then define $s^{1\ a}_{\ b}: T^{(b)}\longrightarrow T^{(1)}\otimes T^{(a)}$ to be the map satisfying the relation $\sum\limits_{b\rightarrow a}s^{1\ a}_{\ b} \circ m_{1\ a}^{\ b}=\operatorname{id}_{T^{(1)}\otimes T^{(b)}}$

\end{example}

Recall from Section \ref{sec:McKayGraphs} that $P(a,b)_k$ is the set paths from $a$ to $b$ of length $k$ and is subset of $P(a,b)$.  Let $\mathbf{p}=\left\{b_0, b_1, \dots, b_k\right\}\in P(a,b)_k$ where $b_0=1$ and $b_k=b$.  We fix $\pi_{\mathbf{p}}$ to be the map from $\left(G^{(1)}\right)^{\otimes k}$ onto the irreducible submodule $G^{(b)}$ using the previously fixed maps,  $m_{1\ a}^{\ b}$ and $s^{1\ a}_{\ b}$, in the following way:

\[
\pi_{\mathbf{p}}:=\left(m_{1\ b_{k-1}}^{\ b}\right)\circ \left(\operatorname{id}_V\otimes m_{1\ b_{k-2}}^{\ b_{k-1}}\right)\circ \cdots \circ \left(\left(\operatorname{id}_{G^{(1)}}\right)^{\otimes (k-3)}\otimes m_{1\ b_1}^{\ b_2}\right)\circ \left(\left(\operatorname{id}_{V}\right)^{\otimes (k-2)}\otimes m_{1\ 1}^{\ b_1}\right)
\]

\noindent Since the identity maps and the $m_{1\ a}^{\ b}$ are canonical (up to scaling), so then is $\pi_{\mathbf{p}}$.

Similarly, we let $\iota_{\mathbf{p}}$ be the map from $G^{(b)}$ into $\left(G^{(1)}\right)^{\otimes k}$ such that $\pi_{\mathbf{p}}\circ \iota_{\mathbf{p}}=\operatorname{id}_{G^{(b)}}$.  For each irreducible $G$-module $G^{(b)}$, there is a minimal $k_b$ such that $G^{(b)}\subset\left(G^{(1)}\right)^{\otimes k_b}$, and since the representation graph has no multiple edges, this corresponds to a single path $\mathbf{q}\in P(a,b)_{k_b}$.  Thus, $G^{(b)}$ shows up exactly once in $\left(G^{(1)}\right)^{\otimes k_b}$, and thus we let $\pi_{\mathbf{q}}$ and $\iota_{\mathbf{q}}$ be the corresponding projection and inclusion maps.


\subsection{The Category \texorpdfstring{$\mathbf{Dgrams}_{R(V,G)}$}{DgramsRVG}}


Now we turn to a diagrammatic category which needs only the data of the representation graphs $R(V,G)$ presented in the previous section to construct.

\begin{definition}
\label{def:DgramsRVG}
We let $\mathbf{Dgrams}_{R(V,G)}$ be the $\CC$-linear monoidal category with objects generated by $k\in I_{G}$ and morphisms generated by the following diagrams:

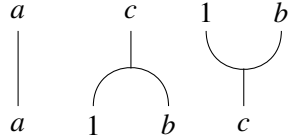
\begin{center}
\begin{tikzpicture}[scale=.5]
	\begin{pgfonlayer}{nodelayer}
		\node  (0) at (-3, -1) {};
		\node  (1) at (-3, 1) {};
		\node  (2) at (-1, -1) {};
		\node  (3) at (1, -1) {};
		\node  (4) at (0, 0) {};
		\node  (5) at (0, 1) {};
		\node  (6) at (2, 1) {};
		\node  (7) at (4, 1) {};
		\node  (8) at (3, 0) {};
		\node  (9) at (3, -1) {};
		\node  (10) at (-3, -1.5) {$a$};
		\node  (11) at (-3, 1.5) {$a$};
		\node  (12) at (-1, -1.5) {$1$};
		\node  (13) at (1, -1.5) {$b$};
		\node  (14) at (0, 1.5) {$c$};
		\node  (15) at (2, 1.5) {$1$};
		\node  (16) at (4, 1.5) {$b$};
		\node  (17) at (3, -1.5) {$c$};
	\end{pgfonlayer}
	\begin{pgfonlayer}{edgelayer}
		\draw (0.center) to (1.center);
		\draw [bend left=90, looseness=1.75] (2.center) to (3.center);
		\draw (4.center) to (5.center);
		\draw [bend right=90, looseness=1.75] (6.center) to (7.center);
		\draw (8.center) to (9.center);
	\end{pgfonlayer}
\end{tikzpicture}
\end{center} 

\noindent where $a,b, \text{ and }c\in I_{G}$ such that $c$ is adjacent to $b$ in the representation graph, $R(V,G)$.

The generators are subject to the following relations:

\begin{equation}
\label{eqn:DgramPop}
\begin{tikzpicture}[scale=0.5, baseline=-2]
	\begin{pgfonlayer}{nodelayer}
		\node   (0) at (-2, 2) {};
		\node   (1) at (-2, 1) {};
		\node   (2) at (-2.5, 0) {};
		\node   (3) at (-1.5, 0) {};
		\node   (4) at (-2, -1) {};
		\node   (5) at (-2, -2) {};
		\node   (6) at (-0.25, 0) {$=$};
		\node   (7) at (2, 1) {};
		\node   (8) at (2, -1) {};
		\node   (9) at (-2, -2.5) {$a$};
		\node   (10) at (2, -1.5) {$a$};
		\node   (11) at (2, 1.5) {$a$};
		\node   (12) at (-2, 2.5) {$a$};
		\node   (13) at (-1, 0) {$b$};
		\node   (14) at (-3, 0) {$1$};
		\node   (15) at (1, 0) {};
	\end{pgfonlayer}
	\begin{pgfonlayer}{edgelayer}
		\draw (0.center) to (1.center);
		\draw [bend left=90, looseness=3.50] (2.center) to (3.center);
		\draw (4.center) to (5.center);
		\draw (7.center) to (8.center);
		\draw [bend right=90, looseness=3.25] (2.center) to (3.center);
	\end{pgfonlayer}
\end{tikzpicture} \hspace{10mm} and\hspace{10mm} \begin{tikzpicture}[scale=0.5, baseline=-2]
	\begin{pgfonlayer}{nodelayer}
		\node   (0) at (-2, 0.5) {};
		\node   (1) at (-2, -0.5) {};
		\node   (2) at (-2.5, 1.5) {};
		\node   (3) at (-1.5, 1.5) {};
		\node   (4) at (-2.5, -1.5) {};
		\node   (5) at (-1.5, -1.5) {};
		\node   (9) at (0, 1) {};
		\node   (10) at (0, -1) {};
		\node   (11) at (1, 1) {};
		\node   (12) at (1, -1) {};
		\node   (13) at (-3.5, 0) {$\sum\limits_{b\rightarrow a}$};
		\node   (14) at (-1.5, 0) {$b$};
		\node   (15) at (-0.75, 0) {$=$};
		\node   (16) at (0, 1.5) {$1$};
		\node   (17) at (1, 1.5) {$a$};
		\node   (18) at (0, -1.5) {$1$};
		\node   (19) at (1, -1.5) {$a$};
		\node   (20) at (-2.5, -2) {$1$};
		\node   (21) at (-1.5, -2) {$a$};
		\node   (22) at (-2.5, 2) {$1$};
		\node   (23) at (-1.5, 2) {$a$};
	\end{pgfonlayer}
	\begin{pgfonlayer}{edgelayer}
		\draw [bend right=90, looseness=3.50] (2.center) to (3.center);
		\draw (0.center) to (1.center);
		\draw [bend left=90, looseness=3.50] (4.center) to (5.center);
		\draw (9.center) to (10.center);
		\draw (11.center) to (12.center);
	\end{pgfonlayer}
\end{tikzpicture}
\end{equation}

\end{definition}

Let us set some notation for some morphisms in the category $\mathbf{Dgrams}_{R(V,G)}$.  Recall the notation we introduced in Section \ref{sec:McKayGraphs}: for $\mathbf{p}=\left(1, b_1,\dots, b_{k-1},b\right)\in P(1,b)_{k}$, we let 

\begin{center}
\begin{tikzpicture}[scale=0.5, baseline=45]
	\begin{pgfonlayer}{nodelayer}
		\node   (0) at (-1, 0) {};
		\node   (1) at (0, 0) {};
		\node   (2) at (1.5, 0) {$\cdots$};
		\node   (3) at (3, 0) {};
		\node   (4) at (4, 0) {};
		\node   (5) at (5, 0) {};
		\node   (6) at (4.5, 1) {};
		\node   (8) at (3.75, 3) {};
		\node   (9) at (3.75, 2) {};
		\node   (11) at (3, 4) {};
		\node   (12) at (2.25, 5) {};
		\node   (13) at (1.25, 7) {};
		\node   (14) at (1.25, 6) {};
		\node   (17) at (-1, -0.5) {$1$};
		\node   (18) at (0, -0.5) {$1$};
		\node   (19) at (1.5, -0.5) {$\cdots$};
		\node   (20) at (3, -0.5) {$1$};
		\node   (21) at (4, -0.5) {$1$};
		\node   (22) at (5, -0.5) {$1$};
		\node   (23) at (3.25, 3.75) {$\cdot$};
		\node   (24) at (3.5, 3.5) {$\cdot$};
		\node   (25) at (3.75, 3.25) {$\cdot$};
		\node   (26) at (2.85, 5.5) {$b_{k-1}$};
		\node   (28) at (4.25, 2.5) {$b_2$};
		\node   (29) at (5, 1.5) {$b_{1}$};
		\node   (30) at (1.25, 7.5) {$b$};
		\node   (31) at (-2, 3.5) {$u_{\mathbf{p}}:=$};
	\end{pgfonlayer}
	\begin{pgfonlayer}{edgelayer}
		\draw [bend left=90, looseness=3.25] (4.center) to (5.center);
		\draw [in=105, out=90, looseness=2.75] (3.center) to (6.center);
		\draw (9.center) to (8.center);
		\draw [in=105, out=75, looseness=1.50] (1.center) to (11.center);
		\draw [in=120, out=75, looseness=1.50] (0.center) to (12.center);
		\draw (13.center) to (14.center);
	\end{pgfonlayer}
\end{tikzpicture}, and \begin{tikzpicture}[scale=0.5, baseline=45]
	\begin{pgfonlayer}{nodelayer}
		\node   (0) at (-1, 7) {};
		\node   (1) at (0, 7) {};
		\node   (2) at (1.5, 7) {$\cdots$};
		\node   (3) at (3, 7) {};
		\node   (4) at (4, 7) {};
		\node   (5) at (5, 7) {};
		\node   (6) at (4.5, 6) {};
		\node   (8) at (3.75, 4) {};
		\node   (9) at (3.75, 5) {};
		\node   (11) at (3, 3) {};
		\node   (12) at (2.25, 2) {};
		\node   (13) at (1.25, 0) {};
		\node   (14) at (1.25, 1) {};
		\node   (17) at (-1, 7.5) {$1$};
		\node   (18) at (0, 7.5) {$1$};
		\node   (19) at (1.5, 7.5) {$\cdots$};
		\node   (20) at (3, 7.5) {$1$};
		\node   (21) at (4, 7.5) {$1$};
		\node   (22) at (5, 7.5) {$1$};
		\node   (23) at (3.25, 3.25) {$\cdot$};
		\node   (24) at (3.5, 3.5) {$\cdot$};
		\node   (25) at (3.75, 3.75) {$\cdot$};
		\node   (26) at (2.85, 1.5) {$b_{k-1}$};
		\node   (28) at (4.25, 4.5) {$b_2$};
		\node   (29) at (5, 5.5) {$b_{1}$};
		\node   (30) at (1.25, -0.5) {$b$};
		\node   (31) at (-2, 3.5) {$d_{\mathbf{p}}:=$};
	\end{pgfonlayer}
	\begin{pgfonlayer}{edgelayer}
		\draw [bend right=90, looseness=3.25] (4.center) to (5.center);
		\draw [in=-105, out=-90, looseness=2.75] (3.center) to (6.center);
		\draw (9.center) to (8.center);
		\draw [in=-105, out=-75, looseness=1.50] (1.center) to (11.center);
		\draw [in=-120, out=-75, looseness=1.50] (0.center) to (12.center);
		\draw (13.center) to (14.center);
	\end{pgfonlayer}
\end{tikzpicture}
\end{center}

\begin{lemma} As morphisms in $\mathbf{Dgrams}_{R(V,G)}$
$$\sum\limits_{b\in I_{G}}\sum\limits_{\mathbf{p}\in P(1,b)_{k}} d_{\mathbf{p}}\circ u_{\mathbf{p}}=\operatorname{id}_{1^{\otimes k}}$$ for all $k\in\NN_{\geq 2}$.
\end{lemma}

\begin{proof}
We proceed by induction on $k$.  For $k=2$, the statement is precisely the second relation in Definition \ref{def:DgramsRVG}.  Now let us suppose that $$\sum\limits_{b\in I_{G}}\sum\limits_{\mathbf{p}\in P(1,b)_{k}} d_{\mathbf{p}}\circ u_{\mathbf{p}}=\operatorname{id}_{1^{\otimes k}}$$ for some $k\geq2$.  Using this hypothesis, we have

\begin{align*}
\operatorname{id}_{1^{\otimes (k+1)}}&=\sum\limits_{b\in I_G}\sum\limits_{\mathbf{p}\in P(1,b)_{k}} \operatorname{id}_{1}\otimes (d_{\mathbf{p}}\circ u_{\mathbf{p}})\\
&=\sum\limits_{b\in I_G}\sum\limits_{\mathbf{p}\in P(1,b)_{k}} \left(\sum\limits_{c\rightarrow b} \operatorname{id}_{1}\otimes d_{\mathbf{p}} \circ\begin{tikzpicture}[scale=0.5, baseline=5]
	\begin{pgfonlayer}{nodelayer}
		\node   (0) at (0, 0) {};
		\node   (1) at (0, 1) {};
		\node   (2) at (-0.5, 2) {};
		\node   (3) at (0.5, 2) {};
		\node   (4) at (-0.5, -1) {};
		\node   (5) at (0.5, -1) {};
		\node   (6) at (-0.5, 2.5) {$1$};
		\node   (7) at (0.5, 2.5) {$b$};
		\node   (8) at (-0.5, -1.5) {$1$};
		\node   (9) at (0.5, -1.5) {$b$};
		\node   (10) at (0.25, 0.5) {$c$};
	\end{pgfonlayer}
	\begin{pgfonlayer}{edgelayer}
		\draw [bend right=90, looseness=3.25] (2.center) to (3.center);
		\draw [bend left=90, looseness=3.25] (4.center) to (5.center);
		\draw (1.center) to (0.center);
	\end{pgfonlayer}
\end{tikzpicture} \circ \operatorname{id}_{1}\otimes u_{\mathbf{p}}\right)
\end{align*}
\[
=\mathlarger{\mathlarger{\sum}}\limits_{c\in I_{G}}\ \mathlarger{\mathlarger{\sum}}\limits_{\mathbf{p}\in P(1,c)_{k+1}} d_{\mathbf{p}}\circ u_{\mathbf{p}},
\] which was to be shown.
\end{proof}

The following gives an example of the construction of a diagrammatic category in this way.

\begin{example}
Recall the representation graph $R\left(T^{(1)},\mathbf{T}\right)$ from (\ref{eqn:RepGraphT}).  Then we can construct the $\CC$-linear monoidal category $\mathbf{Dgrams}_{R\left(T^{(1)},\mathbf{T}\right)}$ be  with objects generated by $k\in I_{\mathbf{T}}$ and morphisms generated and related in the same manner as in Definition \ref{def:DgramsRVG}.
\end{example}


\subsection{The Functor \texorpdfstring{$\mathcal{H}_{R(V,G)}$}{HRVG}}


The following definitions and theorems show that there is a full functor from $\mathbf{Dgrams}_{R(V,G)}$ onto $G\text{-}\textbf{mod}_{\text{irr}}$.  Recall the maps $m_{1\ a}^{\ b}$ and $s^{1\ a}_{\ b}$ given before (\ref{eqn:msidempotents}).

\begin{definition}
\label{def:HRVG}
We let $\mathcal{H}_{R(V,G)}:\mathbf{Dgrams}_{R(V,G)}\longrightarrow G\text{-}\textbf{mod}_{\text{irr}}$ be the monoidal $\CC$-linear functor determined by the following rules:

\begin{center}
\begin{align*}
 \mathcal{H}_{R(V,G)}(a)=G^{(a)}\text{ for }a\in  I_{G},\\
  \mathcal{H}_{R(V,G)}\left(\begin{tikzpicture}[scale=0.5,baseline=10]
	\begin{pgfonlayer}{nodelayer}
		\node  (0) at (-1, 0) {};
		\node   (1) at (0, 0) {};
		\node   (2) at (-0.5, 1) {};
		\node   (3) at (-0.5, 2) {};
		\node   (4) at (-1, -0.5) {$1$};
		\node   (5) at (0, -0.5) {$a$};
		\node   (6) at (-0.5, 2.5) {$b$};
	\end{pgfonlayer}
	\begin{pgfonlayer}{edgelayer}
		\draw [bend left=90, looseness=3.25] (0.center) to (1.center);
		\draw (3.center) to (2.center);
	\end{pgfonlayer}
\end{tikzpicture}\right)=m_{1\ a}^{\ b},\\
\text{and } \mathcal{H}_{R(V,G)}\left(\begin{tikzpicture}[scale=0.5,baseline=10] 
	\begin{pgfonlayer}{nodelayer}
		\node   (0) at (-1, 2) {};
		\node   (1) at (0, 2) {};
		\node   (2) at (-0.5, 1) {};
		\node   (3) at (-0.5, 0) {};
		\node   (4) at (-1, 2.5) {$1$};
		\node   (5) at (0, 2.5) {$a$};
		\node   (6) at (-0.5, -0.5) {$b$};
	\end{pgfonlayer}
	\begin{pgfonlayer}{edgelayer}
		\draw [bend right=90, looseness=3.25] (0.center) to (1.center);
		\draw (3.center) to (2.center);
	\end{pgfonlayer}
\end{tikzpicture}\right)=s_{\ a}^{1\ b}.
\end{align*}
\end{center}
\end{definition}

Note that by the way $m_{1\ a}^{\ b}$ and $s_{\ a}^{1\ b}$ were chosen, the relations in (\ref{eqn:DgramPop}) are automatically satisfied.  If the reader is exploring graphs with multiple edges, this definition must be expanded.

As there will be no confusion as to which representation graph, for the rest of this section we will suppress the $R(V,G)$ in the notation of Definitions \ref{def:DgramsRVG} and \ref{def:HRVG} and say that $$\mathbf{Dgrams}:=\mathbf{Dgrams}_{R(V,G)}\text{ and }\mathcal{H}:=\mathcal{H}_{R(V,G)}.$$

\begin{lemma}
The functor $ \mathcal{H}$ is full onto $\operatorname{Hom}_{G}\left(\left(G^{(1)}\right)^{\otimes k}, G^{(b)}\right)$ and 

\noindent $\operatorname{Hom}_{G}\left(G^{(b)},\left(G^{(1)}\right)^{\otimes k}\right)$ for any $k\in \mathbb{N}$ and any $b\in I_{G}$.
\end{lemma}

\begin{proof}

To prove that $ \mathcal{H}$ is full onto Hom$_{G}\left(\left(G^{(1)}\right)^{\otimes k}, G^{(b)}\right)$, it suffices to show that Hom$_{G}\left((G^{(1)})^{\otimes k},G^{(b)}\right)$ is spanned by 

\begin{equation*}
B_k^b:=\left\{\begin{aligned}
 \mathcal{H}\left(u_{\mathbf{p}}\right)=:\pi_{\mathbf{p}}\ \ \Bigg\vert \ \mathbf{p} \in P(1,b)_k
\end{aligned}\right\},
\end{equation*}

\noindent which we will show by inducting on $k$.

Since the representation graph of $G$ does not contain any multiple edges, then up to scaling $$m_{1\ b}^{\ c}:G^{(1)}\otimes G^{(b)}\longrightarrow G^{(c)}$$ is canonical for all $G^{(c)}\subset G^{(1)}\otimes G^{(b)}$.  Thus each Hom$_{G}\left(G^{(1)}\otimes G^{(b)}, G^{(c)}\right)$ is either $0$ or spanned by $m_{1\ b}^{\ c}$.  Furthermore, since $G^{(1)}$ is simple, and $$ \mathcal{H}\left(\begin{tikzpicture}[scale=0.5, baseline=10]
	\begin{pgfonlayer}{nodelayer}
		\node   (2) at (-0.5, 2) {};
		\node   (3) at (-0.5, 0) {};
		\node   (5) at (-0.5, 2.5) {$a$};
		\node   (6) at (-0.5, -0.5) {$a$};
	\end{pgfonlayer}
	\begin{pgfonlayer}{edgelayer}
		\draw (3.center) to (2.center);
	\end{pgfonlayer}
\end{tikzpicture}\right) = \operatorname{id}_{G^{(1)}},$$ the base case is trivial.

Now suppose that Hom$_{G}\left((G^{(1)})^{\otimes k},G^{(b)}\right)$ is spanned by $D_k^b$ for some $k$ and for all $b$.  Then we consider Hom$_{G}\left((G^{(1)})^{\otimes (k+1)},G^{(c)}\right)$.

Since $G^{\otimes (k+1)}=G^{(1)}\otimes (G^{(1)})^{\otimes k}$, we can construct $\pi_{\mathbf{p}}=m_{1\ b}^{\ c} \circ \left(\operatorname{id}_{G^{(1)}}\otimes \pi_{\mathbf{q}}\right)$, where $\mathbf{p}\in P(1,c)_{k+1}$ and $\mathbf{q}\in P(1,b)_k$.  So up to scaling, we have morphisms 

\[
(G^{(1)})^{\otimes (k+1)}\xrightarrow{\operatorname{id}_{G^{(1)}}\otimes \pi_{\mathbf{q}}}G^{(1)}\otimes G^{(b)}\xrightarrow{m_{1\ b}^{\ c}}G^{(c)}
\]

\noindent which are canonically based on the path in the representation graph.  Thus for each $G^{(c)}\subset (G^{(1)})^{\otimes (k+1)}$, there is a canonical projection $\pi_{\mathbf{p}}$.  Therefore, $ \mathcal{H}$ is full on Hom$_{G}\left((G^{(1)})^{\otimes k},G^{(b)}\right)$ for all $k\in\mathbb{N}$ and $b\in I$.

It is analogously shown using $d_{\mathbf{p}}$, $\iota_{\mathbf{p}}$, and $s_{\ c}^{1\ b}$ where $\mathbf{p}\in P(1,b)_k$ that $ \mathcal{H}$ is full onto Hom$_{G}\left(G^{(b)},\left(G^{(1)}\right)^{\otimes k}\right)$.

\end{proof}

As the reader will have no doubt noticed, the proof of the previous lemma marks the first place where our representation graph is required to contain no multiple edges.  In the following lemmas, we use this as well.  We construct projection and inclusion maps which, in our case, are canonical but will not be for a representation graph which contains multiple edges.  

\begin{lemma}
The functor $ \mathcal{H}$ is full onto $$\operatorname{Hom}_{G}\left(\bigotimes\limits_{i=1}^{n} G^{(a_i)}, G^{(b)}\right)\text{ and }\operatorname{Hom}_{G}\left(G^{b}, \bigotimes\limits_{i=1}^{n} G^{(a_i)}\right)$$ for $a_i,b\in I$.
\end{lemma}

\begin{proof}
By the previous lemma and the fact that $ \mathcal{H}$ is a monoidal, $\CC$-linear functor, it suffices to show that any morphism in Hom$_{G}\left(\bigotimes\limits_{i=1}^{n} G^{(a_i)}, G^{(b)}\right)$ can be realized through $\CC$-linearity, composition, and tensor products of morphisms in Hom$_{G}\left(\left(G^{(1)}\right)^{\otimes k}, G^{(b)}\right)$ and Hom$_{G}\left(G^{(b)}, \left(G^{(1)}\right)^{\otimes k}\right)$.  

Let $f\in\operatorname{Hom}_{G}\left(\bigotimes\limits_{i=1}^{n} G^{(a_i)}, G^{(b)}\right)$.  For each $a_i$, there exists a minimal $k_i$ such that $G^{(a_i)}\subset (G^{(1)})^{\otimes k_i}$, and thus there are morphisms, $\pi_{k_{a_i}}^{a_i}\in$Hom$_{G}\left(\left(G^{(1)}\right)^{\otimes k_{a_i}}, G^{(a_i)}\right)$ and $\iota_{a_i}^{k_{a_i}}\in$Hom$_{G}\left(G^{(b)}, \left(G^{(1)}\right)^{\otimes k_{a_i}}\right)$ such that $\pi_{k_{a_i}}^{a_i}\circ\iota_{a_i}^{k_{a_i}}=\operatorname{id}_{G^{(a_i)}}$.  Thus, 

\[
f=f \circ \left(\bigotimes\limits_{i=1}^{n}\pi_{k_{a_i}}^{a_i}\circ\bigotimes\limits_{i=1}^{n}\iota_{a_i}^{k_{a_i}}\right)=\left(f\circ \bigotimes\limits_{i=1}^{n}\pi_{k_{a_i}}^{a_i}\right)\circ\bigotimes\limits_{i=1}^{n}\iota_{a_i}^{k_{a_i}},
\]

\noindent and since $$\left(f\circ \bigotimes\limits_{i=1}^{n}\pi_{k_{a_i}}^{a_i}\right)\in\operatorname{Hom}_{G}\left(\left(G^{(1)}\right)^{\otimes k_{a_i}}, G^{(b)}\right)\text{ and }\iota_{a_i}^{k_{a_i}}\in\operatorname{Hom}_{G}\left(G^{(a_i)}, \left(G^{(1)}\right)^{\otimes k_{a_i}}\right),$$ then $G^{\text{irr}}$ is full onto Hom$_{G}\left(\bigotimes\limits_{i=1}^{n} G^{(a_i)}, G^{(b)}\right)$ for $a_i,b\in I$.

An analogous argument shows $ \mathcal{H}$ is full onto Hom$_{G}\left(G^{b}, \bigotimes\limits_{i=1}^{n} G^{(a_i)}\right)$.
\end{proof}

\begin{lemma}
The functor $ \mathcal{H}$ is full onto Hom$_{G}\left(\left(G^{(1)}\right)^{\otimes k}, \left(G^{(1)}\right)^{\otimes \ell}\right)$ for any $k, \ell\in \mathbb{N}$.
\end{lemma}

\begin{proof}
Given a morphism, $f\in \text{Hom}_{G}\left(\left(G^{(1)}\right)^{\otimes k}, \left(G^{(1)}\right)^{\otimes \ell}\right)$ and a path $\mathbf{p}\in P(1,b)_k$, we have in the image of $ \mathcal{H}$ canonical projections, $\pi_{\mathbf{p}}$, and inclusions, $\iota_{\mathbf{p}}$, onto and from $G^{(b)}$ such that $\pi_{\mathbf{p}}\circ \iota_{\mathbf{p}}=\operatorname{id}_{G^{(b)}}$ and $\sum\limits_{\mathbf{p}\in P(1,b)_k}\iota_{\mathbf{p}}\circ \pi_{\mathbf{p}}=\operatorname{id}_{\left(G^{\otimes k}\right)}$.  Thus, we have

\begin{align*}
f&=\left(\sum\limits_{\mathbf{p}\in P(1,b)_{\ell}}\iota_{\mathbf{p}}\circ \pi_{\mathbf{p}}\right) \circ f \circ \left(\sum\limits_{\mathbf{q}\in P(1,b)_k}\iota_{\mathbf{q}}\circ \pi_{\mathbf{q}}\right)\\
&=\left(\sum\limits_{\mathbf{p}\in P(1,b)_{\ell}}\sum\limits_{\mathbf{q}\in P(1,b)_k}\iota_{\mathbf{p}}\circ \left(\pi_{\mathbf{p}} \circ f \circ \iota_{\mathbf{q}} \circ \pi_{\mathbf{q}}\right)\right)
\end{align*}

\noindent where the sums are taken over all paths of length $k$ and $\ell$ from $1$ to $b$.  Since $\pi_{\mathbf{p}} \circ f \circ \iota_{\mathbf{q}} \circ \pi_{\mathbf{q}}\in\text{Hom}_{G}\left(\left(G^{(1)}\right)^{\otimes k},G^{(b)}\right)$ and $\iota_{\mathbf{p}}\in\text{Hom}_{G}\left(G^{(b)},\left(G^{(1)}\right)^{\otimes \ell}\right)$, $ \mathcal{H}$ is full onto Hom$_{G}\left(\left(G^{(1)}\right)^{\otimes k}, \left(G^{(1)}\right)^{\otimes \ell}\right)$ for any $k, \ell\in \mathbb{N}$.
\end{proof}

\begin{theorem}
\label{thm:FGirrFull}
The functor $ \mathcal{H}$ is full.
\end{theorem}

\begin{proof}
Consider a morphism $f\in \text{Hom}_{G}\left(\bigotimes\limits_{i=1}^{n}G^{(a_i)}, \bigotimes\limits_{j=1}^{m}G^{(b_j)}\right)$.  Using notation from the proofs above, we have 

\begin{align*}
f&=\bigotimes\limits_{j=1}^{m}\left(\pi_{\ell_{b_j}}^{b_j}\circ \iota_{b_j}^{\ell_{b_j}}\right)\circ f \circ \bigotimes\limits_{i=1}^{n}\left(\pi_{k_{a_i}}^{a_i}\circ \iota_{a_i}^{k_{a_i}}\right)\\
&=\bigotimes\limits_{j=1}^{m}\pi_{\ell_{b_j}}^{b_j}\circ \left(\bigotimes\limits_{j=1}^{m}\iota_{b_j}^{\ell_{b_j}} \circ f \circ \bigotimes\limits_{i=1}^{n} \pi_{k_{a_i}}^{a_i}\right) \circ \bigotimes\limits_{i=1}^{n} \iota_{a_i}^{k_{a_i}},
\end{align*}  

\noindent and by setting $k=\sum\limits_{i=1}^{n}k_{a_i}$ and $\ell=\sum\limits_{j=1}^{m}\ell_{b_j}$, we have that $$\left(\bigotimes\limits_{j=1}^{m}\iota_{b_j}^{\ell_{b_j}} \circ f \circ \bigotimes\limits_{i=1}^{n} \pi_{k_{a_i}}^{a_i}\right)\in\text{Hom}_{G}\left(\left(G^{(1)}\right)^{\otimes k},\left(G^{(1)}\right)^{\otimes \ell}\right),$$ $$\pi_{\ell_{b_j}}^{b_j}\in\text{Hom}_{G}\left(\left(G^{(1)}\right)^{\otimes \ell_{b_j}},G^{(b_j)}\right),$$ and $$\iota_{a_i}^{k_{a_i}}\in\text{Hom}_{G}\left(G^{(a_i)},\left(G^{(1)}\right)^{\otimes k_{a_i}}\right).$$  Therefore the functor $ \mathcal{H}$ is full.
\end{proof}


\subsection{The Induced Functor \texorpdfstring{$\overline{\mathcal{H}}_{R(V,G)}$}{H-RVG}}


We now explore the kernel of $ \mathcal{H}$.  Assume $\mathcal{I}$ is a tensor ideal of $\mathbf{Dgrams}$ such that for all objects $X,Y$ in $\mathbf{Dgrams}$, $\mathcal{H}(f)=0$ for every morphism $f\in \mathcal{I}\left(X,Y\right)$.  Let $\overline{\mathbf{Dgrams}}:=\bigslant{\mathbf{Dgrams}}{\mathcal{I}}$.  Then there is an induced functor $$\overline{\mathcal{H}}:\overline{\mathbf{Dgrams}}\longrightarrow G\text{-}\textbf{mod}.$$ Let us assume that for any $a, b \in I_G$ $$\operatorname{Hom}_{\overline{\mathbf{Dgrams}}}\left(a,b\right)=\begin{cases}
 \CC\cdot \operatorname{id}_a & a=b\\
0 & a\neq b
\end{cases}.$$  That is, in $\overline{\mathbf{Dgrams}}$ we have for all $a,b \in I_G$

\begin{equation}
\label{eqn:CatSchur}
\begin{tikzpicture}[scale=0.5, baseline=-2]
	\begin{pgfonlayer}{nodelayer}
		\node   (0) at (-1, 1) {};
		\node   (1) at (1, 1) {};
		\node   (2) at (-1, -1) {};
		\node   (3) at (1, -1) {};
		\node   (4) at (0, -1) {};
		\node   (5) at (0, -2) {};
		\node   (6) at (0, 1) {};
		\node   (7) at (0, 2) {};
		\node   (9) at (4.5, 1) {};
		\node   (10) at (4.5, -1) {};
		\node   (11) at (1.5, 0) {$=$};
		\node   (12) at (3.25, 0) {$\delta_{a,b}\alpha_{d}$};
		\node   (13) at (0, 0) {$d$};
		\node   (14) at (0, -2.5) {$a$};
		\node   (15) at (0, 2.5) {$b$};
		\node   (16) at (4.5, 1.5) {$a$};
		\node   (17) at (4.5, -1.5) {$a$};
	\end{pgfonlayer}
	\begin{pgfonlayer}{edgelayer}
		\draw (0.center) to (2.center);
		\draw (2.center) to (3.center);
		\draw (3.center) to (1.center);
		\draw (1.center) to (0.center);
		\draw (6.center) to (7.center);
		\draw (4.center) to (5.center);
		\draw (9.center) to (10.center);
	\end{pgfonlayer}
\end{tikzpicture}
\end{equation}

\noindent where $\delta_{a, b}$ is the Kronecker delta and $\alpha_{d}\in\CC$.

\begin{lemma}
Suppose the equality in \ref{eqn:CatSchur} is satisfied.  The functor $ \overline{\mathcal{H}}$ is faithful on Hom$_{\overline{\mathbf{Dgrams}}}\left(1^{\otimes k}, b\right)$ and Hom$_{\overline{\mathbf{Dgrams}}}\left(b, 1^{\otimes k}\right)$ for all $b\in I_{G}$ and $k\in \mathbb{N}$.
\end{lemma}

\begin{proof} Recall that $ \mathcal{H}$ is full, the set $\left\{ \pi_{^1p^b_k} | ^1p^b_k\in ^1P^b_k\right\}$ forms a basis for $$\operatorname{Hom}_{G}\left(\left(G^{(1)}\right)^{\otimes k}, \left(G^{(1)}\right)^{\otimes \ell}\right),$$ and $ \mathcal{H}\left(u_{^1p^b_k}\right)=\pi_{^1p^b_k}$.  Thus, the $u_{^1p^b_k}$ are linearly independent.

It then suffices to show that any diagram in Hom$_{\overline{\mathbf{Dgrams}}}\left(1^{\otimes k},b\right)$ can be written as a linear combination of the diagrams $u_{\mathbf{p}}$ where $\mathbf{p}\in P(1,b)_k$.  It will be convenient to instead show the following equality:

\begin{center}
\begin{tikzpicture}[scale=0.5, baseline=-10]
	\begin{pgfonlayer}{nodelayer}
		\node   (0) at (-1, 1) {};
		\node   (1) at (-1, -1) {};
		\node   (2) at (2, -1) {};
		\node   (3) at (2, 1) {};
		\node   (4) at (0.5, 1) {};
		\node   (5) at (0.5, 2) {};
		\node   (6) at (-0.75, -1) {};
		\node   (7) at (-0.25, -1) {};
		\node   (8) at (-0.75, -2) {};
		\node   (9) at (-0.25, -2) {};
		\node   (10) at (0.5, -1.5) {$\cdots$};
		\node   (11) at (1.25, -1) {};
		\node   (12) at (1.75, -1) {};
		\node   (13) at (1.25, -2) {};
		\node   (14) at (1.75, -2) {};
		\node   (15) at (-0.75, -2.5) {$1$};
		\node   (16) at (-0.25, -2.5) {$1$};
		\node   (17) at (1.25, -2.5) {$1$};
		\node   (18) at (1.75, -2.5) {$a$};
		\node   (19) at (0.5, 2.5) {$b$};
		\node   (20) at (0.5, 0) {$D$};
	\end{pgfonlayer}
	\begin{pgfonlayer}{edgelayer}
		\draw (0.center) to (1.center);
		\draw (1.center) to (2.center);
		\draw (2.center) to (3.center);
		\draw (3.center) to (0.center);
		\draw (6.center) to (8.center);
		\draw (7.center) to (9.center);
		\draw (11.center) to (13.center);
		\draw (12.center) to (14.center);
		\draw (5.center) to (4.center);
	\end{pgfonlayer}
\end{tikzpicture} $=\sum\limits_{\mathbf{p}\in P(a,b)_k}\alpha_{\mathbf{p}}u_{\mathbf{p}}$
\end{center}

\noindent where $D$ is a diagram in $\operatorname{Hom}_{\overline{\mathbf{Dgrams}}}\left(1^{\otimes k}\otimes a,b\right)$.

We induct on $k$.  For $k=0$, an immediate consequence of the relation (\ref{eqn:CatSchur}) is that any diagram is either the identity on $a$, or it is $0$.  For $k=1$, the second relation in (\ref{eqn:DgramPop}) results in the following:

\begin{center}
\begin{tikzpicture}[scale=0.5, baseline=0]
	\begin{pgfonlayer}{nodelayer}
		\node   (0) at (0, 1) {};
		\node   (1) at (0, -1) {};
		\node   (2) at (2, -1) {};
		\node   (3) at (2, 1) {};
		\node   (4) at (1, 1) {};
		\node   (5) at (1, 2) {};
		\node   (6) at (0.5, -1) {};
		\node   (8) at (0.5, -2) {};
		\node   (12) at (1.5, -1) {};
		\node   (14) at (1.5, -2) {};
		\node   (15) at (0.5, -2.5) {$1$};
		\node   (18) at (1.5, -2.5) {$1$};
		\node   (19) at (1, 2.5) {$b$};
		\node   (20) at (1, 0) {$D$};
		\node   (21) at (3, 0) {$=$};
		\node   (22) at (7, 2.5) {};
		\node   (23) at (7, 0.5) {};
		\node   (24) at (9, 0.5) {};
		\node   (25) at (9, 2.5) {};
		\node   (26) at (8, 2.5) {};
		\node   (27) at (8, 3.5) {};
		\node   (28) at (7.5, 0.5) {};
		\node   (29) at (7.5, -0.5) {};
		\node   (30) at (8.5, 0.5) {};
		\node   (31) at (8.5, -0.5) {};
		\node   (32) at (7.5, -4) {$1$};
		\node   (33) at (8.5, -4) {$a$};
		\node   (34) at (8, 4) {$b$};
		\node   (35) at (8, 1.5) {$D$};
		\node   (51) at (8, -1.5) {};
		\node   (52) at (8, -2.5) {};
		\node   (53) at (7.5, -3.5) {};
		\node   (54) at (8.5, -3.5) {};
		\node   (57) at (8.5, -2) {$c$};
		\node   (58) at (5, 0) {$\mathlarger{\mathlarger{\sum}}\limits_{c\rightarrow a}$};
	\end{pgfonlayer}
	\begin{pgfonlayer}{edgelayer}
		\draw (0.center) to (1.center);
		\draw (1.center) to (2.center);
		\draw (2.center) to (3.center);
		\draw (3.center) to (0.center);
		\draw (6.center) to (8.center);
		\draw (12.center) to (14.center);
		\draw (5.center) to (4.center);
		\draw (22.center) to (23.center);
		\draw (23.center) to (24.center);
		\draw (24.center) to (25.center);
		\draw (25.center) to (22.center);
		\draw (28.center) to (29.center);
		\draw (30.center) to (31.center);
		\draw (27.center) to (26.center);
		\draw [bend right=90, looseness=3.50] (29.center) to (31.center);
		\draw (51.center) to (52.center);
		\draw [bend left=90, looseness=3.25] (53.center) to (54.center);
	\end{pgfonlayer}
\end{tikzpicture} $=\sum\limits_{c\rightarrow a}\delta_{c,b}\alpha_{D,c} $\begin{tikzpicture}[scale=0.5, baseline=10]
	\begin{pgfonlayer}{nodelayer}
		\node   (0) at (0, 0) {};
		\node   (1) at (1, 0) {};
		\node   (2) at (0.5, 1) {};
		\node   (3) at (0.5, 2) {};
		\node   (4) at (0, -0.5) {$1$};
		\node   (5) at (1, -0.5) {$a$};
		\node   (6) at (0.5, 2.5) {$c$};
	\end{pgfonlayer}
	\begin{pgfonlayer}{edgelayer}
		\draw [bend left=90, looseness=3.50] (0.center) to (1.center);
		\draw (2.center) to (3.center);
	\end{pgfonlayer}
\end{tikzpicture},
\end{center}

\noindent which was to be shown.

Now, suppose that any diagram \noindent in Hom$_{\overline{\mathbf{Dgrams}}}\left(1^{\otimes k}\otimes a,b\right)$ can be written as a linear combination of the diagrams $u_{\mathbf{p}}$ where $\mathbf{p}\in P(a,b)_k$, and let the following diagram be a diagram in Hom$_{\overline{\mathbf{Dgrams}}}\left(1^{\otimes (k+1)}\otimes a,b\right)$ for some $b\in I_{G}$.  Using the second relation in (\ref{eqn:DgramPop}), we get that

\begin{center}
\begin{tikzpicture}[scale=0.5]
	\begin{pgfonlayer}{nodelayer}
		\node   (0) at (-1, 1) {};
		\node   (1) at (-1, -1) {};
		\node   (2) at (2, -1) {};
		\node   (3) at (2, 1) {};
		\node   (4) at (0.5, 1) {};
		\node   (5) at (0.5, 2) {};
		\node   (6) at (-0.75, -1) {};
		\node   (7) at (-0.25, -1) {};
		\node   (8) at (-0.75, -2) {};
		\node   (9) at (-0.25, -2) {};
		\node   (10) at (0.5, -1.5) {$\cdots$};
		\node   (11) at (1.25, -1) {};
		\node   (12) at (1.75, -1) {};
		\node   (13) at (1.25, -2) {};
		\node   (14) at (1.75, -2) {};
		\node   (15) at (-0.75, -2.5) {$1$};
		\node   (16) at (-0.25, -2.5) {$1$};
		\node   (17) at (1.25, -2.5) {$1$};
		\node   (18) at (1.75, -2.5) {$a$};
		\node   (19) at (0.5, 2.5) {$b$};
		\node   (20) at (0.5, 0) {$D$};
		\node   (21) at (3, 0) {$=$};
		\node   (22) at (7, 2.5) {};
		\node   (23) at (7, 0.5) {};
		\node   (24) at (10, 0.5) {};
		\node   (25) at (10, 2.5) {};
		\node   (26) at (8.5, 2.5) {};
		\node   (27) at (8.5, 3.5) {};
		\node   (28) at (7.25, 0.5) {};
		\node   (29) at (7.75, 0.5) {};
		\node   (30) at (7.25, -3.5) {};
		\node   (31) at (7.75, -3.5) {};
		\node   (32) at (8.5, -2) {$\cdots$};
		\node   (33) at (9.25, 0.5) {};
		\node   (34) at (9.75, 0.5) {};
		\node   (35) at (9.25, -0.5) {};
		\node   (36) at (9.75, -0.5) {};
		\node   (37) at (7.25, -4) {$1$};
		\node   (38) at (7.75, -4) {$1$};
		\node   (39) at (9.25, -4) {$1$};
		\node   (40) at (9.75, -4) {$a$};
		\node   (41) at (8.5, 4) {$b$};
		\node   (42) at (8.5, 1.5) {$D$};
		\node   (43) at (5, 0) {$\sum\limits_{c-a}$};
		\node   (65) at (9.5, -1.5) {};
		\node   (66) at (9.5, -2.5) {};
		\node   (67) at (9.25, -3.5) {};
		\node   (68) at (9.75, -3.5) {};
		\node   (75) at (10, -2) {$c$};
	\end{pgfonlayer}
	\begin{pgfonlayer}{edgelayer}
		\draw (0.center) to (1.center);
		\draw (1.center) to (2.center);
		\draw (2.center) to (3.center);
		\draw (3.center) to (0.center);
		\draw (6.center) to (8.center);
		\draw (7.center) to (9.center);
		\draw (11.center) to (13.center);
		\draw (12.center) to (14.center);
		\draw (5.center) to (4.center);
		\draw (22.center) to (23.center);
		\draw (23.center) to (24.center);
		\draw (24.center) to (25.center);
		\draw (25.center) to (22.center);
		\draw (28.center) to (30.center);
		\draw (29.center) to (31.center);
		\draw (33.center) to (35.center);
		\draw (34.center) to (36.center);
		\draw (27.center) to (26.center);
		\draw [bend right=90, looseness=6.75] (35.center) to (36.center);
		\draw (65.center) to (66.center);
		\draw [bend left=90, looseness=6.75] (67.center) to (68.center);
	\end{pgfonlayer}
\end{tikzpicture}
\end{center}

Now we set 

\begin{center}
\begin{tikzpicture}[scale=0.5]
	\begin{pgfonlayer}{nodelayer}
		\node   (0) at (-1, 1) {};
		\node   (1) at (-1, -1) {};
		\node   (2) at (2, -1) {};
		\node   (3) at (2, 1) {};
		\node   (4) at (0.5, 1) {};
		\node   (5) at (0.5, 2) {};
		\node   (6) at (-0.75, -1) {};
		\node   (7) at (-0.25, -1) {};
		\node   (8) at (-0.75, -2) {};
		\node   (9) at (-0.25, -2) {};
		\node   (10) at (0.5, -1.5) {$\cdots$};
		\node   (11) at (1.25, -1) {};
		\node   (12) at (1.75, -1) {};
		\node   (13) at (1.25, -2) {};
		\node   (14) at (1.75, -2) {};
		\node   (15) at (-0.75, -2.5) {$1$};
		\node   (16) at (-0.25, -2.5) {$1$};
		\node   (17) at (1.25, -2.5) {$1$};
		\node   (18) at (1.75, -2.5) {$c$};
		\node   (19) at (0.5, 2.5) {$b$};
		\node   (20) at (0.5, 0) {$D_c$};
		\node   (21) at (3, 0) {$=$};
		\node   (22) at (4, 1) {};
		\node   (23) at (4, -1) {};
		\node   (24) at (7, -1) {};
		\node   (25) at (7, 1) {};
		\node   (26) at (5.5, 1) {};
		\node   (27) at (5.5, 2) {};
		\node   (28) at (4.25, -1) {};
		\node   (29) at (4.75, -1) {};
		\node   (30) at (4.25, -4) {};
		\node   (31) at (4.75, -4) {};
		\node   (32) at (5.5, -1.5) {$\cdots$};
		\node   (33) at (6.25, -1) {};
		\node   (34) at (6.75, -1) {};
		\node   (35) at (6.25, -2) {};
		\node   (36) at (6.75, -2) {};
		\node   (37) at (4.25, -4.5) {$1$};
		\node   (38) at (4.75, -4.5) {$1$};
		\node   (41) at (5.5, 2.5) {$b$};
		\node   (42) at (5.5, 0) {$D$};
		\node   (65) at (6.5, -3) {};
		\node   (66) at (6.5, -4) {};
		\node   (75) at (6.5, -4.5) {$c$};
		\node   (76) at (6,-1) {};
		\node   (77) at (6,-4) {};
		\node   (78) at (6,-4.5) {$1$};
	\end{pgfonlayer}
	\begin{pgfonlayer}{edgelayer}
		\draw (0.center) to (1.center);
		\draw (1.center) to (2.center);
		\draw (2.center) to (3.center);
		\draw (3.center) to (0.center);
		\draw (6.center) to (8.center);
		\draw (7.center) to (9.center);
		\draw (11.center) to (13.center);
		\draw (12.center) to (14.center);
		\draw (5.center) to (4.center);
		\draw (22.center) to (23.center);
		\draw (23.center) to (24.center);
		\draw (24.center) to (25.center);
		\draw (25.center) to (22.center);
		\draw (28.center) to (30.center);
		\draw (29.center) to (31.center);
		\draw (33.center) to (35.center);
		\draw (34.center) to (36.center);
		\draw (27.center) to (26.center);
		\draw [bend right=90, looseness=6.75] (35.center) to (36.center);
		\draw (65.center) to (66.center);
		\draw (76.center) to (77.center);
	\end{pgfonlayer}
\end{tikzpicture}
\end{center}

\noindent resulting in 

\begin{center}
\begin{tikzpicture}[scale=0.5]
	\begin{pgfonlayer}{nodelayer}
		\node   (0) at (-1, 1) {};
		\node   (1) at (-1, -1) {};
		\node   (2) at (2, -1) {};
		\node   (3) at (2, 1) {};
		\node   (4) at (0.5, 1) {};
		\node   (5) at (0.5, 2) {};
		\node   (6) at (-0.75, -1) {};
		\node   (7) at (-0.25, -1) {};
		\node   (8) at (-0.75, -2) {};
		\node   (9) at (-0.25, -2) {};
		\node   (10) at (0.5, -1.5) {$\cdots$};
		\node   (11) at (1.25, -1) {};
		\node   (12) at (1.75, -1) {};
		\node   (13) at (1.25, -2) {};
		\node   (14) at (1.75, -2) {};
		\node   (15) at (-0.75, -2.5) {$1$};
		\node   (16) at (-0.25, -2.5) {$1$};
		\node   (17) at (1.25, -2.5) {$1$};
		\node   (18) at (1.75, -2.5) {$a$};
		\node   (19) at (0.5, 2.5) {$b$};
		\node   (20) at (0.5, 0) {$D$};
		\node   (21) at (3, 0) {$=$};
		\node   (22) at (6, 1.5) {};
		\node   (23) at (6, -0.5) {};
		\node   (24) at (9, -0.5) {};
		\node   (25) at (9, 1.5) {};
		\node   (26) at (7.5, 1.5) {};
		\node   (27) at (7.5, 2.5) {};
		\node   (28) at (6.25, -0.5) {};
		\node   (29) at (6.75, -0.5) {};
		\node   (30) at (6.25, -2.5) {};
		\node   (31) at (6.75, -2.5) {};
		\node   (32) at (7.5, -2) {$\cdots$};
		\node   (33) at (8, -0.5) {};
		\node   (34) at (8.75, -0.5) {};
		\node   (35) at (8, -2.5) {};
		\node   (36) at (8.75, -1.5) {};
		\node   (37) at (6.25, -3) {$1$};
		\node   (38) at (6.75, -3) {$1$};
		\node   (39) at (8.5, -3) {$1$};
		\node   (40) at (9, -3) {$a$};
		\node   (41) at (7.5, 3) {$b$};
		\node   (42) at (7.5, 0.5) {$D_c$};
		\node   (43) at (4.5, 0) {$\sum\limits_{c\rightarrow a}$};
		\node   (65) at (8.75, -1.5) {};
		\node   (67) at (8.5, -2.5) {};
		\node   (68) at (9, -2.5) {};
		\node   (75) at (9.25, -1.25) {$c$};
		\node   (76) at (8, -3) {$1$};
	\end{pgfonlayer}
	\begin{pgfonlayer}{edgelayer}
		\draw (0.center) to (1.center);
		\draw (1.center) to (2.center);
		\draw (2.center) to (3.center);
		\draw (3.center) to (0.center);
		\draw (6.center) to (8.center);
		\draw (7.center) to (9.center);
		\draw (11.center) to (13.center);
		\draw (12.center) to (14.center);
		\draw (5.center) to (4.center);
		\draw (22.center) to (23.center);
		\draw (23.center) to (24.center);
		\draw (24.center) to (25.center);
		\draw (25.center) to (22.center);
		\draw (28.center) to (30.center);
		\draw (29.center) to (31.center);
		\draw (33.center) to (35.center);
		\draw (34.center) to (36.center);
		\draw (27.center) to (26.center);
		\draw [bend left=90, looseness=6.75] (67.center) to (68.center);
	\end{pgfonlayer}
\end{tikzpicture}
\end{center}

\noindent and thus, by the induction hypothesis,

\begin{center}
\begin{tikzpicture}[scale=0.75]
	\begin{pgfonlayer}{nodelayer}
		\node   (0) at (-1, 1) {};
		\node   (1) at (-1, -1) {};
		\node   (2) at (2, -1) {};
		\node   (3) at (2, 1) {};
		\node   (4) at (0.5, 1) {};
		\node   (5) at (0.5, 2) {};
		\node   (6) at (-0.75, -1) {};
		\node   (7) at (-0.25, -1) {};
		\node   (8) at (-0.75, -2) {};
		\node   (9) at (-0.25, -2) {};
		\node   (10) at (0.5, -1.5) {$\cdots$};
		\node   (11) at (1.25, -1) {};
		\node   (12) at (1.75, -1) {};
		\node   (13) at (1.25, -2) {};
		\node   (14) at (1.75, -2) {};
		\node   (15) at (-0.75, -2.5) {$1$};
		\node   (16) at (-0.25, -2.5) {$1$};
		\node   (17) at (1.25, -2.5) {$1$};
		\node   (18) at (1.75, -2.5) {$a$};
		\node   (19) at (0.5, 2.5) {$b$};
		\node   (20) at (0.5, 0) {$D$};
		\node   (21) at (3, 0) {$=$};
		\node   (28) at (9, 1) {};
		\node   (29) at (9.5, 1) {};
		\node   (30) at (9, -1) {};
		\node   (31) at (9.5, -1) {};
		\node   (32) at (10.25, 0) {$\cdots$};
		\node   (33) at (10.75, 1) {};
		\node   (34) at (11.5, 1) {};
		\node   (35) at (10.75, -1) {};
		\node   (36) at (11.5, 0) {};
		\node   (37) at (9, -1.5) {$1$};
		\node   (38) at (9.5, -1.5) {$1$};
		\node   (39) at (11.25, -1.5) {$1$};
		\node   (40) at (11.75, -1.5) {$a$};
		\node   (65) at (11.5, 0) {};
		\node   (67) at (11.25, -1) {};
		\node   (68) at (11.75, -1) {};
		\node   (75) at (12, 0.25) {$c$};
		\node   (76) at (10.75, -1.5) {$1$};
		\node   (81) at (6, -0.25) {$\sum\limits_{c\rightarrow a}\sum\limits_{\mathbf{p}\in P(c,b)_k}\alpha_{\mathbf{p}}u_{\mathbf{p}}\circ $};
	\end{pgfonlayer}
	\begin{pgfonlayer}{edgelayer}
		\draw (0.center) to (1.center);
		\draw (1.center) to (2.center);
		\draw (2.center) to (3.center);
		\draw (3.center) to (0.center);
		\draw (6.center) to (8.center);
		\draw (7.center) to (9.center);
		\draw (11.center) to (13.center);
		\draw (12.center) to (14.center);
		\draw (5.center) to (4.center);
		\draw (28.center) to (30.center);
		\draw (29.center) to (31.center);
		\draw (33.center) to (35.center);
		\draw (34.center) to (36.center);
		\draw [bend left=90, looseness=6.75] (67.center) to (68.center);
	\end{pgfonlayer}
\end{tikzpicture}
\end{center}

\noindent which shows the desired result.  Therefore, the $ \overline{\mathcal{H}}$ is faithful on Hom$_{\overline{\mathbf{Dgrams}}}\left(1^{\otimes k}, b\right)$ for all $b\in I_{G}$.


By considering the vertical reflection of each diagram, the analogous argument shows that $ \overline{\mathcal{H}}$ is faithful from Hom$_{\overline{\mathbf{Dgrams}}}\left(b, 1^{\otimes k}\right)$.

\end{proof}

\begin{lemma}
The functor $ \overline{\mathcal{H}}$ is faithful on Hom$_{\overline{\mathbf{Dgrams}}}\left(1^{\otimes k}, 1^{\otimes \ell}\right)$ for all $k,\ell\in \mathbb{N}$.
\end{lemma}

\begin{proof} Let $D_k^{\ell}$ be the set of all diagrams in Hom$_{\overline{\mathbf{Dgrams}}}\left(1^{\otimes k},1^{\otimes \ell}\right)$.  Now suppose that $$ \overline{\mathcal{H}}\left(\sum\limits_{D\in D_k^{\ell}} \alpha_{D} \regDdiag\right)=0$$
where only finitely many of the $\alpha_D$ are non-zero.

Let $\mathbf{q}\in P(1,b)_{\ell}$.  Then $$0= \overline{\mathcal{H}}\left( u_{\mathbf{q}}\right)\circ \overline{\mathcal{H}}\left(\sum\limits_{D\in D_k^{\ell}} \alpha_{D} \regDdiag\right)$$
$$= \overline{\mathcal{H}}\left(\sum\limits_{D\in D_k^{\ell}} \alpha_{D} \left( u_{\mathbf{q}}\circ \regDdiag\right)\right).$$
Now notice that for any $b\in I_{G}$, $$u_{\mathbf{q}}\circ \regDdiag \in \operatorname{Hom}_{\overline{\mathbf{Dgrams}}}\left(1^{\otimes k},b\right),$$ and thus by the previous lemma, we have the following equalities:
$$0=\sum\limits_{D\in D_k^{\ell}} \alpha_{D} \left(u_{\mathbf{q}}\circ \regDdiag\right)=  \sum\limits_{D\in D_k^{\ell}} \alpha_{D} \left(d_{\mathbf{q}} \circ u_{\mathbf{q}}\circ \regDdiag\right)$$
$$= \sum\limits_{b\in I_{G}}\sum\limits_{\mathbf{p}\in P(1,b)_{\ell}}\sum\limits_{D\in D_k^{\ell}} \alpha_{D} \left(d_{\mathbf{p}} \circ u_{\mathbf{p}}\circ \regDdiag\right),$$

\noindent and since $\sum\limits_{\mathbf{p}\in P(1,b)_{\ell}} d_{\mathbf{p}}\circ u_{\mathbf{p}}=\operatorname{id}_{1^{\otimes \ell}}$, we have 
$$0 = \sum\limits_{D\in D_k^{\ell}} \alpha_{D} \regDdiag .$$

Therefore, $ \overline{\mathcal{H}}$ is faithful from Hom$_{\overline{\mathbf{Dgrams}}}\left(1^{\otimes k}, 1^{\otimes \ell}\right)$ for all $k,\ell\in \mathbb{N}$.
\end{proof}

\begin{theorem}
\label{thm:FGirrFaithful}
The functor $ \overline{\mathcal{H}}$ is faithful on $\overline{\mathbf{Dgrams}}$.
\end{theorem}

\begin{proof}
Let $D_{a_n}^{b_m}$ be the set of all diagrams in Hom$_{\overline{\mathbf{Dgrams}}}\left(\bigotimes\limits_{i=1}^n a_i,\bigotimes\limits_{j=1}^m b_j\right)$ where $a_i, b_j\in \operatorname{Obj}\left(\overline{\mathbf{Dgrams}}\right)$ for all $1\leq i\leq n$ and $1\leq j\leq n$.

Suppose $ \overline{\mathcal{H}}\left(\sum\limits_{d\in D_{a_n}^{b_m}}\alpha_d d\right)=0$.  Then we also have 

\[
\left(\iota_{\ell_{b_1}}^{b_1}\otimes\cdots\otimes \iota_{\ell_{b_m}}^{b_m}\right)\circ \overline{\mathcal{H}}\left(\sum\limits_{d\in D_{a_n}^{b_m}}\alpha_d d\right) \circ \left(\pi_{k_{a_1}}^{a_1}\otimes\cdots\otimes \pi_{k_{a_n}}^{a_n}\right)=0,
\]

\noindent and using the fact that $ \overline{\mathcal{H}}$ is a monoidal $\CC$-linear functor along side the lemmas above, we have the following string of equalities:

$$\left(\iota_{\ell_{b_1}}^{b_1}\otimes\cdots\otimes \iota_{\ell_{b_m}}^{b_m}\right)\circ \overline{\mathcal{H}}\left(\sum\limits_{d\in D_{a_n}^{b_m}}\alpha_d d\right) \circ \left(\pi_{k_{a_1}}^{a_1}\otimes\cdots\otimes \pi_{k_{a_n}}^{a_n}\right)$$ 
$$= \overline{\mathcal{H}}\left(\left(d_{\ell_{b_1}}^{b_1}\otimes\cdots\otimes d_{\ell_{b_m}}^{b_m}\right)\circ\sum\limits_{d\in D_{a_n}^{b_m}}\alpha_d d \circ \left(u_{k_{a_1}}^{a_1}\otimes\cdots\otimes u_{k_{a_n}}^{a_n}\right)\right)$$ 
$$= \overline{\mathcal{H}}\left(\sum\limits_{d\in D_{a_n}^{b_m}}\left(\alpha_d \left(d_{\ell_{b_1}}^{b_1}\otimes\cdots\otimes d_{\ell_{b_m}}^{b_m}\right)\circ d \circ \left(u_{k_{a_1}}^{a_1}\otimes\cdots\otimes u_{k_{a_n}}^{a_n}\right)\right)\right)=0.$$

\noindent Since $ \overline{\mathcal{H}}$ is faithful on Hom$_{\overline{\mathbf{Dgrams}}}\left(1^{\otimes k},1^{\otimes \ell}\right)$, then 
$$\sum\limits_{d\in D_{a_n}^{b_m}}\left(\alpha_d \left(d_{\ell_{b_1}}^{b_1}\otimes\cdots\otimes d_{\ell_{b_m}}^{b_m}\right)\circ d \circ \left(u_{k_{a_1}}^{a_1}\otimes\cdots\otimes u_{k_{a_n}}^{a_n}\right)\right)=0.$$ 
\noindent This implies that 
\begin{multline*}
\left(u_{\ell_{b_1}}^{b_1}\otimes\cdots\otimes u_{\ell_{b_m}}^{b_m}\right)\cdot\\ 
\cdot\sum\limits_{d\in D_{a_n}^{b_m}}\left(\alpha_d \left(d_{\ell_{b_1}}^{b_1}\otimes\cdots\otimes d_{\ell_{b_m}}^{b_m}\right)\circ d \circ \left(u_{k_{a_1}}^{a_1}\otimes\cdots\otimes u_{k_{a_n}}^{a_n}\right)\right)\cdot\\
\cdot\left(d_{k_{a_1}}^{a_1}\otimes\cdots\otimes d_{k_{a_n}}^{a_n}\right)=0,
\end{multline*}
\noindent and thus

\[
\sum\limits_{d\in D_{a_n}^{b_m}}\alpha_d \left(u_{\ell_{b_1}}^{b_1}\circ d_{\ell_{b_1}}^{b_1}\otimes\cdots\otimes u_{\ell_{b_m}}^{b_m}\circ d_{\ell_{b_m}}^{b_m}\right)\circ
d \circ \left(u_{k_{a_1}}^{a_1}\circ d_{k_{a_1}}^{a_1}\otimes\cdots\otimes u_{k_{a_n}}^{a_n}\circ d_{k_{a_n}}^{a_n}\right)=0.\]
Now notice, since $u_{k_{a}}^{a}\circ d_{k_{a}}^{a}=\alpha_{a,k_{a}}\operatorname{id}_{a}$ with $\alpha_{a,k_{a}}$ not $0$, we finally have $$\alpha\sum\limits_{d\in D_{a_n}^{b_m}}\alpha_d d=0$$ where $\alpha$ is the non-zero scalar $\alpha=\prod\limits_{i=1}^{n}\prod\limits_{j=1}^{m}\alpha_{a_i,k_{a_i}}\alpha_{b_j,\ell_{b_j}}$.  Therefore, $ \overline{\mathcal{H}}$ is faithful.
\end{proof}

Combining the fullness result given in Theorem \ref{thm:FGirrFull} and the faithfulness results in Theorem \ref{thm:FGirrFaithful} yields the following result.

\begin{theorem}
Let $R(V,G)$ be a representation graph which is connected and contains no multiple parallel edges.  Let $\mathcal{I}$ be a tensor ideal of $\mathbf{Dgrams}$ which satisfies (\ref{eqn:CatSchur}).  Then there is an equivalence of categories $$\overline{H}: \bigslant{\mathbf{Dgrams}}{\mathcal{I}}\longrightarrow G-\mathbf{mod}_{\text{irr}}.$$
\end{theorem}

Of course, it remains to determine $\mathcal{I}$, for example, by giving a set of relations.  This will presumably depend on the specifics of the representation theory of $G$ and would need to be determined on a case-by-case basis.  

In this current setting, determining $\mathcal{I}$ seems to require computation over on the representation theory side to determine diagrammatic relations.  Depending on the specific example, this can be a non-trivial task.  In the next section, we give another way of determining $\mathcal{I}$.


\section{Final Remarks}
\label{chpt:FinalRems}


It is worth noting that we can be even more general in our set up with much the same result.  Suppose instead that we begin with a semi-simple, monoidal, $\CC$-linear category $\mathcal{M}$ and restrict to the full subcategory monoidally generated by the simple objects, which we can denote as $\mathcal{M}_{\text{irr}}$, the objects of which can be indexed by $I_{\mathcal{M}}$.  By fixing an object $V$ of $\mathcal{M}$, we may construct a directed graph $\Gamma_{\mathcal{M},V}$ in an analogous way to a representation graph.  Assume $\Gamma:=\Gamma_{\mathcal{M},V}$ is a directed, connected graph which does not contain any multiple parallel edges between vertices, we may form the following definition.  For convenience, we will identify the unit object of $\mathcal{M}$, $\mathbb{1}$, with the vertex of $\Gamma$ corresponding to $\mathbb{1}$.

\begin{definition}Let $\mathbf{Dgrams}_{\Gamma, \star}$ be defined as the monoidal $\CC$-linear category with objects generated by $\star$ and $a\in I_{\mathbb{M}}$ and morphisms generated by 

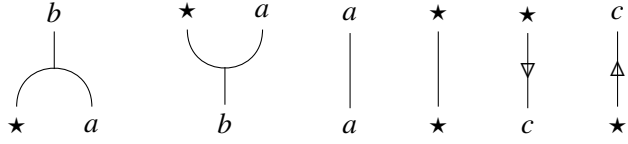
\begin{center}
\begin{tikzpicture}[scale=.5]
	\begin{pgfonlayer}{nodelayer}
		\node   (2) at (-1, -1) {};
		\node   (3) at (1, -1) {};
		\node   (4) at (0, 0) {};
		\node   (5) at (0, 1) {};
		\node   (12) at (-1, -1.5) {$\star$};
		\node   (13) at (1, -1.5) {$a$};
		\node   (14) at (0, 1.5) {$b$};
	\end{pgfonlayer}
	\begin{pgfonlayer}{edgelayer}
		\draw [bend left=90, looseness=1.75] (2.center) to (3.center);
		\draw (4.center) to (5.center);
	\end{pgfonlayer}
\end{tikzpicture} \hspace{6mm} \begin{tikzpicture}[scale=.5]
	\begin{pgfonlayer}{nodelayer}
		\node   (2) at (-1, 1) {};
		\node   (3) at (1, 1) {};
		\node   (4) at (0, 0) {};
		\node   (5) at (0, -1) {};
		\node   (12) at (-1, 1.5) {$\star$};
		\node   (13) at (1, 1.5) {$a$};
		\node   (14) at (0, -1.5) {$b$};
	\end{pgfonlayer}
	\begin{pgfonlayer}{edgelayer}
		\draw [bend right=90, looseness=1.75] (2.center) to (3.center);
		\draw (4.center) to (5.center);
	\end{pgfonlayer}
\end{tikzpicture}\hspace{6mm} \begin{tikzpicture}[scale=.5]
	\begin{pgfonlayer}{nodelayer}
		\node   (5) at (0, -1) {};
		\node   (14) at (0, -1.5) {$a$};
		\node   (15) at (0, 1) {};
		\node   (16) at (0, 1.5) {$a$};
	\end{pgfonlayer}
	\begin{pgfonlayer}{edgelayer}
		\draw (15.center) to (5.center);
	\end{pgfonlayer}
\end{tikzpicture}\hspace{6mm} \begin{tikzpicture}[scale=.5]
	\begin{pgfonlayer}{nodelayer}
		\node   (5) at (0, -1) {};
		\node   (14) at (0, -1.5) {$\star$};
		\node   (15) at (0, 1) {};
		\node   (16) at (0, 1.5) {$\star$};
	\end{pgfonlayer}
	\begin{pgfonlayer}{edgelayer}
		\draw (15.center) to (5.center);
	\end{pgfonlayer}
\end{tikzpicture} \hspace{6mm}\begin{tikzpicture}[scale=.5]
	\begin{pgfonlayer}{nodelayer}
		\node   (5) at (0, -1) {};
		\node   (14) at (0, -1.5) {$c$};
		\node   (15) at (0, 1) {};
		\node   (16) at (0, 1.5) {$\star$};
		\node   (17) at (0, 0) {$\triangledown$};
	\end{pgfonlayer}
	\begin{pgfonlayer}{edgelayer}
		\draw (15.center) to (5.center);
	\end{pgfonlayer}
\end{tikzpicture} \hspace{6mm}\begin{tikzpicture}[scale=.5]
	\begin{pgfonlayer}{nodelayer}
		\node   (5) at (0, -1) {};
		\node   (14) at (0, -1.5) {$\star$};
		\node   (15) at (0, 1) {};
		\node   (16) at (0, 1.5) {$c$};
		\node   (17) at (0, 0) {$\vartriangle$};
	\end{pgfonlayer}
	\begin{pgfonlayer}{edgelayer}
		\draw (15.center) to (5.center);
	\end{pgfonlayer}
\end{tikzpicture}
\end{center}
\noindent for all $a,b\in I_{\mathcal{M}}$ and $c\in I_{\mathcal{M}}$ adjacent to $\mathbb{1}$ in the directed graph $\Gamma$.  The generating diagrams are subjected to the following relations:

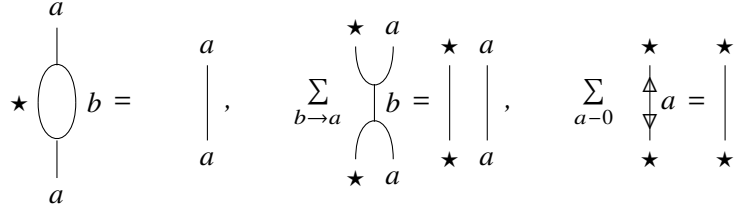
\begin{center}
 \begin{tikzpicture}[scale=0.5, baseline=-2]
	\begin{pgfonlayer}{nodelayer}
		\node   (0) at (-2, 2) {};
		\node   (1) at (-2, 1) {};
		\node   (2) at (-2.5, 0) {};
		\node   (3) at (-1.5, 0) {};
		\node   (4) at (-2, -1) {};
		\node   (5) at (-2, -2) {};
		\node   (6) at (-0.25, 0) {$=$};
		\node   (7) at (2, 1) {};
		\node   (8) at (2, -1) {};
		\node   (9) at (-2, -2.5) {$a$};
		\node   (10) at (2, -1.5) {$a$};
		\node   (11) at (2, 1.5) {$a$};
		\node   (12) at (-2, 2.5) {$a$};
		\node   (13) at (-1, 0) {$b$};
		\node   (14) at (-3, 0) {$\star$};
		\node   (15) at (1, 0) {};
	\end{pgfonlayer}
	\begin{pgfonlayer}{edgelayer}
		\draw (0.center) to (1.center);
		\draw [bend left=90, looseness=3.50] (2.center) to (3.center);
		\draw (4.center) to (5.center);
		\draw (7.center) to (8.center);
		\draw [bend right=90, looseness=3.25] (2.center) to (3.center);
	\end{pgfonlayer}
\end{tikzpicture}, \hspace{5mm} \begin{tikzpicture}[scale=0.5, baseline=-2]
	\begin{pgfonlayer}{nodelayer}
		\node   (0) at (-2, 0.5) {};
		\node   (1) at (-2, -0.5) {};
		\node   (2) at (-2.5, 1.5) {};
		\node   (3) at (-1.5, 1.5) {};
		\node   (4) at (-2.5, -1.5) {};
		\node   (5) at (-1.5, -1.5) {};
		\node   (9) at (0, 1) {};
		\node   (10) at (0, -1) {};
		\node   (11) at (1, 1) {};
		\node   (12) at (1, -1) {};
		\node   (13) at (-3.5, 0) {$\sum\limits_{b\rightarrow a}$};
		\node   (14) at (-1.5, 0) {$b$};
		\node   (15) at (-0.75, 0) {$=$};
		\node   (16) at (0, 1.5) {$\star$};
		\node   (17) at (1, 1.5) {$a$};
		\node   (18) at (0, -1.5) {$\star$};
		\node   (19) at (1, -1.5) {$a$};
		\node   (20) at (-2.5, -2) {$\star$};
		\node   (21) at (-1.5, -2) {$a$};
		\node   (22) at (-2.5, 2) {$\star$};
		\node   (23) at (-1.5, 2) {$a$};
	\end{pgfonlayer}
	\begin{pgfonlayer}{edgelayer}
		\draw [bend right=90, looseness=3.50] (2.center) to (3.center);
		\draw (0.center) to (1.center);
		\draw [bend left=90, looseness=3.50] (4.center) to (5.center);
		\draw (9.center) to (10.center);
		\draw (11.center) to (12.center);
	\end{pgfonlayer}
\end{tikzpicture}, \hspace{5mm} \begin{tikzpicture}[scale=0.5, baseline=-2] 
	\begin{pgfonlayer}{nodelayer}
		\node   (9) at (0, 1) {};
		\node   (10) at (0, -1) {};
		\node   (13) at (-3.5, 0) {$\sum\limits_{a-0}$};
		\node   (15) at (-0.75, 0) {$=$};
		\node   (16) at (0, 1.5) {$\star$};
		\node   (18) at (0, -1.5) {$\star$};
		\node   (20) at (-2, -1) {};
		\node   (21) at (-2, 1) {};
		\node   (22) at (-2, -0.5) {$\triangledown$};
		\node   (23) at (-2, 1.5) {$\star$};
		\node   (24) at (-2, -1.5) {$\star$};
		\node   (25) at (-2, 0.5) {$\vartriangle$};
		\node   (26) at (-1.5, 0) {$a$};
	\end{pgfonlayer}
	\begin{pgfonlayer}{edgelayer}
		\draw (9.center) to (10.center);
		\draw (21.center) to (20.center);
	\end{pgfonlayer}
\end{tikzpicture}
\end{center}
\end{definition}

Denote by $\mathcal{M}^{(a)}$ the simple object of $\mathcal{M}$ corresponding to the index $a\in I_{\mathcal{M}}$.  Let $\pi_{a,b}$ be a map in $\operatorname{Hom}_{\mathcal{M}}\left(V\otimes \mathcal{M}^{(a)},\mathcal{M}^{(b)}\right)$.  As $\Gamma$ has no multiple edges and $\mathcal{M}^{(b)}$ is simple, $\pi_{a,b}$ is unique, up to scaling.  Let $\iota_{a,b}$ be in $\operatorname{Hom}_{\mathcal{M}}\left(\mathcal{M}^{(b)},V\otimes \mathcal{M}^{(a)}\right)$, such that $\pi_{a,b}\circ\iota_{a,b}=\operatorname{id}_{\mathcal{M}^{(b)}}$.  Furthermore, we can define a unique, up to scaling, map, $\pi_{V,c}$ in $\operatorname{Hom}_{\mathcal{M}}\left(V,\mathcal{M}^{(c)}\right)$ when $\mathcal{M}^{(c)}$ is a direct summand of $V$, and let $\iota_{V,c}$ in $\operatorname{Hom}_{\mathcal{M}}\left(\mathcal{M}^{(c)},V\right)$ such that $\pi_{V,c}\circ \iota_{V,c}=\operatorname{id}_{\mathcal{M}^{(c)}}$.

Now, we can define a monoidal, $\CC$-linear functor $$\mathcal{H}:\mathbf{Dgrams}_{\Gamma,\star}\longrightarrow \mathcal{M}_{\text{irr}}$$ which is given on the generating objects and morphisms as follows:

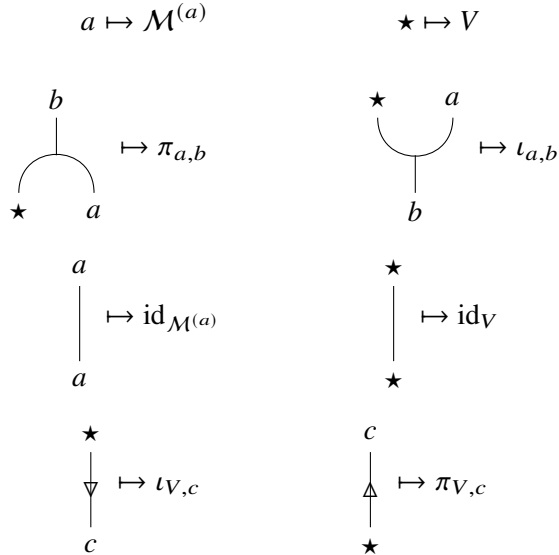
\[
a \mapsto \mathcal{M}^{(a)} \hspace{25mm}
\star \mapsto V
\]

\[
\begin{tikzpicture}[scale=.5, baseline=0]
	\begin{pgfonlayer}{nodelayer}
		\node   (2) at (-1, -1) {};
		\node   (3) at (1, -1) {};
		\node   (4) at (0, 0) {};
		\node   (5) at (0, 1) {};
		\node   (12) at (-1, -1.5) {$\star$};
		\node   (13) at (1, -1.5) {$a$};
		\node   (14) at (0, 1.5) {$b$};
	\end{pgfonlayer}
	\begin{pgfonlayer}{edgelayer}
		\draw [bend left=90, looseness=1.75] (2.center) to (3.center);
		\draw (4.center) to (5.center);
	\end{pgfonlayer}
\end{tikzpicture}  \mapsto \pi_{a,b}\hspace{20mm} \begin{tikzpicture}[scale=.5, baseline=0]
	\begin{pgfonlayer}{nodelayer}
		\node   (2) at (-1, 1) {};
		\node   (3) at (1, 1) {};
		\node   (4) at (0, 0) {};
		\node   (5) at (0, -1) {};
		\node   (12) at (-1, 1.5) {$\star$};
		\node   (13) at (1, 1.5) {$a$};
		\node   (14) at (0, -1.5) {$b$};
	\end{pgfonlayer}
	\begin{pgfonlayer}{edgelayer}
		\draw [bend right=90, looseness=1.75] (2.center) to (3.center);
		\draw (4.center) to (5.center);
	\end{pgfonlayer}
\end{tikzpicture} \mapsto \iota_{a,b}
\]

\[
\begin{tikzpicture}[scale=.5, baseline=0]
	\begin{pgfonlayer}{nodelayer}
		\node   (5) at (0, -1) {};
		\node   (14) at (0, -1.5) {$a$};
		\node   (15) at (0, 1) {};
		\node   (16) at (0, 1.5) {$a$};
	\end{pgfonlayer}
	\begin{pgfonlayer}{edgelayer}
		\draw (15.center) to (5.center);
	\end{pgfonlayer}
\end{tikzpicture}  \mapsto \operatorname{id}_{\mathcal{M}^{(a)}} \hspace{20mm} \begin{tikzpicture}[scale=.5, baseline=0]
	\begin{pgfonlayer}{nodelayer}
		\node   (5) at (0, -1) {};
		\node   (14) at (0, -1.5) {$\star$};
		\node   (15) at (0, 1) {};
		\node   (16) at (0, 1.5) {$\star$};
	\end{pgfonlayer}
	\begin{pgfonlayer}{edgelayer}
		\draw (15.center) to (5.center);
	\end{pgfonlayer}
\end{tikzpicture} \mapsto \operatorname{id}_{V}
\]

\[
\begin{tikzpicture}[scale=.5, baseline=0]
	\begin{pgfonlayer}{nodelayer}
		\node   (5) at (0, -1) {};
		\node   (14) at (0, -1.5) {$c$};
		\node   (15) at (0, 1) {};
		\node   (16) at (0, 1.5) {$\star$};
		\node   (17) at (0, 0) {$\triangledown$};
	\end{pgfonlayer}
	\begin{pgfonlayer}{edgelayer}
		\draw (15.center) to (5.center);
	\end{pgfonlayer}
\end{tikzpicture}  \mapsto \iota_{V,c}\hspace{20mm} \begin{tikzpicture}[scale=.5, baseline=0]
	\begin{pgfonlayer}{nodelayer}
		\node   (5) at (0, -1) {};
		\node   (14) at (0, -1.5) {$\star$};
		\node   (15) at (0, 1) {};
		\node   (16) at (0, 1.5) {$c$};
		\node   (17) at (0, 0) {$\vartriangle$};
	\end{pgfonlayer}
	\begin{pgfonlayer}{edgelayer}
		\draw (15.center) to (5.center);
	\end{pgfonlayer}
\end{tikzpicture} \mapsto \pi_{V,c}
\]

\noindent and extend monoidally and $\CC$-linearly.  The proofs are analogous to show $\mathcal{H}$ is a full functor from $\mathbf{Dgrams}_{\Gamma, \star}$ to $\mathcal{M}_{\text{irr}}$.

Furthermore, we can define an induced faithful functor from $\bigslant{\mathbf{Dgrams}_{\Gamma, \star}}{\mathcal{I}}$ where we let $\mathcal{I}$ be the tensor ideal generated by the following relations: 

\begin{center}
\begin{tikzpicture}[scale=0.5, baseline=-2]
	\begin{pgfonlayer}{nodelayer}
		\node   (0) at (-1, 1) {};
		\node   (1) at (1, 1) {};
		\node   (2) at (-1, -1) {};
		\node   (3) at (1, -1) {};
		\node   (4) at (0, -1) {};
		\node   (5) at (0, -2) {};
		\node   (6) at (0, 1) {};
		\node   (7) at (0, 2) {};
		\node   (9) at (4.5, 1) {};
		\node   (10) at (4.5, -1) {};
		\node   (11) at (1.5, 0) {$=$};
		\node   (12) at (3.25, 0) {$\delta_{a,b}\alpha_{d}$};
		\node   (13) at (0, 0) {$d$};
		\node   (14) at (0, -2.5) {$a$};
		\node   (15) at (0, 2.5) {$b$};
		\node   (16) at (4.5, 1.5) {$a$};
		\node   (17) at (4.5, -1.5) {$a$};
	\end{pgfonlayer}
	\begin{pgfonlayer}{edgelayer}
		\draw (0.center) to (2.center);
		\draw (2.center) to (3.center);
		\draw (3.center) to (1.center);
		\draw (1.center) to (0.center);
		\draw (6.center) to (7.center);
		\draw (4.center) to (5.center);
		\draw (9.center) to (10.center);
	\end{pgfonlayer}
\end{tikzpicture}, where $\delta_{a, b}$ is the Kronecker delta, and $\alpha_{d}\in \CC$,

\begin{tikzpicture}[scale=0.5, baseline=-2]
	\begin{pgfonlayer}{nodelayer}
		\node   (0) at (-1, 1) {};
		\node   (1) at (1, 1) {};
		\node   (2) at (-1, -1) {};
		\node   (3) at (1, -1) {};
		\node   (4) at (0, -1) {};
		\node   (5) at (0, -2) {};
		\node   (6) at (0, 1) {};
		\node   (7) at (0, 2) {};
		\node   (13) at (0, 0) {$d$};
		\node   (14) at (0, -2.5) {$a$};
		\node   (15) at (0, 2.5) {$\star$};
	\end{pgfonlayer}
	\begin{pgfonlayer}{edgelayer}
		\draw (0.center) to (2.center);
		\draw (2.center) to (3.center);
		\draw (3.center) to (1.center);
		\draw (1.center) to (0.center);
		\draw (6.center) to (7.center);
		\draw (4.center) to (5.center);
	\end{pgfonlayer}
\end{tikzpicture} $=\begin{cases} \alpha_{d}\begin{tikzpicture}[scale=.5, baseline=-2]
	\begin{pgfonlayer}{nodelayer}
		\node   (5) at (0, -1) {};
		\node   (14) at (0, -1.5) {$a$};
		\node   (15) at (0, 1) {};
		\node   (16) at (0, 1.5) {$\star$};
		\node   (17) at (0, 0) {$\triangledown$};
	\end{pgfonlayer}
	\begin{pgfonlayer}{edgelayer}
		\draw (15.center) to (5.center);
	\end{pgfonlayer}
\end{tikzpicture} & \text{for } a-\mathbb{1}\\
0 & \text{otherwise}\end{cases}$ \hspace{10mm} and \hspace{10mm}\begin{tikzpicture}[scale=0.5, baseline=-2]
	\begin{pgfonlayer}{nodelayer}
		\node   (0) at (-1, 1) {};
		\node   (1) at (1, 1) {};
		\node   (2) at (-1, -1) {};
		\node   (3) at (1, -1) {};
		\node   (4) at (0, -1) {};
		\node   (5) at (0, -2) {};
		\node   (6) at (0, 1) {};
		\node   (7) at (0, 2) {};
		\node   (13) at (0, 0) {$d$};
		\node   (14) at (0, -2.5) {$\star$};
		\node   (15) at (0, 2.5) {$a$};
	\end{pgfonlayer}
	\begin{pgfonlayer}{edgelayer}
		\draw (0.center) to (2.center);
		\draw (2.center) to (3.center);
		\draw (3.center) to (1.center);
		\draw (1.center) to (0.center);
		\draw (6.center) to (7.center);
		\draw (4.center) to (5.center);
	\end{pgfonlayer}
\end{tikzpicture} $=\begin{cases} \alpha_{d}\begin{tikzpicture}[scale=.5, baseline=-2]
	\begin{pgfonlayer}{nodelayer}
		\node   (5) at (0, -1) {};
		\node   (14) at (0, -1.5) {$\star$};
		\node   (15) at (0, 1) {};
		\node   (16) at (0, 1.5) {$a$};
		\node   (17) at (0, 0) {$\vartriangle$};
	\end{pgfonlayer}
	\begin{pgfonlayer}{edgelayer}
		\draw (15.center) to (5.center);
	\end{pgfonlayer}
\end{tikzpicture} & \text{for } a-\mathbb{1}\\
0 & \text{otherwise}\end{cases}$.
\end{center}

The proofs are analogous to show this construction admits of a fully faithful functor.  

Let us now explore some limitations on these constructions.  First, we need connectedness in our representation graph as the following will illuminate.  Consider a finite group $G$ with the set $\left\{G^{(i)}\right\}_{0\leq i\leq n}$ an exhaustive list of irreducible $G$-modules, up to isomorphism, and let the defining module, $V$, for our representation graph be the trivial module for $G$.  Then we have that $R(V,G)$ is

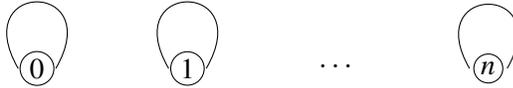
\begin{center}
\begin{tikzpicture}
	\begin{pgfonlayer}{nodelayer}
		\node   (0) at (-3.25, -0.25) {};
		\node   (1) at (-2.75, -0.25) {};
		\node   (2) at (-1.25, -0.25) {};
		\node   (6) at (-0.75, -0.25) {};
		\node   (7) at (1, -0.25) {$\cdots$};
		\node   (8) at (2.75, -0.25) {};
		\node   (9) at (3.25, -0.25) {};
		\node  [style=emptynode] (10) at (-3, -0.25) {$0$};
		\node  [style=emptynode] (11) at (-1, -0.25) {$1$};
		\node  [style=emptynode] (12) at (3, -0.25) {$n$};
	\end{pgfonlayer}
	\begin{pgfonlayer}{edgelayer}
		\draw [bend left=120, looseness=7.00] (0.center) to (1.center);
		\draw [bend left=120, looseness=7.00] (2.center) to (6.center);
		\draw [bend left=120, looseness=6.75] (8.center) to (9.center);
	\end{pgfonlayer}
\end{tikzpicture}
\end{center}

\noindent which would result in a diagrammatic category which does not recover any homomorphisms in $\operatorname{Hom}_G\left(G^{(a)}\otimes G^{(b)}, G^{(c)}\right)$ even when $G^{(c)}$ shows up in the direct sum decomposition of $G^{(a)}\otimes G^{(b)}$.  

Another limitation of our construction is that we assumed the representation graph has no multiple edges between two vertices.  This corresponds to a multiplicity free condition on the direct sum decomposition of the tensor product of the defining module $V$ and each simple module.  Without this condition, there is not a canonical way to decompose this tensor product into simples, and we must make non-trivial choices. 

On the other hand, please note that we make no assumption that these graphs be finite.  In particular, the representation graphs $R(SU(2),V)$, $R(\mathbf{C}_{\infty},V)$, and $R(\mathbf{D}_{\infty},V)$ where $V$ is the natural module for $SU(2)$, have infinitely many nodes.  In this situation, our approach still applies and we can construct a diagrammatic category which encodes the corresponding representation theory.  

Beyond the subgroups of $SU(2)$, there are numerous representation graphs in the literature.  They often are aptly named McKay graphs.  Here are just a few places the reader can explore: \cite{Ost20}\cite{KO02}, \cite{Fren10}, \cite{AEA22}, and \cite{EP14}.  In particular, the following example uses Evans' and Pugh's explicit computation of a representation graph $R(PSL(2;8),\Sigma_7^{(1)})$ to construct a diagrammatic category for a setting unlike any other in this paper.  

\begin{example}
Let $G=PSL(2;8)$ denote the projective special linear group of degree $2$ over the finite field of order $8$.  This is an irreducible primitive group of order $504$.  There are nine irreducible $G$-modules.  We will follow the notation from Evans and Pugh and let the irreducible $G$-modules be indexed in the following way:  let the trivial module of dimension $1$ be denoted $\Sigma_1$; there are four $7$-dimensional irreducible modules denoted $\Sigma_7^{(1)}$, $\Sigma_7^{(1)'}$, $\Sigma_7^{(1)''}$, and $\Sigma_7^{(2)}$; there is one $8$-dimensional irreducible module denoted $\Sigma_8$; finally, there are three $9$-dimensional irreducible modules $\Sigma_9$, $\Sigma_9'$, and $\Sigma_9''$.  For consistency with the notation in this paper, let $\Sigma_7^{(1)}=V$.  Thus, Evans and Pugh compute the representation graph $R(V,G)$ to be the following undirected graph:

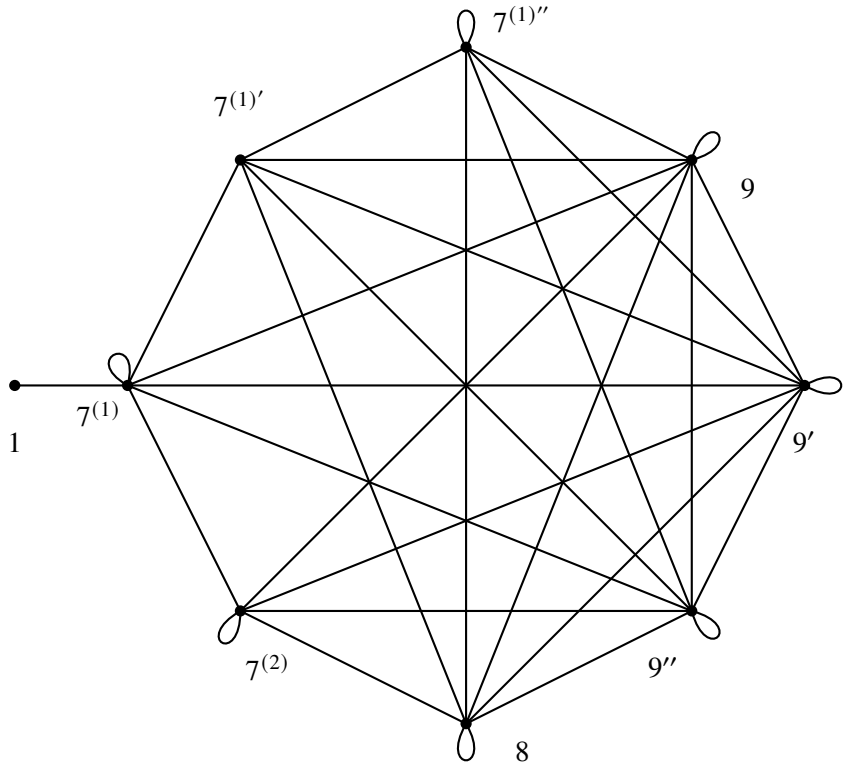
\begin{center}
\begin{tikzpicture}[scale=1.5]
	\begin{pgfonlayer}{nodelayer}
		\node [style=smallnode] (0) at (0, 3) {};
		\node [style=smallnode] (1) at (2, 2) {};
		\node [style=smallnode] (2) at (3, 0) {};
		\node [style=smallnode] (3) at (2, -2) {};
		\node [style=smallnode] (4) at (0, -3) {};
		\node [style=smallnode] (5) at (-2, -2) {};
		\node [style=smallnode] (6) at (-3, 0) {};
		\node [style=smallnode] (7) at (-2, 2) {};
		\node [style=smallnode] (8) at (-4, 0) {};
		\node   (9) at (-3.25, -0.25) {$7^{(1)}$};
		\node   (10) at (-4, -0.5) {$1$};
		\node   (11) at (-2, 2.5) {$7^{(1)'}$};
		\node   (12) at (0.5, 3.25) {$7^{(1)''}$};
		\node   (13) at (2.5, 1.75) {$9$};
		\node   (14) at (3, -0.5) {$9'$};
		\node   (15) at (1.75, -2.5) {$9''$};
		\node   (16) at (0.5, -3.25) {$8$};
		\node   (17) at (-1.75, -2.5) {$7^{(2)}$};
	\end{pgfonlayer}
	\begin{pgfonlayer}{edgelayer}
		 \draw [thick] (8.center) to (6.center);
		 \draw [thick] (6.center) to (7.center);
		 \draw [thick] (7.center) to (0.center);
		 \draw [thick] (0.center) to (1.center);
		 \draw [thick] (1.center) to (2.center);
		 \draw [thick] (2.center) to (3.center);
		 \draw [thick] (3.center) to (4.center);
		 \draw [thick] (4.center) to (5.center);
		 \draw [thick] (5.center) to (6.center);
		 \draw [thick] (7.center) to (4.center);
		 \draw [thick] (7.center) to (3.center);
		 \draw [thick] (7.center) to (2.center);
		 \draw [thick] (7.center) to (1.center);
		 \draw [thick] [in=120, out=60, loop] (0.center) to ();
		 \draw [thick] (0.center) to (4.center);
		 \draw [thick] (0.center) to (3.center);
		 \draw [thick] (0.center) to (2.center);
		 \draw [thick] (6.center) to (2.center);
		 \draw [thick] (6.center) to (3.center);
		 \draw [thick] (6.center) to (1.center);
		 \draw [thick] [in=150, out=75, loop] (6.center) to ();
		 \draw [thick] [in=75, out=15, loop] (1.center) to ();
		 \draw [thick] [in=30, out=-30, loop] (2.center) to ();
		 \draw [thick] (1.center) to (5.center);
		 \draw [thick] (1.center) to (4.center);
		 \draw [thick] (1.center) to (3.center);
		 \draw [thick] (2.center) to (4.center);
		 \draw [thick] (2.center) to (5.center);
		 \draw [thick] (3.center) to (5.center);
		 \draw [thick] [in=-15, out=-75, loop] (3.center) to ();
		 \draw [thick] [in=-120, out=-60, loop] (4.center) to ();
		 \draw [thick] [in=-150, out=-90, loop] (5.center) to ();
	\end{pgfonlayer}
\end{tikzpicture}
\end{center}

\noindent Now we can construct the monoidal $\CC$-linear diagrammatic category $\mathbf{Dgrams}_{R(V,PSL(2;8))}$.  Keeping consistent with the notation from Definition \ref{def:DgramsRVG}, let $$I_{PSL(2;8)}=\{1,7^{(1)},7^{(1)'},7^{(1)''},7^{(2)},8,9,9',9''\}.$$  The set of objects are generated by $a\in I_{PSL(2;8)}$ and the morphisms are generated by 

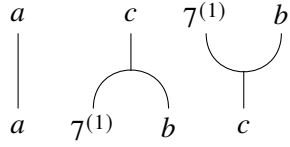
\begin{center}
\begin{tikzpicture}[scale=.5]
	\begin{pgfonlayer}{nodelayer}
		\node  (0) at (-3, -1) {};
		\node  (1) at (-3, 1) {};
		\node  (2) at (-1, -1) {};
		\node  (3) at (1, -1) {};
		\node  (4) at (0, 0) {};
		\node  (5) at (0, 1) {};
		\node  (6) at (2, 1) {};
		\node  (7) at (4, 1) {};
		\node  (8) at (3, 0) {};
		\node  (9) at (3, -1) {};
		\node  (10) at (-3, -1.5) {$a$};
		\node  (11) at (-3, 1.5) {$a$};
		\node  (12) at (-1, -1.5) {$7^{(1)}$};
		\node  (13) at (1, -1.5) {$b$};
		\node  (14) at (0, 1.5) {$c$};
		\node  (15) at (2, 1.5) {$7^{(1)}$};
		\node  (16) at (4, 1.5) {$b$};
		\node  (17) at (3, -1.5) {$c$};
	\end{pgfonlayer}
	\begin{pgfonlayer}{edgelayer}
		\draw (0.center) to (1.center);
		\draw [bend left=90, looseness=1.75] (2.center) to (3.center);
		\draw (4.center) to (5.center);
		\draw [bend right=90, looseness=1.75] (6.center) to (7.center);
		\draw (8.center) to (9.center);
	\end{pgfonlayer}
\end{tikzpicture}
\end{center}

\noindent where $a, b,$ and $c\in I_{PSL(2;8)}$.  Then by Theorem \ref{thm:FGirrFull}, there is an essentially surjective and full functor from $\mathbf{Dgrams}_{R(V,PSL(2;8))}$ onto the full subcategory of $PSL(2;8)$-mod with objects generated by the irreducible $PSL(2;8)$-modules.  By Theorem \ref{thm:FGirrFaithful}, if relations are imposed to ensure (\ref{eqn:CatSchur}), then the resulting category $\overline{\mathbf{Dgrams}}_{R(V,PSL(2;8))}$ is equivalent to $PSL(2;8)$-\textbf{mod}$_{\text{irr}}$.

\end{example}

We shall finish this paper with two more examples.  The following examples show that the constructions in this paper apply to situations outside of representation theory.  The first is the universal Verlinde category, which is a symmetric fusion category defined in \cite{Ost20}.  The second is a slightly more exotic fusion category.  In particular, the Fibonacci category has objects which have non-integer dimensions.  For a more comprehensive understanding of this setting, see \cite{BD12}.

\begin{example}
    Let $\mathbf{k}$ be a field of positive characteristic $p$, $C_p$ be the cyclic group of order $p$, and Rep$_{\mathbf{k}}(C_p)$ be the category of finite dimensional $C_p$-modules over $\mathbf{k}$.  Furthermore, we say that a morphism $f$ is negligible if $\operatorname{tr}(f\circ g)=0$ for all other morphisms $g$, where tr represents the trace. Ostrik defines the universal Verlinde category, Ver$_p$, to be the quotient of Rep$_{\mathbf{k}}(C_p)$ by the negligible morphisms.  Ostrik further gives a characterization of the simple objects and the fusion rule for Ver$_p$, which results in the following fusion graph:

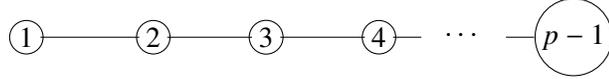
\begin{center}
    \begin{tikzpicture}[scale=0.75, {circ/.style}={shape=circle, inner sep=1pt, draw, node contents=}]
	\begin{pgfonlayer}{nodelayer}
		\node   (0) at (-4, 0) {};
		\node [style=emptynode] (1) at (-2, 0) {$2$};
		\node [style=emptynode] (2) at (0, 0) {$3$};
		\node [style=emptynode] (3) at (2, 0) {$4$};
		\node   (4) at (2.75, 0) {};
		\node [style=emptynode] (7) at (-4.25, 0) {$1$};
		\node   (8) at (-2.25, 0) {};
		\node   (9) at (-0.25, 0) {};
		\node   (10) at (0.25, 0) {};
		\node   (11) at (-1.75, 0) {};
		\node   (12) at (1.75, 0) {};
		\node   (13) at (2.25, 0) {};
		\node [style=emptynode] (14) at (5.45, 0) {$p-1$};
		\node   (15) at (4.75, 0) {};
		\node   (16) at (4.25, 0) {};
		\node   (17) at (3.5, 0) {$\cdots$};
	\end{pgfonlayer}
	\begin{pgfonlayer}{edgelayer}
		\draw (0.center) to (8.center);
		\draw (11.center) to (9.center);
		\draw (10.center) to (12.center);
		\draw (13.center) to (4.center);
		\draw (16.center) to (15.center);
	\end{pgfonlayer}
\end{tikzpicture}
\end{center}

    \noindent where the node labeled by $2$ corresponds to defining object.  Thus, our construction provides a diagrammatic realization of Ver$_p$.
\end{example}

\begin{example}
Let $\mathcal{F}ib$ be the semi-simple, monoidal, $\CC$-linear category in which $\operatorname{Obj}\left(\mathcal{F}ib\right)$ are generated by two objects, $I$ and $X$, which follow the following tensor product decomposition rules:

\[
I\otimes I \cong I,\hspace{5mm} I \otimes X\cong X \otimes I\cong X,\hspace{5mm} \text{and} \hspace{5mm} X \otimes X \cong I \oplus X.
\]

Now, define $\mathcal{F}ib_{\text{irr}}$ as the full subcategory monoidally generated by $X$ and $I$.  Thus, we can construct the following graph $\Gamma:=\Gamma_{\mathcal{F}ib_{\text{irr}}}$ with generating object $X$:

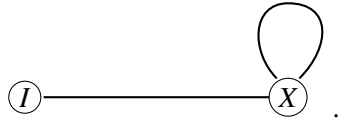
\begin{center}
\begin{tikzpicture}
	\begin{pgfonlayer}{nodelayer}
		\node  (0) at (-3.25, 0) {};
		\node   (1) at (-0.25, 0) {};
		\node   (2) at (0, 1.25) {};
		\node   (3) at (0.15, 0.25) {};
		\node [style=emptynode] (4) at (0, 0) {$X$};
		\node   (5) at (-0.15, 0.25) {};
		\node [style=emptynode] (6) at (-3.5, 0) {$I$};
	\end{pgfonlayer}
	\begin{pgfonlayer}{edgelayer}
		\draw [thick] (0.center) to (1.center);
		\draw [in=135, out=180, looseness=1.25, thick] (2.center) to (5.center);
		\draw [in=45, out=0, looseness=1.50, thick] (2.center) to (3.center);
	\end{pgfonlayer}
\end{tikzpicture}.
\end{center}

From the graph $\Gamma$, we can construct $\mathbf{Dgrams}_{\Gamma}$ and functor $$\mathcal{H}: \mathbf{Dgrams}_{\Gamma}\longrightarrow \mathcal{F}ib_{\text{irr}}$$ consistent with the other constructions in this paper.

\end{example}

\phantomsection

\addcontentsline{toc}{section}{Bibliography}

\bibliography{DiagramCatsFromDirectedGraphs}
\bibliographystyle{unsrt}

}

\end{document}